\documentclass[nonblindrev,mnsc]{informs3} % current default for manuscript submission
\usepackage{xfrac}  
\usepackage{pdflscape}
\usepackage[ruled,linesnumbered]{algorithm2e}

\usepackage{natbib}
 \NatBibNumeric
 \def\bibfont{\small}%
 \bibpunct[, ]{[}{]}{,}{n}{}{,}%

\usepackage[caption=false]{subfig}
\usepackage[english]{babel}
\usepackage{float}
\TheoremsNumberedThrough     % Preferred (Theorem 1, Lemma 1, Theorem 2)
\ECRepeatTheorems

\EquationsNumberedThrough    % Default: (1), (2), ...
\MANUSCRIPTNO{}

\usepackage{afterpage}

\usepackage[normalem]{ulem}
\newcommand\redsout{\bgroup\markoverwith{\textcolor{red}{\rule[0.5ex]{2pt}{0.4pt}}}\ULon}
\newcommand\gsout{\bgroup\markoverwith{\textcolor{green}{\rule[0.5ex]{2pt}{0.4pt}}}\ULon}

\usepackage{float}

\usepackage{array}
\newcommand{\PreserveBackslash}[1]{\let\temp=\\#1\let\\=\temp}
\newcolumntype{C}[1]{>{\PreserveBackslash\centering}p{#1}}
\newcolumntype{R}[1]{>{\PreserveBackslash\raggedleft}p{#1}}
\newcolumntype{L}[1]{>{\PreserveBackslash\raggedright}p{#1}}

\RequirePackage[shortlabels]{enumitem}
\usepackage{mathtools} % for coloneqq
\usepackage{enumitem} % for enumeration
\usepackage{dsfont} % for indicator function
\usepackage{mathrsfs}  
\usepackage{verbatim}
\usepackage{romannum}
\AtBeginDocument{\pagenumbering{arabic}}
\newcommand{\ba}{{\bf a}}

\newcommand{\bd}{{\bf d}}

\newcommand{\bff}{{\bf f}}

\newcommand{\bq}{{\bf q}}

\newcommand{\bw}{{\bf w}}
\newcommand{\bx}{{\bf x}}
\newcommand{\by}{{\bf y}}
\newcommand{\bz}{{\bf z}}
\newcommand{\bba}{{\bf A}}

\newcommand{\bbp}{{\bf P}}

\newcommand{\bmu}{{\boldsymbol{\mu}}}

\newcommand{\balpha}{{\boldsymbol{\alpha}}}

\newcommand{\bgamma}{{\boldsymbol{\gamma}}}

\newcommand{\bomega}{{\boldsymbol{\omega}}}

\newcommand{\pran}[1]{\left(#1\right)}

\newenvironment{itemize*}%
{\begin{itemize}%
\setlength{\itemsep}{0.5em}%
\setlength{\parskip}{0pt}}%
{\end{itemize}}

\newcommand{\bzeta}{{\boldsymbol{\zeta}}}

\let\R\Real

\newcommand{\bzero}{{\boldsymbol{0}}}

\usepackage{graphicx}
\usepackage{bbm}
\usepackage{tikz,ifthen}
\usepackage{fixltx2e}    
\usetikzlibrary{math}
\usetikzlibrary{shapes.misc}
\usetikzlibrary{shapes.multipart,positioning,patterns,backgrounds}
\newcommand{\ubar}[1]{\text{\b{$#1$}}}

\usepackage{multirow}
\usepackage{array}

\makeatletter
\let\save@mathaccent\mathaccent
\newcommand*\if@single[3]{%
\setbox0\hbox{${\mathaccent"0362{#1}}^H$}%
\setbox2\hbox{${\mathaccent"0362{\kern0pt#1}}^H$}%
\ifdim\ht0=\ht2 #3\else #2\fi
}
\newcommand*\rel@kern[1]{\kern#1\dimexpr\macc@kerna}
\newcommand*\widebar[1]{\@ifnextchar^{{\wide@bar{#1}{0}}}{\wide@bar{#1}{1}}}
\newcommand*\wide@bar[2]{\if@single{#1}{\wide@bar@{#1}{#2}{1}}{\wide@bar@{#1}{#2}{2}}}
\newcommand*\wide@bar@[3]{%
\begingroup
\def\mathaccent##1##2{%
\let\mathaccent\save@mathaccent
\if#32 \let\macc@nucleus\first@char \fi
\setbox\z@\hbox{$\macc@style{\macc@nucleus}_{}$}%
\setbox\tw@\hbox{$\macc@style{\macc@nucleus}{}_{}$}%
\dimen@\wd\tw@
\advance\dimen@-\wd\z@
\divide\dimen@ 3
\@tempdima\wd\tw@
\advance\@tempdima-\scriptspace
\divide\@tempdima 10
\advance\dimen@-\@tempdima
\ifdim\dimen@>\z@ \dimen@0pt\fi
\rel@kern{0.6}\kern-\dimen@
\if#31
  \overline{\rel@kern{-0.6}\kern\dimen@\macc@nucleus\rel@kern{0.4}\kern\dimen@}%
  \advance\dimen@0.4\dimexpr\macc@kerna
  \let\final@kern#2%
  \ifdim\dimen@<\z@ \let\final@kern1\fi
  \if\final@kern1 \kern-\dimen@\fi
\else
  \overline{\rel@kern{-0.6}\kern\dimen@#1}%
\fi
}%
\macc@depth\@ne
\let\math@bgroup\@empty \let\math@egroup\macc@set@skewchar
\mathsurround\z@ \frozen@everymath{\mathgroup\macc@group\relax}%
\macc@set@skewchar\relax
\let\mathaccentV\macc@nested@a
\if#31
\macc@nested@a\relax111{#1}%
\else
\def\gobble@till@marker##1\endmarker{}%
\futurelet\first@char\gobble@till@marker#1\endmarker
\ifcat\noexpand\first@char A\else
  \def\first@char{}%
\fi
\macc@nested@a\relax111{\first@char}%
\fi
\endgroup
}
\makeatother

\begin{document}
%%%%%%%%%%%%%%%%

\RUNAUTHOR{Lu and Sturt}

\RUNTITLE{On the Sparsity of Optimal Linear Decision Rules in Robust Optimization}

\TITLE{On the Sparsity of Optimal Linear Decision Rules for a Class of Robust Optimization Problems with Box Uncertainty Sets}

% Block of authors and their affiliations starts here:
% NOTE: Authors with same affiliation, if the order of authors allows,
%   should be entered in ONE field, separated by a comma.
%   \EMAIL field can be repeated if more than one author
\ARTICLEAUTHORS{%
\AUTHOR{Haihao Lu}
\AFF{Sloan School of Management \\
MIT, \EMAIL{haihao@mit.edu}}%
%% Enter all authors la
\AUTHOR{Bradley Sturt}
\AFF{Department of Information and Decision Sciences\\
University of Illinois Chicago, \EMAIL{bsturt@uic.edu}}
%% Enter all authors
} % end of the block
%%\SingleSpacedXI
%%\OneAndAHalfSpacedXI

\ABSTRACT{%
We consider a class of production-inventory problems with box uncertainty sets from the seminal work of Ben-Tal et al. (2004) on linear decision rules in robust optimization. We prove that there always exists an optimal linear decision rule for this class of problems in which the number of nonzero parameters in the linear decision rule grows linearly in the number of time periods. This is the first result to prove that optimal linear decision rules are sparse in a widely-studied class of robust optimization problems with many time periods.  Harnessing this sparsity guarantee,  we introduce a  reformulation technique that allows robust optimization problems such as production-inventory problems to be solved as a compact linear optimization problem when most of the parameters of the linear decision rules are forced to be equal to zero. We also develop an active set method for identifying the parameters of linear decision rules that are equal to zero at optimality.  In numerical experiments on production-inventory problems with hundreds of time periods, we find that our  reformulation technique coupled with the active set method yield more than a 32x speedup over state-of-the-art linear optimization solvers in computing linear decision rules that are within 1\% of optimal. Our proofs and algorithms are based on a principled analysis of extreme points of linear optimization formulations.
}

% Sample

\HISTORY{First version: Mar.\, 24, 2022. Accepted for publication on Mar.\,17, 2025. }
%\HISTORY{First version: Nov.\,1, 2023. Revisions submitted on June 24, 2024 and Sep 3, 2024.  This version: July 11, 2023. }

\maketitle
%%%%%%%%%%%%%%%%%%%%%%%%%%%%%%%%%%%%%%%%%%%%%%%%%%%%%%%%%%%%%%%%%%%%%%
%\OneAndAHalfSpacedXI
%\DoubleSpacedXI
%\vspace{-1em}

\section{Introduction} \label{sec:intro}
Over the past two decades, {robust optimization} has emerged as a leading approach in operations research and management science for sequential decision-making  under uncertainty. 
One of the major reasons for its popularity is {computational}: dynamic robust optimization problems 
are often amenable to efficient approximations in complex, real-world operational planning problems.  The successful approximation techniques  for dynamic robust optimization typically rely on restricting the control policies to a simple functional form, with the most popular restriction being to {linear decision rules}~\cite{ben2004adjustable}.

 The success of linear decision rules in addressing real-world problems is  very impressive. On the empirical side, 
linear decision rules have been found to exhibit strong performance in a myriad of high-stakes robust optimization applications such as disaster response,  personalized healthcare, sustainable energy management, transportation routing, and many others \citep{simchi2019designing,lia2022distributional,eikelder2019adjustable,sani2020affine,gauvin2017decision}.  
On the theoretical side, a burgeoning literature has established that linear decision rules are provably optimal control policies in many classes of robust optimization problems  \citep{bertsimas2012power,bertsimas2010optimality,iancu2013supermodularity,gounaris2013robust,ardestani2016robust,simchi2019designing,el2021optimality,georghiou2021optimality,zhen2018adjustable}. 
 The aforementioned papers all build upon the seminal work of Ben-Tal et al. \cite{ben2004adjustable}, which showed that optimal linear decision rules for  robust optimization can be computed in polynomial time even in ``{fairly complicated models with high-dimensional state spaces and many stages}"  \citep[p. 374]{ben2004adjustable}.\looseness=-1

  However, the impressive performance of linear decision rules comes at a price. 
Compared to simpler classes of control policies such as {static decision rules}, where the decision does not depend on the past uncertainty,  linear decision rules are represented using a significantly greater number of parameters. Specifically,  the number of parameters in linear decision rules grows {{quadratically}} in the number of stages of the robust optimization problem, whereas the number of parameters for static decision rules grows {{linearly}} in the number of stages. 
Because the number of parameters for representing linear decision rules can be enormous, computing  optimal linear decision rules can require  ``the solution of monolithic and often dense optimization problems" \citep[p.814]{georghiou2019robust} which can exceed computer memory in robust optimization problems with as few as seventy-five stages \citep[p.827]{georghiou2019robust}.  
   This is problematic  because real-world applications routinely have many hundreds of stages, particularly  when a discrete-time robust optimization problem is approximating a continuous-time operational planning problem.

To get around this, a recent stream of  research has advocated  for imposing {sparsity constraints} onto linear decision rules in robust optimization problems. The driving insight behind this stream of  research
is that if many of the parameters of linear decision rules are forced to be equal to zero, then the problem of optimizing the remaining parameters can often be solved more efficiently. Numerical studies have shown that applying this insight can yield  tremendous improvements in computation times with only small losses in performance in applications such as  newsvendor networks in pharmaceutical  supply chains \cite[\S 4]{bandi2019sustainable},  unit commitment problems in power systems \cite[\S 4.3]{lorca2016multistage}, and location-transportation problems \cite{ardestani2018value}. Column-generation techniques have been proposed for identifying the subset of parameters of linear decision rules to set to zero \cite{delage2021column}. The harnessing of sparsity thus stands as one of the most promising directions  for efficiently computing optimal linear decision rules for the sizes of robust optimization problems that arise in industry.  

Ultimately, the  success of using sparsity to develop faster  algorithms will  hinge on whether robust optimization problems have optimal linear decision rules that are sparse. Indeed, if a robust optimization problem does not have  optimal linear decision rules that are sparse, then deploying a sparsity-imposing algorithm can risk leading to control policies with unexpectedly and undesirably suboptimal performance. Moreover, if the number of zero parameters in optimal linear decision rules is not a significant proportion of the total parameters, then the benefits of developing sparsity-imposing algorithms can be limited, and research efforts may be better spent on developing alternative algorithmic techniques such as those based on primal-dual saddle point  formulations or online convex optimization \cite{postek2021first,ho2018online,ben2015oracle}. To the best of our knowledge, the fundamental question of whether a significant number of zero parameters can be expected in optimal linear decision rules in any class of robust optimization problems with many time periods has remained open.   

In this paper, we resolve this open question 
by considering a class of production-inventory problems from the seminal work of Ben-Tal et al.~\cite{ben2004adjustable} on linear decision rules in robust optimization. This class of production-inventory problems served as the key illustration  in \cite{ben2004adjustable} that optimal linear decision rules can be computed in polynomial time and provide excellent performance in realistic and complex robust inventory management problems. It has since become one of the most popular classes of problems in the robust optimization literature, serving as a test bed for the complex real-world applications in which linear decision rules are routinely used; see \cite{de2016impact} and references therein. The class of production-inventory problems involves a firm which dynamically determines production quantities at multiple factories over a selling season,  a single product with uncertain demand that lies in box uncertainty sets, and complex business constraints that link the inventory levels and production decisions across multiple periods of time. All of the decisions in the class of production-inventory problems are continuous.

Our main theoretical result  of this paper (Theorem~\ref{thm:main}) establishes that the minimum number of nonzero parameters for optimal linear decision rules for the class of production-inventory problems grows \emph{subquadratically} in the number of time periods.  To state our  sparsity guarantee more formally, we recall for any instance of such production-inventory problems with $E$ factories and $T$ time periods that the number of parameters for representing linear decision rules is  equal to $\frac{1}{2} ET(T+1) = \mathcal{O}(ET^2)$. For this class of production-inventory problems, we prove in Theorem~\ref{thm:main} that there always exists an optimal linear decision rule in which the number of nonzero parameters is at most equal to $2 + 8E + 10T + 6E(T-\delta) = \mathcal{O}(ET)$ whenever an optimal linear decision rule for the instance exists, where $\delta \ge 0$ is the minimum lead time across the factories. In other words, although the number of parameters for representing linear decision rules grows \emph{quadratically} in the number of time periods, Theorem~\ref{thm:main} shows that the minimum number of {nonzero} parameters for representing {optimal} linear decision rules grows only \emph{linearly} in the number of time periods.  Our proof is  based on a principled analysis of extreme points of linear optimization formulations, and we demonstrate
via numerical experiments that our bounds are indicative of practice.\looseness=-1

Our sparsity guarantee (Theorem~\ref{thm:main}) contributes to the robust optimization literature by proving that sparsity can be a fundamental property of optimal linear decision rules in  robust optimization problems when the number of time periods is large. 
While our proof techniques for establishing sparsity are   tailored to the aforementioned  class of production-inventory problems with box uncertainty sets, our proof techniques  can also be extended to establish sparsity guarantees for a class of dynamic newsvendor problems with box uncertainty sets from \cite{ben2005retailer,bertsimas2010optimality,iancu2013supermodularity} and to production-inventory problems with non-box uncertainty sets (see \S\ref{sec:extensions}). 
 We hope that our developments for these  particular  problem classes  inspire and aid the robust optimization community to identify more examples of robust optimization problems where the existence of sparse optimal linear decision rules  is guaranteed.

Our sparsity guarantees also have  implications from the perspective of designing practically efficient algorithms for computing linear decision rules in dynamic robust optimization problems with box uncertainty sets and many time periods. 
Specifically, harnessing our sparsity guarantees,  we  introduce a novel  reformulation technique  (see \S\ref{sec:algorithm:fixed_active_set}) that allows robust optimization problems such as the production-inventory problem to be solved very efficiently when most of the parameters of the linear decision rules are forced to be equal to zero. The key insight behind our  reformulation technique is that imposing sparsity constraints onto linear decision rules can have a side effect of inducing redundancy in the constraints of the robust counterpart. By exploiting this redundancy in the constraints of the robust counterpart, our reformulation technique drastically decreases the number of auxiliary decision variables in the robust counterpart so that they match, up to  a constant factor,  the number of nonzero parameters of the linear decision rules (Proposition~\ref{prop:productioninventory_reformulation} in \S\ref{sec:algorithm:fixed_active_set}). It follows from Theorem~\ref{thm:main} and Proposition~\ref{prop:productioninventory_reformulation} that if the nonzero parameters of an optimal linear decision rule for a production-inventory problem were {known}, then an optimal linear decision rule for the production-inventory problem can be computed by solving a linear optimization problem with only $\mathcal{O}(TE)$ decision variables and constraints.

While Theorem~\ref{thm:main} does not specify which parameters of an optimal linear decision rules for the production-inventory problems will be nonzero, we show  that the set of nonzero parameters can be found {algorithmically} to high or perfect accuracy. 
Specifically, we develop a simple algorithm based on the active set method (see \S\ref{sec:algorithm:activeset}) for identifying the parameters of sparse optimal linear decision rules that are nonzero. 
%Specifically, to identify the parameters of sparse optimal linear decision rules that are nonzero, we have implemented and open sourced a simple algorithm based on the active set method. 
The algorithm consists of iterating
over different ‘active sets’ of nonzero parameters, finding a linear decision rule that is
optimal when the parameters of the linear decision rule that are not in the active set
are constrained to be equal to zero, and then using shadow prices to update the active
set. The algorithm is guaranteed to converge to an optimal linear decision rule after a finite number of
iterations and overcomes the issue that we do not know a priori which parameters of optimal linear decision rules will be nonzero. From a memory-efficiency standpoint, the algorithm can be terminated early with a suboptimal active set to avoid solving a large-scale linear optimization problem that exceeds computer memory.

We show through numerical experiments  that our  reformulation technique (\S\ref{sec:algorithm:fixed_active_set}) and active set method (\S\ref{sec:algorithm:activeset}) can find high-quality linear decision rules in applications that are {significantly} larger than could be tackled by  extant methods. For example,  in production-inventory problems with $T = 240$ time periods and $E = 5$ factories, we show in \S\ref{sec:experiment} that our approach 
offers a 32x speedup over state-of-the-art linear optimization solvers in computing linear decision rules that are within 1\% of optimal. Moreover, we show that our approach scales to applications with hundreds of time periods and dozens of state and decision variables in each time period, a regime which to our knowledge cannot be addressed by any extant algorithms from the robust optimization literature.\footnote{Our open source implementation of the  reformulation technique and active set method for computing linear decision rules in robust optimization can be accessed at \url{https://github.com/brad-sturt/LDRSolver.git}.}\looseness=-1
 
In summary, our main contributions in this paper are as follows:

\begin{enumerate}
\item We prove that the minimum number of {nonzero} parameters for representing {optimal} linear decision rules can grow subquadratically in the number of time periods  in widely-studied classes of robust optimization problems. 
\item We  introduce a   reformulation technique  that allows applications such as production-inventory problems to be solved very efficiently when sparsity constraints are imposed on linear decision rules.\looseness=-1
\item We show that our  reformulation technique can be combined with an active set method to significantly speed up and scale up the size of robust optimization problems that can be solved in practical computation times. 
\end{enumerate}

The rest of this paper is organized as follows. In \S\ref{sec:prelim}, we present a background on linear decision rules in robust optimization and production-inventory problems. In \S\ref{sec:main_result}, we state our sparsity guarantee for production-inventory problems and present an overview of our proof.  In \S\ref{sec:algorithm}, we present a reformulation technique and an active set method  for computing linear decision rules in robust optimization problems with huge numbers of time periods.  In \S\ref{sec:experiment}, we illustrate the practical implications of our main results via numerical experiments. In \S\ref{sec:extensions}, we discuss extensions of our main sparsity guarantee to other classes of robust optimization problems. In \S\ref{sec:conclusion}, we conclude and discuss future directions of research. All technical proofs can be found in the supplemental appendices.

\textbf{Notations.} We let $\R$ denote the real numbers, we use boldface letters like ${\bx}$ to denote non-scalar quantities like vectors and matrices, and we let $\| {\bx}\|_0$ denote the number of nonzeros in ${\bx}$. Given an optimization problem  $\min_\bx f(\bx)$, we say that a solution $\bar{\bx}$ is feasible for the optimization problem if and only if the cost satisfies $f(\bar{\bx}) < \infty$. It follows from this notation that  $\min_\bx f(\bx) < \infty$ if and only if an optimization problem $\min_\bx f(\bx)$ has a nonempty feasible region.

\section{Background and Problem Setting} \label{sec:prelim}
\subsection{Robust Optimization and Linear Decision Rules}
We consider robust optimization problems faced by firms in which decisions $\bx_1,\ldots,\bx_T \in \R^n$ are made sequentially over a planning horizon of $T$ discrete stages. In the beginning of each stage $t \in [T] \equiv \{1,\ldots,T\}$, the firm observes an uncertain variable $\zeta_t \in \R$ that is chosen adversarially from  an uncertainty set denoted by $\mathcal{U}_t $. 
 The goal of the firm is to choose decisions sequentially, that is, adapting to the uncertain variables observed in the past stages, to minimize the firm's cost under an adversarial choice of the uncertain variables. Such problems are denoted generically by % jointly violate the problem constraints. 
\begin{align} \tag{RO} \label{prob:main}
\max_{\zeta_1 \in \mathcal{U}_1} \min_{\bx_1 \in \R^n}\cdots \max_{\zeta_T \in \mathcal{U}_T} \min_{\bx_T \in \R^n}  C(\bx_1,\ldots,\bx_T,\zeta_1,\ldots,\zeta_T),
\end{align}
where we adopt the convention that the cost 
$C(\bx_1,\ldots,\bx_T,\zeta_1,\ldots,\zeta_T)  \in \R \cup \{ \infty \}$ 
is equal to infinity if the chosen decisions are infeasible for the given realization of the uncertain variables. We  assume without loss of generality throughout the paper that $\zeta_1 = 1$. 
  For comprehensive introductions to the many applications of dynamic robust optimization problems of the form~\eqref{prob:main},  we refer the reader to excellent survey papers such as  \cite{delage2015robust,georghiou2019decision,yanikouglu2019survey,bertsimas2011theory} as well as recent papers such as \cite{postek2023multi,daryalal2023two}.

Our work focuses on one of the most popular approximation methods for solving \eqref{prob:main}, denoted by the optimization problem~\eqref{prob:ldr}. 
Proposed in the seminal work of Ben-Tal et al.~\cite{ben2004adjustable}, \eqref{prob:ldr} aims to obtain a computationally tractable approximation of \eqref{prob:main} by restricting the decisions in each stage to be a linear function of the uncertain variables observed in the past. Formally, the set of optimal linear decision rules for \eqref{prob:main} is defined as the set of optimal solutions for   % of the form % of the uncertain quantities:
\begin{align} \tag{LDR} \label{prob:ldr}
\underset{\by_{t1},\ldots,\by_{tt} \in \R^{n}: \; \forall t \in [T] }{\textnormal{min}}  \;\;\max_{\zeta_1 \in \mathcal{U}_1,\ldots,\zeta_T \in \mathcal{U}_T} C \left(\sum_{s=1}^1 \by_{1s} \zeta_s,\ldots,\sum_{s=1}^T \by_{Ts} \zeta_s, \;\zeta_1,\ldots,\zeta_T \right). 
\end{align}
The goal of \eqref{prob:ldr} is to obtain the best parameters for the linear decision rules, denoted by  $\by_{t1},\ldots,\by_{tt} \in \R^n$ for all $t \in [T]$. Given the parameters of the linear decision rules, the decision to make on each stage $t \in [T]$ is computed as  $\bx_t=\sum_{s=1}^t \by_{ts} \zeta_s$. 
It follows from the above notation that the number of parameters used for representing linear decision rules is $\sum_{t=1}^T n t = \frac{1}{2}nT(T+1) = \mathcal{O}(nT^2)$.\looseness=-1

\subsection{The Production-Inventory Problem} \label{sec:prodinv}
Our work focuses on a class of production-inventory problems from \cite{ben2004adjustable}.  
The class of problems considers a firm with a central warehouse and $E$ factories that aim to satisfy uncertain demand for a single product over a selling season.  The selling season of the firm's product is discretized into $T$ time periods, which are spaced equally over the selling season. 
 In each time period $t \in [T]$,  the firm sequentially performs the following three steps:
 \begin{enumerate}
     \item The firm replenishes the inventory level at the central warehouse by producing additional products at their  factories.  Let  $x_{te} \ge 0$ denote the number of product units that the firm decides to produce at each of the factories $e \in [E] \equiv \{1,\ldots,E\}$ at a per-unit cost of $c_{te}$. The production quantity $x_{te}$ at factory $e$ on period $t \in [T]$ becomes available at the central warehouse on period $t + \delta_e$, with $\delta_e \ge 0$ denoting the lead time for factory $e \in [E]$. 
     \item The firm observes the customer demand  at the central warehouse. The demand at the central warehouse is denoted by $\zeta_{t+1} \in \mathcal{U}_{t+1} \equiv [\ubar{D}_{t+1},\bar{D}_{t+1}]$, which must be satisfied immediately without backlogging from the inventory in  the central warehouse. The lower and upper bounds in the uncertainty set, denoted by  $\ubar{D}_{t+1} < \bar{D}_{t+1}$, capture the minimum and maximum level of customer demand that the firm anticipates  receiving in each time period $t$.
          \item The firm verifies that the remaining  inventory  in the warehouse lies within a  pre-specified interval given by  $[V_{\text{min}}, V_{\text{max}}]$. %The purpose of the interval is to capture any physical or contractual constraints regarding the inventory level in the warehouse. 
          Specifically, the remaining inventory level in the central warehouse at the end of each time period $t \in [T]$ must satisfy
 \begin{align*}
    %V_{\text{min}} \le  v_1 + \sum_{\ell =1}^t  \sum_{e=1}^E x_{\ell e} - \sum_{s=2}^{t+1} \zeta_s \le V_{\text{max}},
   V_{\textnormal{min}} \le v_1 + \sum_{e=1}^E \sum_{\ell=1}^{t-\delta_e}   x_{\ell e} - \sum_{s=2}^{t+1} \zeta_s \le  V_{\textnormal{max}},
 \end{align*}
 where $v_1$ is the initial inventory level in the central warehouse at the beginning of the selling horizon,  $\sum_{e=1}^E \sum_{\ell=1}^{t-\delta_e}   x_{\ell e}$ is the cumulative number of product units that have become available at the central warehouse up through time period $t$, and $\sum_{s=2}^{t+1} \zeta_s$ is the cumulative customer demand that has been observed at the central warehouse up through time period $t$.
 \end{enumerate}
   In addition to satisfying the constraints on the inventory level in the central warehouse at the end of each time period, the firm's production decisions must satisfy $x_{te} \le p_{te}$ in each time period $t \in [T]$ and each factory $e \in [E]$, where $p_{te}$ is the maximum production level for factory $e$ in time period $t$. Furthermore, the firm's total production quantity across the selling season for each factory $e \in [E]$ must satisfy $\sum_{t=1}^T x_{te}  \le Q_e$, where $Q_e$ is the maximum total production level for factory $e$. 
The goal of the firm is to satisfy the customer demand at minimal cost while satisfying production and warehouse constraints.

We observe that the above class of production-inventory problems from Ben-Tal et al.~\cite{ben2004adjustable} is a special case of \eqref{prob:main} in which the cost function has the form  
\begin{subequations}
\begin{align} 
%\begin{aligned}
C(\bx_1,\ldots,\bx_T, \zeta_1,\ldots,\zeta_{T+1}) =  &  \sum_{e=1}^E \sum_{t=1}^{T}  c_{te} x_{te}\label{prob:example1:a} \\
\textnormal{subject to}\quad &\sum_{t=1}^{T} x_{te} \le  Q_e &&  \forall e  \in [E]  \label{prob:example1:b}\\
&0 \le x_{te} \le  p_{te} && \forall  e \in [E],  t \in [T] \label{prob:example1:c}\\
&V_{\textnormal{min}} \le v_1 + \sum_{e=1}^E \sum_{\ell=1}^{t-\delta_e}   x_{\ell e} - \sum_{s=2}^{t+1} \zeta_s \le  V_{\textnormal{max}} && \forall  t \in [T] \label{prob:example1:d}
\end{align}  
\end{subequations}%
and where the uncertainty sets have the form $\mathcal{U}_1 \equiv [\ubar{D}_1,\bar{D}_1], \ldots, \mathcal{U}_{T+1} \equiv [\ubar{D}_{T+1},\bar{D}_{T+1}]$ 
with $\ubar{D}_1 = \bar{D}_1 = 1$ and $\ubar{D}_{t+1} < \bar{D}_{t+1}$ for each time period $t \in [T]$. 
We use the convention that the cost function evaluates to \eqref{prob:example1:a} if the constraints \eqref{prob:example1:b}-\eqref{prob:example1:d} are satisfied and equals infinity otherwise.\footnote{We recall that \eqref{prob:main} and \eqref{prob:ldr} involve cost functions in which the number of stages with decisions is equal to the number of stages with uncertain variables. The cost function~\eqref{prob:example1:a}-\eqref{prob:example1:d} can be easily modified to match this format  by adding a dummy decision variable $\bx_{T+1} \in \R^E$ and constants $p_{T+1,e} =0$ and $ c_{T+1,e} > 0$ for each $e \in [E]$.} As a result, optimal linear decision rules for production-inventory problems can be obtained by solving~\eqref{prob:ldr}, which we observe can be written as
  \begin{equation} \label{prob:ldr_1} \tag{\textnormal{LDR-1}}
    \begin{aligned}
      & \underset{\by_{t,1},\ldots,\by_{t,t} \in \R^E:\;  \forall t \in [T]}{\textnormal{minimize}}&&\max_{\zeta_1 \in \mathcal{U}_1,\ldots,\zeta_{T+1} \in \mathcal{U}_{T+1}} \left \{ \sum_{t=1}^T \sum_{e=1}^E c_{te}  \left( \sum_{s=1}^{t} y_{t,s,e} \zeta_s \right)  \right \} \\ 
    &\textnormal{subject to}&& \sum_{t=1}^T  \left( \sum_{s=1}^t y_{t,s,e} \zeta_s \right) \le  Q_e &&  \forall e \in[E] \\
   &&& 0 \le \left( \sum_{s=1}^t y_{t,s,e} \zeta_s \right) \le  p_{te} && \forall  e\in[E],\; t \in [T]\\
   &&&V_{\textnormal{min}} \le  v_1 +  \sum_{e=1}^E \sum_{\ell=1}^{t-\delta_e} \left( \sum_{s=1}^\ell y_{\ell,s,e} \zeta_s \right)  - \sum_{s=2}^{t+1} \zeta_s  \le  V_{\textnormal{max}}&& \forall  t \in [T]\\
   &&& \quad \forall \zeta_1 \in \mathcal{U}_1,\ldots,\zeta_{T+1} \in \mathcal{U}_{T+1}.
    \end{aligned}
  \end{equation}

\section{Main Theoretical Result} \label{sec:main_result}
In this section, we present and prove the main theoretical result of this paper, which establishes the sparsity of optimal linear decision rules for the class of production-inventory problems from \S\ref{sec:prodinv}. 
%\subsection{Statement of Theorem~\ref{thm:main}}
Throughout this section, we make the following assumptions.\looseness=-1 
\begin{assumption}\label{ass:1}
\begin{enumerate}[label=\alph*.,ref={\ref{ass:1}\alph*}]
    \item  \eqref{prob:ldr} is feasible and the optimal objective value for \eqref{prob:ldr} is finite.\label{ass:1a}
    \item The uncertainty sets are intervals  $\mathcal{U}_t=[\ubar{D}_t, \bar{D}_t]$ with $\ubar{D}_1 = \bar{D}_1 = 1$ 
    and $\ubar{D}_t < \bar{D}_t$ for all $t \ge 2$.\label{ass:1b}
\end{enumerate}
\end{assumption}
%\vspace{-1em}
The first assumption states that there exist   optimal linear decision rules for the robust optimization problem. 
The second assumption imposes that the uncertain variables in the robust optimization problem are chosen from box uncertainty sets. In the production-inventory problem, this second assumption corresponds to the fact that the customer demand is uncertain but bounded in each stage $t \ge 2$. The claim that $\ubar{D}_1 = \bar{D}_1 = 1$ ensures without loss of generality that linear decision rules can have a nonzero offset. 

In our main theoretical result, presented below as Theorem~\ref{thm:main}, we establish that there  always exists a sparse optimal linear decision rule for the class of production-inventory problems \eqref{prob:ldr_1}:

\begin{theorem} \label{thm:main}
Consider a cost function of the form \eqref{prob:example1:a}-\eqref{prob:example1:d} and let   Assumption~\ref{ass:1} hold. Suppose that $c_{te}>0$ for every $t\in [T]$ and $e\in [E]$, and let $\delta \triangleq\min_{e\in [E]} \delta_e$ denote the minimum lead time. Then there exists an optimal solution $\bar{\by}$ for \eqref{prob:ldr} that satisfies $\| \bar{\by} \|_0 \le  2 + 8E + 10T + 6E(T-\delta).$
\end{theorem}

\noindent Let us make three remarks about the above theorem.

First, we recall that the production-inventory problem with linear decision rules~\eqref{prob:ldr_1} has  $\frac{1}{2}ET(T+1) = \mathcal{O}(E T^2)$ parameters. Theorem~\ref{thm:main} shows that if \eqref{prob:ldr_1} has an optimal solution, then there always exists optimal linear decision rules with $\mathcal{O}(ET)$  nonzero parameters. %As far as we are aware, this is the first theoretical result to show the existence of sparse optimal linear decision rules for a practically-important class of dynamic robust optimization problems with many time periods. 
Furthermore, we notice that the number of parameters for representing static decision rules in production-inventory problems is equal to $TE$. Theorem~\ref{thm:main} thus reveals that the complexity of optimal linear decision rules is at the same level as the complexity of static decision rules, which provides a new perspective for the success of linear decision rules: namely, the fundamental reason that linear decision rules exhibit superior performance to static decision rules in production-inventory problems is not due to the enormous parameter space of linear decision rules, but rather a very small cardinality of significant parameters that are contained only in the linear decision rules.

Second, we observe from Theorem~\ref{thm:main} that if the minimum lead time across the factories $\delta \triangleq\min_{e\in [E]} \delta_e$ converges to $T$, then the number of nonzero parameters in the optimal linear decision rules converges to $\mathcal{O}(E + T)$.  Hence, the upper bound provided by Theorem~\ref{thm:main} conforms with the observation that  additional inventory should only purchased in the early time periods when the lead times of all of the factories are close to the total number of time periods $T$.

Third, 
        we note that Theorem~\ref{thm:main} can be directly extended to production-inventory problems with multiple products in each time period. Specifically, if $\mathcal{U}_t$ is a multi-dimensional box and $\by_{t,s,e}$ is a vector in the corresponding dimension, one can show the existence of a sparse optimal solution using the same proof techniques as Theorem \ref{thm:main}. Furthermore,  our proof techniques behind Theorem~\ref{thm:main} are not limited to the class of production-inventory problems from Ben-Tal et al.~\cite{ben2004adjustable}, and we present extensions of Theorem~\ref{thm:main} to other classes of robust optimization problems in \S\ref{sec:extensions}.\looseness=-1

\subsection*{Proof Roadmap of Theorem~\ref{thm:main}} \label{sec:sparsity:proof}

In the remainder of this section, we present a high-level roadmap for the proof of Theorem~\ref{thm:main}. Our proof is based on a new understanding of the extreme points of the feasible regions of linear decision rules for a general class of dynamic robust optimization problems. We postpone the detailed proof of  Theorem~\ref{thm:main} to Appendix \ref{appx:proof_thm12}. %Our theorem and proof techniques in this section are not limited to the class of production-inventory problems from Ben-Tal et al.~\cite{ben2004adjustable}, and we present extensions of our results to other robust inventory management problems in \S\ref{sec:extensions}.

Our proof of Theorem~\ref{thm:main} focuses on a class of robust optimization   problems that includes production-inventory problems as a special case. This class of robust optimization problems is characterized by cost functions of the form 
\begin{equation} \label{line:cost} \tag{C-G}
\begin{aligned}
C(\bx_1,\ldots,\bx_T,\zeta_1,\ldots,\zeta_T) =&  \sum_{t=1}^T \ba_{0,t}^\intercal \bx_t - \sum_{t=1}^T b_{0,t} \zeta_t  \\
\textnormal{subject to}\quad&  \sum_{t=1}^T \ba_{i,t}^\intercal \bx_t - \sum_{t=1}^T b_{i,t}   \zeta_t  \le c_i   \quad \forall i \in [m],
\end{aligned} 
\end{equation} 
where the parameters of the cost function are  $\ba_{0,t},\ldots,\ba_{m,t} \in \R^{n}$ and $b_{0,t},\ldots,b_{m,t} \in \R$ for each stage $t \in [T]$ and $c_i \in \R$ for each constraint $i \in [m]$. For  robust optimization problems~\eqref{prob:main} with cost functions of the form \eqref{line:cost}, 
we observe that \eqref{prob:ldr} can be reformulated by its epigraph as
\begin{equation}  \tag{LDR-G} \label{prob:ldr_b}
\begin{aligned}
&\underset{\substack{c_0 \in \R\\ \by_{t,1},\ldots,\by_{t,t} \in \R^n:\;  \forall t \in [T]}}{\textnormal{minimize}}&& c_0 \\
&\textnormal{subject to}&& \max_{\zeta_1 \in \mathcal{U}_1, \ldots,\zeta_T \in \mathcal{U}_T}  \left \{ \sum_{s=1}^T\left( - b_{i,s} +  \sum_{t=s}^T \ba_{i,t}^\intercal \by_{t,s} \right)  \zeta_s  \right \}  \le c_i  \;\forall i \in \{0,\ldots,m\}.
\end{aligned}
\end{equation}
The decision variables in the above optimization problem include the parameters of the linear decision rules as well as an epigraph decision variable $c_0 \in \R$. 
For this general class of problems, we make the following additional assumption. 
\begin{assumption}\label{ass:2}  
If $C(\bx_1,\ldots,\bx_T,\zeta_1,\ldots,\zeta_T) < \infty$, then the decisions satisfy $\bx_1,\ldots,\bx_T \ge \bzero$. 
\end{assumption}
The above assumption essentially stipulates that the constraints of~\eqref{line:cost} ensure that feasible decisions for the robust optimization problem satisfy $\bx_t\ge \bzero$ for each stage $t\in[T]$. This is a reasonable assumption because, in practice, the decisions  usually refer to levers like order quantities or prices of products that must take nonnegative values. For example, in the class of  production-inventory problems from Ben-Tal et al.~\cite{ben2004adjustable}, the decisions $\bx_t \in \R^E$ refer to the production quantities at the factories at time period $t$, and the constraint~\eqref{prob:example1:c} guarantees that the production levels are nonnegative.\looseness=-1

Equipped with the above notation for a general class of  dynamic robust optimization problems, our proof of Theorem \ref{thm:main} is split into three major steps, organized below as Lemmas~\ref{lem:1}, \ref{lem:2}, and \ref{lem:zeros}. 
In our first step, denoted below by Lemma~\ref{lem:1}, we characterize the feasible region of \eqref{prob:ldr_b}  and show that the set of feasible solutions of \eqref{prob:ldr_b} is a polyhedron with at least one extreme point.\looseness=-1

\begin{lemma}\label{lem:1}
Let Assumptions~\ref{ass:1} and~\ref{ass:2} hold. Then  the set of feasible solutions to \eqref{prob:ldr_b} is a nonempty polyhedron with at least one extreme point. 
\end{lemma}
Let us make two observations about the above lemma. First, since \eqref{prob:ldr_b} has a linear objective function and a polyhedral feasible region, we observe that \eqref{prob:ldr_b} is a linear optimization problem. Therefore, whenever this linear optimization problem has at least one extreme point and has a finite optimal objective value, there must exist an optimal solution for \eqref{prob:ldr_b} that is an extreme point of its feasible set (see, e.g.,  \cite[Theorem 2.7]{bertsimas1997introduction}). 
Second, it follows from routine arguments in linear optimization that if $(\bar{\by},\bar{c}_0)$ is an extreme point  of \eqref{prob:ldr_b}, then the value of $\bar{c}_0$ is uniquely determined by the value of  $\bar{\by}$.\footnote{Suppose that $(\bar{\by},\bar{c}_0)$ is an extreme point of the set of feasible solutions for \eqref{prob:ldr_b}. Then we readily observe that the value $\bar{c}_0$ must  satisfy  $\bar{c}_0 = \max_{\zeta_1 \in \mathcal{U}_1,\ldots,\zeta_T \in \mathcal{U}_T}\left \{ \sum_{t=1}^T \ba_{i,t}^\intercal  \left(\sum_{s=1}^t \bar{\by}_{t,s} \zeta_s \right) - \sum_{t=1}^T b_{i,t}   \zeta_t  \right \}.$} Therefore, we will for the sake of simplicity omit $\bar{c}_0$ when referring to an extreme point $(\bar{\by},\bar{c}_0)$ of the set of feasible solutions to \eqref{prob:ldr_b}.

In the second step of our proof of Theorem~\ref{thm:main}, we provide an explicit characterization of the extreme points for the set of feasible solutions of \eqref{prob:ldr_b}. This key step, which is presented below as Lemma~\ref{lem:2}, reveals that the extreme points of  the set of feasible solutions of \eqref{prob:ldr_b} can always be represented as the unique solution of a certain decomposable system of equations. 
\begin{lemma} \label{lem:2}
Let Assumption~\ref{ass:1} hold, and let  $(\bar{\by},\bar{c}_0)$ be an extreme point of the feasible set of \eqref{prob:ldr_b}. Then there exists 
an index set $\mathcal{I}^{\bar{\by}} \subseteq \{0,\ldots,m\}$,  
%\item 
an index set $\mathcal{T}_i^{\bar{\by}} \subseteq [T]$  for each $i \in \mathcal{I}^{\bar{\by}}$, and 
%\item  
a hyperplane $(\balpha_{i}^{\bar{\by}}, \beta_{i}^{\bar{\by}})$ for each $i \in \mathcal{I}^{\bar{\by}}$ 
 such that $\bar{\by}$ is the unique solution of the following system.
 \begin{align}
&&\sum_{s=1}^T \sum_{t=s}^T \balpha_{i,t,s}^{\bar{\by}} \cdot  \by_{t,s}   &= \beta_{i}^{\bar{\by}} && \forall i \in \mathcal{I}^{\bar{\by}}, \tag{HARD} \label{line:type1} \\
 && \sum_{t=s}^T \ba_{i,t} \cdot \by_{t,s} &=  b_{i,s} && \forall i \in \mathcal{I}^{\bar{\by}}, s \in \mathcal{T}_i^{\bar{\by}}. \tag{EASY} \label{line:type2}
 \end{align}
\end{lemma}

 In a nutshell, Lemma~\ref{lem:2} establishes that every extreme point of the feasible set of \eqref{prob:ldr_b} is the unique solution to a linear system that can be decomposed into two types of equations: a small number of hard equations  and a large number of easy equations. Indeed, the first type of equations \eqref{line:type1} is defined by hyperplanes $(\balpha_{i}^{\bar{\by}}, \beta_{i}^{\bar{\by}})$ which are functions of $\bar{\by}$, the extreme point of the set of feasible solutions of \eqref{prob:ldr_b}. We refer to this first type of equations by the moniker  \eqref{line:type1} because the structure of the hyperplanes $(\balpha_{i}^{\bar{\by}}, \beta_{i}^{\bar{\by}})$ cannot be analyzed independently of extreme point $\bar{\by}$. 
In contrast, the second type of equations \eqref{line:type2} is defined by hyperplanes that are independent of $\bar{\by}$, and so the structure of the second type of equations can be analyzed statically using the structure of the underlying robust optimization problem. 
The number of equations in \eqref{line:type1} is at most equal to $m+1 = \mathcal{O}(m)$, where $m$ is the number of constraints in \eqref{line:cost}, whereas the number of equations in \eqref{line:type2} is at most equal to  $(m+1)T = \mathcal{O}(mT)$. Note for the production-inventory problem from \S\ref{sec:prodinv} that the number of constraints in lines~\eqref{prob:example1:a}-\eqref{prob:example1:d} is $m = 1 + E + 2ET + 2T = \mathcal{O}(ET)$. 
Hence,  when the number of stages is large, we observe that there can be significantly more equations of type~\eqref{line:type2} than of  type~\eqref{line:type1}.

The third and final step in our proof of Theorem~\ref{thm:main} is to show that the unique solution to the system of equations \eqref{line:type1}-\eqref{line:type2} is sparse when \eqref{line:cost} is equal to the cost function of the production-inventory problem from lines~\eqref{prob:example1:a}-\eqref{prob:example1:d}. 
Specifically, it turns out in many practical problems including the production-inventory problem from Ben-Tal et al.~\cite{ben2004adjustable} 
that the system of equations \eqref{line:type1}-\eqref{line:type2} can be massaged into an instance of a system of equations denoted below \eqref{line:system_lemma:1}-\eqref{line:system_lemma:3}, in which the number of equations in line~\eqref{line:system_lemma:1} grows  linearly in the number of stages of the robust optimization problem. Through this insight, the following Lemma~\ref{lem:zeros} proves that the number of nonzeros in every extreme point of \eqref{prob:ldr_b} in  problems such as the  production-inventory problem from Ben-Tal et al.~\cite{ben2004adjustable} grows linearly with respect to the number of stages.

\begin{lemma} \label{lem:zeros}
Let $\bbp_1 \in \R^{m_1 \times n}$, $\bbp_2 \in \R^{m_2 \times n}$, $\bq \in \R^{m_1}$, and $\mathcal{N} \subseteq [n]$, where $m_1, m_2 \le n$. Suppose that there is a unique $\bar{\bz} \in \R^n$  that satisfies the system of equations
\begin{align}
\bbp_1 \bz &= \bq \label{line:system_lemma:1} \tag{S-1}\\
\bbp_2 \bz &= \bzero \label{line:system_lemma:2} \tag{S-2}\\
z_j &= 0 \quad \forall j \in \mathcal{N} \label{line:system_lemma:3} \tag{S-3},
%\end{aligned} %\label{line:system_lemma}
\end{align}
and suppose that each column of $\bbp_2$ has at most one nonzero entry, that is, $\sum_{i=1}^{m_2} \mathbb{I} \left \{ p_{2,i,j} \neq 0 \right \} \le 1$ for each $j \in [n]$. Then the unique solution $\bar{\bz}$ has at most $2 m_1$ nonzero entries, that is, $\| \bar{\bz} \|_0 \le 2m_1$. 
\end{lemma}

In summary, our proof of Theorem \ref{thm:main} considers a more general class of cost functions \eqref{line:cost} and looks at the epigraph formulation~\eqref{prob:ldr_b} of the corresponding problem of computing optimal linear decision rules. Under reasonable assumptions, we show in~Lemma \ref{lem:1} that \eqref{prob:ldr_b} is indeed a linear optimization problem with at least one extreme point in its feasible region, thus establishing the existence of an optimal extreme point. Furthermore, we show in Lemma \ref{lem:2} that any extreme point of the feasible region is the unique solution to a linear system with two types of equations \eqref{line:type1} and \eqref{line:type2}, where the number of equations in \eqref{line:type1} is significantly smaller than the number of equations in \eqref{line:type2}. Applying Lemma \ref{lem:2} to the production-inventory problem, we obtain the corresponding \eqref{line:type1} and \eqref{line:type2}. With some further reformulation of these equations, we can then utilize Lemma~\ref{lem:zeros} to conclude that the number of nonzeros in every extreme point grows at most linearly with respect to the number of stages of the robust optimization problem. Our formal proof of Theorem~\ref{thm:main} is found in Appendix~\ref{appx:proof_thm12}. 

Lastly, we highlight that in this proof, the structure of the production-inventory problem is utilized only to obtain the corresponding \eqref{line:type1} and \eqref{line:type2} equations using Lemma \ref{lem:2}, and the reformulation of equations is tied with this problem structure. We will show in Theorem \ref{thm:3} how the same proof idea can be utilized to obtain the sparsity results for the dynamic newsvendor problem with a different set of \eqref{line:type1} and \eqref{line:type2} equations and a different reformulation technique.

\section{The Value of Sparsity in Algorithm Design} \label{sec:algorithm}
In this section, we discuss the implications of Theorem~\ref{thm:main} on the design of practically-efficient algorithms for computing linear decision rules in dynamic robust optimization problems. Specifically, our main contribution in this section is a novel algorithm for solving \eqref{prob:ldr} in applications in which the optimal linear decision rules are sparse and the number of time periods is large.  Our  algorithm is applicable to the broad class of dynamic robust optimization    problems in which the  uncertainty sets satisfy Assumption~\ref{ass:1} and the cost function has the form~\eqref{line:cost}\footnote{A discussion of the class of dynamic robust optimization problems in which the  uncertainty sets satisfy Assumption~\ref{ass:1} and the cost function has the form~\eqref{line:cost} can be found in \S\ref{sec:main_result}.}. In \S\ref{sec:experiment},  we show through numerical experiments  that our  algorithm from this section can provide {significant} speedups over state-of-the-art linear optimization solvers for computing near-optimal linear decision rules in robust optimization problems with huge numbers of time periods. 
\subsection{Preliminaries} \label{sec:algorithm:prelim}

The overarching goal of our algorithm, which is presented formally in \S\ref{sec:algorithm:activeset},  is to find an optimal or near-optimal solution to \eqref{prob:ldr} without ever needing to solve a large-scale linear optimization problem. To accomplish this goal, each iteration of our  algorithm consists of solving a restricted version of \eqref{prob:ldr} in which  sparsity is imposed onto the linear decision rules.\looseness=-1

In greater detail, each iteration of our algorithm begins with a choice of an active set $\mathcal{A} \subseteq \{(t,s,j): t,s \in [T], j \in [n], s \le t \}$ corresponding to the parameters of the linear decision rules that will be allowed to take values other than zero. If the uncertainty sets satisfy Assumption~\ref{ass:1} and the cost function  has the form~\eqref{line:cost},  then the optimization problem~\eqref{prob:ldr} with sparsity constraints from  the active set $\mathcal{A}$  is given by\looseness=-1
\begin{equation*} 
\begin{aligned}
&\underset{\substack{c_0 \in \R\\ \by_{t,1},\ldots,\by_{t,t} \in \R^n:\;  \forall t \in [T]}}{\textnormal{minimize}}&& c_0 \\
&\textnormal{subject to}&& \max_{\zeta_1 \in \mathcal{U}_1, \ldots,\zeta_T \in \mathcal{U}_T}  \left \{ \sum_{s=1}^T\left( - b_{i,s} +  \sum_{t=s}^T \ba_{i,t}^\intercal \by_{t,s} \right)  \zeta_s  \right \}  \le c_i  &&\forall i \in \{0,\ldots,m\}\\
&&& y_{t,s,i} = 0 && \forall (t,s,i) \notin \mathcal{A},
\end{aligned}
\end{equation*}
which can be rewritten equivalently as
\begin{equation}  \tag{LDR-$\mathcal{A}$} \label{prob:ldr_A}
\begin{aligned}
&\underset{\substack{c_0 \in \R\\ y_{t,s,j} \in \R :\;  \forall (t,s,j) \in \mathcal{A}}}{\textnormal{minimize}}&& c_0 \\
&\textnormal{subject to}&&   \sum_{s = 1}^T  \max_{\zeta_s \in \mathcal{U}_s } \left \{ \left( - b_{i,s} + \sum_{t,j: (t,s,j) \in \mathcal{A}} a_{i,t,j} y_{t,s,j} \right) \zeta_s \right \}  \le c_i  &&\forall i \in \{0,\ldots,m\}.\end{aligned}
\end{equation}
By solving the optimization problem~\eqref{prob:ldr_A} in each iteration of our algorithm, our algorithm obtains a linear decision rule in which at most $| \mathcal{A}|$ of the parameters of the linear decision rule are nonzero. We readily observe that the optimization problem~\eqref{prob:ldr_A} is a restricted version of the optimization problem~\eqref{prob:ldr}, in the sense that any linear decision rule that is feasible for \eqref{prob:ldr_A} is a feasible linear decision rule for \eqref{prob:ldr}. 

Our motivation for designing an algorithm around the optimization problem~\eqref{prob:ldr_A} is based on problem size. Indeed,  \eqref{prob:ldr_A} is an optimization problem with approximately $| \mathcal{A}|$ decision variables, whereas the optimization problem~\eqref{prob:ldr}  has approximately $nT^2$ decision variables. Hence, if each iteration of our algorithm chooses an active set that is sparse, meaning that $|\mathcal{A}| \ll nT^2$, then  storing and manipulating the decision variables of the optimization problem~\eqref{prob:ldr_A} will require significantly less computer memory than storing and manipulating the decision variables of \eqref{prob:ldr}. Moreover, we established in  \S\ref{sec:main_result} that there can exist  optimal linear decision rule for real-world applications of robust optimization that satisfy $ \| \by \|_0 = \mathcal{O}(nT)$. Therefore,  we observe that it can be possible to choose a sparse active set satisfying $| \mathcal{A}| \ll nT^2$ for which the optimization problem~\eqref{prob:ldr_A} is equivalent to the optimization problem~\eqref{prob:ldr}.

In view of the above motivation, the rest of this section is organized as follows. In  \S\ref{sec:algorithm:fixed_active_set}, we  develop a  reformulation technique that enables us to find an optimal solution to \eqref{prob:ldr_A} in reasonable computation time  when the active set is sparse.
In \S\ref{sec:algorithm:activeset}, we present 
an iterative algorithm that uses duality to find an optimal  active set $\mathcal{A}$ in the optimization problem~\eqref{prob:ldr_A}. 

% In \S\ref{sec:algorithm:improvements}, we discuss several improvements and implementation details to the active set method. % and then 

\subsection{Compact Reformulation For Fixed Active Set} \label{sec:algorithm:fixed_active_set}
In this subsection, we develop a novel    reformulation technique for the optimization problem~\eqref{prob:ldr_A}. 
When the chosen active set is sparse, we show that our  linear optimization reformulation can  have drastically fewer  decision variables and constraints compared to the standard linear optimization reformulation that is obtained by the  traditional robust counterpart technique.\looseness=-1

To motivate our proposed reformulation technique, we begin by discussing the standard reformulation of \eqref{prob:ldr_A} obtained using the robust counterpart technique. 
 Indeed, we observe from strong duality for linear optimization and from Assumption~\ref{ass:1} that the following equality holds for each time period $s \in \{1,,\ldots,T\}$ and  each constraint $i \in \{0,\ldots,m\}$.
 \begin{align*}
&\max_{\zeta_s \in \mathcal{U}_s} \left \{ \left( - b_{i,s} + \sum_{t,j: (t,s,j) \in \mathcal{A}} a_{i,t,j} y_{t,s,j} \right) \zeta_s \right \} = \left[ \begin{aligned}
    &\underset{\bar{\omega}_{i,s},\ubar{\omega}_{i,s} \in \R}{\text{minimize}}&& \bar{D}_s  \bar{\omega}_{i,s}  - \ubar{D}_s \ubar{\omega}_{i,s}  \\
    &\textnormal{subject to}&& \bar{\omega}_{i,s} - \ubar{\omega}_{i,s} =  -b_{i,s} + \sum_{t,j: (t,s,j) \in \mathcal{A}}  a_{i,t,j}  y_{t,s,j} \\
        &&& \bar{\omega}_{i,s}, \ubar{\omega}_{i,s} \ge 0 
    \end{aligned} \right]. 
\end{align*}
Applying the above equality to each period $s \in \{1,\ldots,T\}$ and  constraint $i \in \{0,\ldots,m\}$, the robust counterpart technique yields the  following standard reformulation of \eqref{prob:ldr_A} as a linear optimization problem.
\begin{equation}   \tag{P-$\mathcal{A}$'}\label{prob:P_A_1}
\begin{aligned}
&\underset{\substack{c_0 \in \R\\y_{t,s,j} \in \R :\;  \forall (t,s,j) \in \mathcal{A}\\ \bar{\bomega}_i, \ubar{\bomega}_i \in \R^T: \forall i \in \{0,\ldots,m\}}}{\textnormal{minimize}}&& c_0 \\%\sum_{s =1}^T \left(  \bar{D}_s  \bar{\omega}_{0,s}  - \ubar{D}_s \ubar{\omega}_{0,s} \right) \\
&\textnormal{subject to}&&\sum_{s =1}^T \left(  \bar{D}_s  \bar{\omega}_{i,s}  - \ubar{D}_s \ubar{\omega}_{i,s} \right) \le c_i  && \forall i \in \{0,\ldots,m\} \\
    &&& \bar{\omega}_{i,s} - \ubar{\omega}_{i,s} =  -b_{i,s} + \sum_{t,j: (t,s,j) \in \mathcal{A}}  a_{i,t,j}  y_{t,s,j}  && \forall i \in  \{0,\ldots,m  \}, \;s \in [T] \\
    &&&  \bar{\omega}_{i,s}, \ubar{\omega}_{i,s} \ge 0 && \forall i \in \{0,\ldots,m\}, \; s \in [T].
\end{aligned}
\end{equation}

Unfortunately, regardless of our choice of the active set, we observe that the numbers of decision variables and constraints in the linear optimization problem~\eqref{prob:P_A_1} still grow \emph{quadratically} with respect to the number of time periods. Indeed, the linear optimization problem \eqref{prob:P_A_1} consists of approximately $| \mathcal{A}| + 2mT$ decision variables and $mT$ constraints. Moreover, the number of constraints $m$ in real-world dynamic robust optimization problems typically satisfies the inequality $m \ge nT$.\footnote{For example, in the production-inventory problem from  \S\ref{sec:prodinv} with $E$ factories and $T$ periods, the number of decisions in each period is equal to the number of factories,  $n = E$, and the number of constraints is given by $m = E + E(T+1) + (T+1)$. More generally, the number of constraints will satisfy the inequality $m \ge nT$ in real-world applications in which there are nonnegativity constraints on each of the decisions in each period of the planning problem.}  
For this reason, the numbers of decision variables and constraints in the linear optimization problem~\eqref{prob:P_A_1} will typically grow quadratically with respect to the number of time periods, even when the active set satisfies $| \mathcal{A}| \ll n T^2$.  

 Our novel linear optimization reformulation of the optimization problem~\eqref{prob:ldr_A}, which is stated at the end of this subsection as \eqref{prob:P_A},  can be viewed as a simple modification of the linear optimization problem~\eqref{prob:P_A_1} that   exploits the fact that imposing sparsity constraints onto linear decision rules can have a side effect of inducing redundancy in the constraints of the optimization problem~\eqref{prob:ldr_A}. 
 Indeed, for any active set $\mathcal{A}$ and each period $s \in \{1,\ldots,T\}$, let us define the following set of tuples.
\begin{align*}
    \mathcal{K}^{\mathcal{A},s} \triangleq \bigcup_{i=0}^m \left\{  \left(b_{i,s}, \left( a_{i,t,j} \right)_{t,j: (t,s,j) \in \mathcal{A}}\right)  \right \}. 
\end{align*}
 For each period $s \in \{1,\ldots,T\}$, the cardinality of the set of tuples $K^{\mathcal{A},s} \triangleq | \mathcal{K}^{\mathcal{A},s}|$ can be interpreted as the number of unique  optimization problems $\max_{\zeta_s \in \mathcal{U}_s}  \{ ( - b_{i,s} + \sum_{t,j: (t,s,j) \in \mathcal{A}} a_{i,t,j} y_{t,s,j} ) \zeta_s  \}$ across all of the constraints $i \in \{0,\ldots,m\}$. The key observation that underpins our subsequent reformulation is that the total number of unique tuples $K^{\mathcal{A}} \triangleq \sum_{s=1}^{T} K^{\mathcal{A},s}$ is often proportional to the cardinality of the active set $\mathcal{A}$. In particular, for the production-inventory problem from \S\ref{sec:prodinv}, we have the following result.  \begin{proposition} \label{prop:productioninventory_reformulation}
For the production-inventory problem with no lead times ($\delta_e = 0$ for all $e \in [E]$), we have  $ K^{\mathcal{A}} \le  4 | \mathcal{A}| + ET + 5T + E + 1$. 
\end{proposition}
Hence, the above proposition shows for production-inventory problems that the number of unique tuples satisfies $K^{\mathcal{A}} = \mathcal{O}(ET)$  when the active set satisfies $|\mathcal{A}|= \mathcal{O}(ET)$.

We now complete the development of our proposed reformulation  by demonstrating that \eqref{prob:ldr_A} can be rewritten as a linear optimization problem with $\mathcal{O}(K^{\mathcal{A}} + | \mathcal{A}|)$ decision variables and $\mathcal{O}(K^{\mathcal{A}} + m)$   constraints. 
We begin with the following lemma, which shows how the sets of tuples $\mathcal{K}^{\mathcal{A},1},\ldots, \mathcal{K}^{\mathcal{A},1}$ can be used to combine auxiliary decision variables in the linear optimization problem~\eqref{prob:P_A_1}:\begin{lemma}\label{lem:reform}
    If Assumption~\ref{ass:1} holds, then there exists an optimal solution for \eqref{prob:P_A_1} such that
        \begin{align*}
       (\bar{\omega}_{i,s}, \ubar{\omega}_{i,s}) = (\bar{\omega}_{i',s}, \ubar{\omega}_{i',s})
    \end{align*}
    for all $s \in [T]$ and for all $i,i' \in \{0,\ldots,m\}$ that map to the same tuple in $\mathcal{K}^{\mathcal{A},s}$, where we say that $i,i'$ map to the same tuple in $\mathcal{K}^{\mathcal{A},s}$ if and only if
    \begin{align*}
\left(b_{i,s}, \left( a_{i,t,j} \right)_{t,j: (t,s,j) \in \mathcal{A}}\right) = \left(b_{i',s}, \left( a_{i',t,j} \right)_{t,j: (t,s,j) \in \mathcal{A}}\right).
    \end{align*}
\end{lemma}
The key takeaway from Lemma~\ref{lem:reform} is that we can reduce the size of the linear optimization problem~\eqref{prob:P_A_1}  by combining auxiliary decision variables that are guaranteed  to be equal at an optimal solution. To apply this takeaway,   let the $k$th tuple in $\mathcal{K}^{\mathcal{A},s}$ for each $k \in \{1,\ldots,| \mathcal{K}^{\mathcal{A},s}|\} \equiv [K^{\mathcal{A},s}]$ and each period $s \in [T]$ be denoted by $
   ({b}^{\mathcal{A},s}_{k,s}, ( a^{\mathcal{A},s}_{k,t,j} )_{t,j: (t,s,j) \in \mathcal{A}})$, 
and let $\pi^{\mathcal{A},s}: \{0,\ldots,m\} \to [K^{\mathcal{A},s}]$  be defined for each period $s \in [T]$ as  the mapping that satisfies the following equality for each $i \in \{0,\ldots,m\}$.\looseness=-1
\begin{align*}
 \left(b_{i,s}, \left( a_{i,t,j} \right)_{t,j: (t,s,j) \in \mathcal{A}}\right) =   \left({b}^{\mathcal{A},s}_{ \pi^{\mathcal{A},s}(i),s}, \left( a^{\mathcal{A},s}_{ \pi^{\mathcal{A},s}(i),t,j} \right)_{t,j: (t,s,j) \in \mathcal{A}}\right). 
\end{align*}
The purpose of the mapping $\pi^{\mathcal{A},s}$ is to identify the index of the tuple in $\mathcal{K}^{\mathcal{A},s}$ that is mapped to by a given constraint $i \in \{0,\ldots,m\}$. Said another way, it follows from the terminology in Lemma~\ref{lem:reform} that if $s \in [T]$, then $i,i' \in \{0,\ldots,m\}$ map to the same tuple in $\mathcal{K}^{\mathcal{A},s}$ if and only if $\pi^{\mathcal{A},s}(i) = \pi^{\mathcal{A},s}(i')$.   With the above notation, we  observe that \eqref{prob:P_A_1} can be rewritten equivalently as
\begin{equation*}  
\begin{aligned}
&\underset{\substack{c_0 \in \R\\y_{t,s,j} \in \R :\;  \forall (t,s,j) \in \mathcal{A}\\ \bar{\bomega}_i, \ubar{\bomega}_i \in \R^T: \forall i \in \{0,\ldots,m\}}}{\textnormal{minimize}}&& c_0 \\
&\textnormal{subject to}&&\sum_{s =1}^T \left(  \bar{D}_s  \bar{\omega}_{i,s}  - \ubar{D}_s \ubar{\omega}_{i,s} \right) \le c_i  && \forall i \in \{0,\ldots,m\} \\
    &&& \bar{\omega}_{i,s} - \ubar{\omega}_{i,s} =   - b^{\mathcal{A},s}_{\pi^{\mathcal{A},s}(i),s} + \sum_{t,j: (t,s,j) \in \mathcal{A}} a^{\mathcal{A},s}_{\pi^{\mathcal{A},s}(i),t,j} y_{t,s,j}  && \forall i \in  \{0,\ldots,m  \}, \;s \in [T] \\
    &&&  \bar{\omega}_{i,s}, \ubar{\omega}_{i,s} \ge 0 && \forall i \in \{0,\ldots,m\}, \; s \in [T].
\end{aligned}
\end{equation*}
In particular, it follows from Lemma~\ref{lem:reform} that there exists an optimal solution for the above optimization problem for which the equality  $(\bar{\omega}_{i,s}, \ubar{\omega}_{i,s}) = (\bar{\omega}_{i',s}, \ubar{\omega}_{i',s})$ holds for all constraints $i,i' \in \{0,\ldots,m\}$ and all periods $s \in [T]$ that satisfy the equality  $\pi^{\mathcal{A},s}(i) = \pi^{\mathcal{A},s}(i')$. Therefore, the above linear optimization problem can be rewritten as 
\begin{equation}  \tag{P-$\mathcal{A}$}\label{prob:P_A}
\begin{aligned}
&\underset{\substack{c_0 \in \R\\y_{t,s,j} \in \R :\;  \forall (t,s,j) \in \mathcal{A}\\ \bar{\omega}^{\mathcal{A},s}_{k,s}, \ubar{\omega}^{\mathcal{A},s}_{k,s} \in \R: \forall s \in [T], k \in [K^{\mathcal{A},s}]}}{\textnormal{minimize}}&& c_0 \\
&\textnormal{subject to}&&\sum_{s =1}^T \left(  \bar{D}_s  \bar{\omega}^{\mathcal{A},s}_{\pi^{\mathcal{A},s}(i),s}  - \ubar{D}_s \ubar{\omega}^{\mathcal{A},s}_{\pi^{\mathcal{A},s}(i),s}  \right) \le c_i  && \forall i \in \{0,\ldots,m\} \\
    &&&\bar{\omega}^{\mathcal{A},s}_{k,s}  - \ubar{\omega}^{\mathcal{A},s}_{k,s}  =   - b^{\mathcal{A},s}_{k,s} + \sum_{t,j: (t,s,j) \in \mathcal{A}} a_{k,t,j}^{\mathcal{A},s} y_{t,s,j}  && \forall s \in [T], k \in [K^{\mathcal{A},s}] \\
    &&&  \bar{\omega}^{\mathcal{A},s}_{k,s} , \ubar{\omega}^{\mathcal{A},s}_{k,s}  \ge 0 && \forall s \in [T], k \in [K^{\mathcal{A},s}],
\end{aligned}
\end{equation} with the understanding that any optimal solution of~\eqref{prob:P_A} can be transformed into an optimal solution for \eqref{prob:P_A_1} by applying the  transformation $\left( \bar{\omega}_{i,s},\ubar{\omega}_{i,s} \right) \triangleq \left( \bar{\omega}^{\mathcal{A},s}_{\pi^{\mathcal{A},s}(i),s},\ubar{\omega}^{\mathcal{A},s}_{\pi^{\mathcal{A},s}(i),s} \right) $ for each period $s \in [T]$ and constraint $i \in \{0,\ldots,m\}$. We observe from inspection that the linear optimization problem~\eqref{prob:P_A} has $1 + | \mathcal{A}| + 2 K^{\mathcal{A}} = \mathcal{O}\left( K^{\mathcal{A}}  + | \mathcal{A}|  \right) $ decision variables and $1 + m + K^{\mathcal{A}} = \mathcal{O} \left( K^{\mathcal{A}}  +m  \right)$ constraints. The linear optimization problem~\eqref{prob:P_A} thus constitutes our novel linear optimization reformulation of the optimization problem~\eqref{prob:ldr_A}.

\subsection{Active Set Method} \label{sec:algorithm:activeset}
Equipped with our compact linear optimization reformulation of \eqref{prob:ldr_A} from the end of \S\ref{sec:algorithm:fixed_active_set}, we now formally present our algorithm for solving the optimization problem~\eqref{prob:ldr}. 
Our algorithm for finding an optimal solution for \eqref{prob:ldr} consists of iteratively  solving the optimization problem~\eqref{prob:ldr_A} with different choices of the active set $\mathcal{A}$. 

A line-by-line description of the algorithm is found in Algorithm~\ref{al:active_set}. 
Firstly, we initialize the active set with 
\begin{align}\label{eq:initial_as}
\mathcal{A} &= \left \{ (t,1,j): t \in [T], j \in [n] \right \} \cup  \left \{ (t,t,j): t \in [T], j \in [n] \right \}. 
\end{align} 
We refer to the initial active set defined above as the \emph{Markovian active set}. In the case of the Markovian active set,   the optimization problem~\eqref{prob:ldr_A} will yield linear decision rules where the decisions in each period $t \in [T]$ have the form 
$\bx_{t} = \by_{t1}  + \by_{tt} \zeta_t$. We then perform the following steps.

\SetKwBlock{Repeat}{repeat}{}
\begin{algorithm}
    \textbf{Initialization:} initialize the active set with the Markovian active set~\eqref{eq:initial_as}.
    
    \Repeat{
    {Solve the dual LP~\eqref{prob:D_A} with the current active set $\mathcal{A}$ (Step 1)\;}
    \If{The solution to \eqref{prob:D_A} satisfies the termination condition \eqref{line:terminate} (Step 2)}{break\;}
    {Update active set $\mathcal{A}$ by the procedure described in Step 3\;} 
    }
    \textbf{Output:} The KKT solution to the dual LP \eqref{prob:D_A}.
    \caption{An active set method for solving \eqref{prob:ldr}}
    \label{al:active_set}
\end{algorithm}

\subsubsection*{Step 1.} Given the active set $\mathcal{A}$, the first step of the current iteration of the active set method consists of solving the dual of the linear optimization problem~\eqref{prob:P_A}, which is given by
{\small
\begin{equation}  \label{prob:D_A} \tag{D-$\mathcal{A}$}
\begin{aligned}
&\underset{ \substack{
\lambda_0,\ldots,\lambda_m \in \R, \\
\zeta^{\mathcal{A},s}_{k,s} \in \R \; \forall s \in [T], k \in [K^{\mathcal{A},s}]}}{\textnormal{maximize}}&& - \sum_{i=1}^m c_i \lambda_i -  \sum_{s=1}^T  \sum_{k=1}^{K^{\mathcal{A},s}} b^{\mathcal{A},s}_{k,s} \zeta^{\mathcal{A},s}_{k,s}  \\
&\textnormal{subject to}&& \begin{aligned}[t]
&\sum_{k=1}^{K^{\mathcal{A},s}}  a^{\mathcal{A},s}_{k,t,j}  \zeta^{\mathcal{A},s}_{k,s}   = 0 && \forall (t,s,j) \in \mathcal{A} \\
&  \ubar{D}_s \left( \sum_{i \in \{0,\ldots,m\}: \pi^{\mathcal{A},s}(i) = k} \lambda_{i} \right)   \le \zeta^{\mathcal{A},s}_{k,s}  \le \bar{D}_s \left( \sum_{i \in \{0,\ldots,m\}: \pi^{\mathcal{A},s}(i) = k} \lambda_{i} \right)  && \forall  s \in [T], k \in [ K^{\mathcal{A},s}]\\
& \lambda_0 = 1 \\
&\lambda_i \ge 0 && \forall i \in \{1,\ldots,m\}. 
\end{aligned}
\end{aligned} 
\end{equation}
}%
We present how \eqref{prob:D_A} can be efficiently solved by removing redundant constraints in Appendix \ref{app:solveDA}. 
To simplify the discussion of our algorithm, we assume for now that the optimal objective value of \eqref{prob:D_A} is finite (we discuss how to handle infeasible cases in Appendix \ref{sec:algorithm:infeas}). In that case, it follows from strong duality for linear optimization  that the optimal objective value of \eqref{prob:D_A} is equal to the optimal objective value of \eqref{prob:ldr_A}. Moreover, we observe that an optimal linear decision rule for the optimization problem~\eqref{prob:ldr_A} can be obtained by extracting the KKT solutions of the equality constraints $\sum_{k=1}^{K^{\mathcal{A},s}}  a^{\mathcal{A},s}_{k,t,j}  \zeta^{\mathcal{A},s}_{k,s}   = 0$ for all $(t,s,j) \in \mathcal{A}$. Hence,  if the optimization problem~\eqref{prob:D_A} has a feasible solution, then a feasible solution for the optimization problem~\eqref{prob:ldr} with objective value equal to the optimal objective value of \eqref{prob:ldr_A} can be obtained from solving the linear optimization problem~\eqref{prob:D_A} using the simplex method.\footnote{If the linear optimization problem~\eqref{prob:D_A} is solved using the simplex algorithm, then the KKT solution can be extracted directly from the optimal basis, where $y_{t,s,j}$ is the KKT solution.  }

\subsubsection*{Step 2.} The next step of the current iteration of our algorithm consists of checking whether the linear decision rule extracted from the KKT solution of the linear optimization problem~\eqref{prob:D_A} is an 
 optimal solution of the optimization problem~\eqref{prob:ldr}. To perform this step, we utilize the following proposition. In the following proposition, and throughout the rest of the paper,  we define the fraction $0/0$ to be equal to $0$. 
   \begin{proposition} \label{prop:terminate}
Consider any optimal solution for \eqref{prob:D_A}. If the solution satisfies
\begin{align}
 \sum_{i=0}^m   a_{i,t,j} \left( \frac{\lambda_i}{\sum_{i' \in \{0,\ldots,m\}: \pi^{\mathcal{A},s}(i') = \pi^{\mathcal{A},s}(i)} \lambda_{i'}} \right) \zeta^{\mathcal{A},s}_{\pi^{\mathcal{A},s}(i),s} = 0 \quad \forall (t,s,j) \notin \mathcal{A}, \label{line:terminate}
 \end{align} 
then the linear decision rule extracted from the KKT solution of the linear optimization problem~\eqref{prob:D_A} is an 
 optimal solution of the optimization problem~\eqref{prob:ldr}. 
 \end{proposition}
 The above proposition provides the condition that we will use for terminating the algorithm.  
Namely, after solving the optimization problem~\eqref{prob:D_A}, our algorithm checks whether the optimal solution that was obtained from solving the optimization problem~\eqref{prob:D_A} satisfies line~\eqref{line:terminate}. If line~\eqref{line:terminate} is satisfied, then the algorithm terminates. Otherwise, our algorithm proceeds to Step 3. We note that if the active set satisfies $| \mathcal{A}| \ll nT^2$, then the equalities in line~\eqref{line:terminate} can be checked efficiently from the perspectives of computation time as well as computer memory (see Appendix~\ref{sec:evaluating_termination_condition} for implementation details). 
 
\subsubsection*{Step 3.}
 If the algorithm does not terminate in Step 2, then the last step of the current iteration of our algorithm consists of choosing the active set that will be used in the next iteration. To choose the new active set, we employ a heuristic that uses the optimal solution for the linear optimization problem~\eqref{prob:D_A} to construct a new active set by adding elements to the current  active set. In greater detail, let us define the following sets.
 \begin{align*}
 \mathcal{A}^{\neq} &\triangleq \left \{(t,s,j) \notin \mathcal{A}:  \quad  \sum_{i=0}^m   a_{i,t,j} \left( \frac{\lambda_i}{\sum_{i' \in \{0,\ldots,m\}: \pi^{\mathcal{A},s}(i') = \pi^{\mathcal{A},s}(i)} \lambda_{i'}} \right) \zeta^{\mathcal{A},s}_{\pi^{\mathcal{A},s}(i),s} \neq 0 \right \};\\
  \mathcal{A}^{=} &\triangleq \left \{(t,s,j) \notin \mathcal{A}:  \quad  \sum_{i=0}^m   a_{i,t,j} \left( \frac{\lambda_i}{\sum_{i' \in \{0,\ldots,m\}: \pi^{\mathcal{A},s}(i') = \pi^{\mathcal{A},s}(i)} \lambda_{i'}} \right) \zeta^{\mathcal{A},s}_{\pi^{\mathcal{A},s}(i),s} = 0 \right \}.
 \end{align*}
  The set $ \mathcal{A}^{\neq}$ can be interpreted as the set of all tuples $(t,s,j) \notin \mathcal{A}$ that do not satisfy the equality in line~\eqref{line:terminate}, whereas the set $\mathcal{A}^=$ can be interpreted as the set of all tuples $(t,s,j) \notin \mathcal{A}$ that do satisfy the equality in line~\eqref{line:terminate}. To make sense of the above definitions, we offer the following remarks.
  \begin{remark}
   It follows from construction that $\mathcal{A} \cup \mathcal{A}^{\neq} \cup \mathcal{A}^= = \{(t,s,j): t,s \in [T], j \in [n], s \le t \}$.
   \end{remark}
  \begin{remark}
It follows from the fact that the algorithm did not terminate in Step 2 that $\mathcal{A}^{\neq} \neq \emptyset$. 
\end{remark}
\begin{remark}
The sets $\mathcal{A}^{\neq}, \mathcal{A}^=$ are computed for free during Step 2. 
\end{remark}
In view of the above definitions of the sets $\mathcal{A}^{\neq}$ and $\mathcal{A}^=$, we now state our heuristic for  choosing the active set that will be used in the next iteration. Specifically, our heuristic chooses the active set that will be used in the next iteration as $\mathcal{A} \cup \widetilde{\mathcal{A}}$ for some nonempty subset $\widetilde{\mathcal{A}} \subseteq \mathcal{A}^{\neq}$, and the motivation behind our heuristic is given by the following proposition. 
\begin{proposition} \label{prop:monotone} 
Consider any optimal solution of the linear optimization problem~\eqref{prob:D_A}, and let $\mathcal{A}^{\neq}$ and $\mathcal{A}^=$ be the subsets of tuples corresponding to that optimal solution.  If the active set in the next iteration is chosen to be $\mathcal{A} \cup \widetilde{\mathcal{A}}$ for some subset $\widetilde{\mathcal{A}} \subseteq \mathcal{A}^=$, then the optimal objective value of \textnormal{(LDR-$\mathcal{A} \cup \widetilde{\mathcal{A}}$)} will be equal to the optimal objective value of \eqref{prob:ldr_A}. 
\end{proposition}
To make sense of the above proposition, we recall that the goal of our algorithm is to find an active set $\mathcal{A}$ for which the optimal objective value of \eqref{prob:ldr_A} is equal to the optimal objective value of \eqref{prob:ldr}. Moreover, we recall that the optimal objective value of \eqref{prob:ldr_A} (which is equal to the optimal objective value of \eqref{prob:D_A}) is always greater than or equal to the optimal objective value of \eqref{prob:ldr}. In view of those recollections, Proposition~\ref{prop:monotone} shows that choosing an active set in the next iteration of algorithm that includes elements from $\mathcal{A}^{\neq}$ is imperative for decreasing the optimal objective value of the linear optimization problem~\eqref{prob:D_A}.  This motivates our heuristic for choosing the active set that will be used in the next iteration to be  $\mathcal{A} \cup \widetilde{\mathcal{A}}$ for some nonempty subset $\widetilde{\mathcal{A}} \subseteq \mathcal{A}^{\neq}$.\looseness=-1

\begin{remark}In our implementation of the algorithm,  we choose the subset of tuples by initializing $\widetilde{\mathcal{A}} \leftarrow \emptyset$. We then iterate over each $t \in [T]$ and $j \in [n]$, draw a random period $s \in \{1,\ldots,s\}$ uniformly over all of the periods $s$ that satisfy $(t,s,j) \in \mathcal{A}^{\neq}$, and append $\widetilde{\mathcal{A}} \leftarrow \widetilde{\mathcal{A}} \cup \{(t,s,j) \}$. By following this procedure, the active set will increase by at most $Tn$ components in each iteration. 
\end{remark}
\begin{remark}
To prevent the active set from becoming unnecessarily  large, we  employ a classical removal strategy from the cutting plane literature (Deletion Rule II  from \cite{elzinga1975central}). To employ this removal strategy, we maintain a record $v_{t,s,j} \in \R$ for each tuple $(t,s,j) \in \mathcal{A}$ of the optimal objective value of the optimization problem~\eqref{prob:ldr_A} in the iteration in which the tuple $(t,s,j)$ was added into the active set. At the end of Step 3 of each iteration of our algorithm, we remove each tuple $(t,s,j) \in \mathcal{A}$ from the active set for which the optimal objective value of \eqref{prob:D_A} is strictly less than $v_{t,s,j}$ and the dual variable $y_{t,s,j}$ associated with the equality constraint  $\sum_{k=1}^{K^{\mathcal{A},s}}  a^{\mathcal{A},s}_{k,t,j}  \zeta^{\mathcal{A},s}_{k,s}   = 0 $ is equal to zero. It is established by \cite{elzinga1975central} that the active set method with this removal strategy is guaranteed to converge to an optimal solution for \eqref{prob:ldr} after finitely many iterations.  

\end{remark}

\looseness=-1

\section{Numerical Experiments} \label{sec:experiment}

In this section, we showcase the practical value of our  reformulation technique and active set method from \S\ref{sec:algorithm} for computing linear decision rules for robust optimization problems with hundreds of time periods and as many as fifty decision and state variables in each time period. Additional numerical experiments that investigate the bounds from  Theorem~\ref{thm:main} can be found in Appendix~\ref{appx:numerical:implications}. All experiments were conducted on a Apple M3 Max chip with all threads and 128 GB of memory.

We perform numerical experiments  in this section that generalize those from  \cite{ben2004adjustable,de2016impact}.  Specifically, our numerical experiments focus on instances of \eqref{prob:ldr_1} in which the customer demand and production costs of a new product follow a cyclic  pattern due to seasonality over a selling horizon of one year. 
Given a discretization of the selling season into $T$ stages, the  customer demand in each stage $t \in \{2,\ldots,T+1\}$ is given by
\begin{align}
    \phi_t &= 1 + 0.5 \sin\left(\frac{2 \pi (t-2)}{T} \right), & \theta_t &= 0.2, &  \mathcal{U}_t &=  \left[ \frac{1000  (1 - \theta)  \phi_t}{T / 24}, \frac{1000 (1+\theta)  \phi_t}{T/24}  \right], \label{prob:experiment_setup:1}
\end{align}
where parameters $\phi_t$ and $\theta_t$ are interpreted, respectively, as a phase parameter, which captures seasonality, and a demand parameter, which controls the radius of the uncertainty sets. Given $E$ factories with zero lead time ($\delta_e = 0$ for all $e \in [E]$) available to the firm, 
the production costs and capacities for each stage $t \in [T]$ and each factory $e \in [E]$ are 
\begin{align}
    c_{te} &= \left(1 + \frac{e-1}{E-1} \right) \phi_t, & p_{te} &= \frac{567}{\left( T/24\right)\left(E/3 \right)}, & Q_e &= \frac{13600}{E / 3},\label{prob:experiment_setup:2}
\end{align}
    and the capacities and initial inventory at the central warehouse are 
    \begin{gather} V_{\text{min}} = 500, \quad V_{\text{max}} = 2000, \quad v_1 = 500.\label{prob:experiment_setup:3}
\end{gather}
In the case of $T=24$ and $E =3$, our parameters are equivalent to those of \cite[\S 5]{ben2004adjustable} and \cite[Table 1]{de2016impact}. We follow this setup from \cite{ben2004adjustable,de2016impact} to capture a realistic setting for the parameters of the problem, and our generalization of the setup allows us to explore different choices of $T$ and $E$.

Our analysis here focuses on the regime where the number of periods $T$, in which case the robust optimization problem serves as an approximation of a continuous-review ordering system, and the number of factories $E$ grow large.\footnote{We note for this application that the dimension of the state and decision spaces in each period scale linearly in the number of  factories. Specifically, for the production-inventory problem, the state space is determined by the inventory at the warehouse and the cumulative production quantity at each factory in each time period. The decisions include the production quantity at each factory in each time period.}   For this regime, we compare the practical efficiency of our reformulation technique and active set method  from \S\ref{sec:algorithm}  to the practical efficiency of using the robust counterpart technique. In our implementation of the active set method, we apply all of the improvements to the active set method from Appendix~\ref{sec:algorithm:improvements} and solve \eqref{prob:D_A} in each iteration with Gurobi using the dual simplex method.  In our implementation of the robust counterpart, we solve the dual linear optimization formulation of the robust counterpart of \eqref{prob:ldr_1}  (see problem~\eqref{prob:D} in Appendix~\ref{appx:D}) with Gurobi  using the barrier method.\footnote{We here present the comparison with the Gurobi barrier method because it is the most effective algorithm of Gurobi in numerical performance. Comparisons of computation times for solving the robust counterpart with primal simplex, dual simplex,  barrier method, and the HiGHS implementation of PDLP} are found in Appendix~\ref{appx:numerical:lpcomparison}.  

In Figure~\ref{fig:active_set_fixed_E}, we compare the computation time of our  reformulation technique coupled with the active set method to the  computation time of the robust counterpart technique for approximately solving production-inventory problems with linear decision rules.  These experiments focus on the scalability of these methods with $E=5$ factories and $T \in \{48,96,144,192,240 \}$ time periods. Figure~\ref{fig:active_set_fixed_E} reports on the computation time for the active set method to obtain feasible linear decision rules that are within 10\%, 1\%, and 0.1\% of optimal. To perform a direct comparison, we terminate the barrier method when the primal-dual gap reaches 10\%, 1\%, and 0.1\%. To make the experimental setup more favorable to the barrier method, we  disable the crossover functionality of the barrier method to reduce its computation time.\footnote{Since Gurobi utilizes infeasible-interior-point method in their barrier implementation, with such tolerance level, it returns neither a primal feasible solution nor a dual feasible solution to the linear optimization problem obtained by taking the robust counterpart. Rather, the barrier method only produces bounds on the optimal objective value of the robust counterpart and near feasible solutions. In contrast, our approach outputs a primal feasible solution whenever \eqref{prob:D_A} has an optimal solution (see Appendix~\ref{sec:algorithm:infeas}). In all numerical experiments shown in this subsection, \eqref{prob:D_A} provided an optimal solution in each iteration of the active set method. }  Optimality gaps for the active set method are obtained by comparing the objective value of the  feasible linear decision rules obtained by the active set method to the optimal objective value obtained by solving the robust counterpart to optimality.\looseness=-1

The results of these experiments reveal that our reformulation technique coupled with the active set method offers significant improvements in computation time compared to the robust counterpart technique. Indeed, for problems with $T = 240$ time periods and $E = 5$ factories, our approach is over 170x faster  in computing linear decision rules that are within 10\% of optimal, and over 32x faster in computing linear decision rules that are within 1\% of optimal. %our active set method found a feasible linear decision rule with a robust objective value that is within 10\% of optimal in fewer than 3 seconds, and it found a feasible linear decision rule that is within 1\% of optimal in fewer than 13 seconds. In contrast, the barrier method required than 400 seconds to obtain a primal-dual gap that is less than 10\%. 
More broadly, the computation times for our  reformulation technique and active set method are a very small fraction of the  computation times for the  robust counterpart for all problem sizes, and the gap in computation times becomes increasingly pronounced as the number of time periods increases. %These experiments thus demonstrate that exploiting sparsity guarantees can significantly reduce  the computation time in order to obtain near-optimal linear decision rules for dynamic robust optimization problems with large time horizons. 

In Figure~\ref{fig:active_set_fixed_T}, we show the relationship between computation time and objective value  of linear decision rules obtained by our  reformulation technique and active set method in experiments with $T=240$ time periods and $E \in \{10,20,30,40,50\}$ factories. The takeaways from this figure are three-fold. First, we observe from Figure~\ref{fig:active_set_fixed_T} that  sparse linear decision rules with the  Markovian active set (see \S\ref{sec:algorithm:activeset})  are not optimal in general, as the objective value of the linear decision rules found by the active set method decreases as the number of iterations increases. Second, Figure~\ref{fig:active_set_fixed_T} shows that our  reformulation technique and active set method scales to applications with high-dimensional state spaces and hundreds of time periods. In particular, the  large numbers of factories here leads to a number of decision and state variables that prevent methods such as robust dual dynamic programming~\cite{georghiou2019robust} from being applied.\footnote{Numerical experiments from \cite[p.827-p.828]{georghiou2019robust} suggest that the robust dual dynamic programming technique scales gracefully with respect to the number of time periods but scales slowly with respect to the number of decision and state variables in each time period.} Third, the practical efficiency of our approach is particularly notable given that the size of the linear optimization problem obtained by the robust counterpart for $T = 240$ time periods and $E = 50$ factories is well beyond the capabilities of classic linear optimization solvers. Indeed, for this problem setting, the (primal) linear optimization problem obtained by the robust counterpart technique would require 13,281,993 decision variables to represent parameters of the linear decision rule and the auxiliary variables, and would require 5,936,502 constraints and 250,895,943 nonzeros. 
\begin{figure}[t]
\FIGURE{\centering \includegraphics[width=1\linewidth]{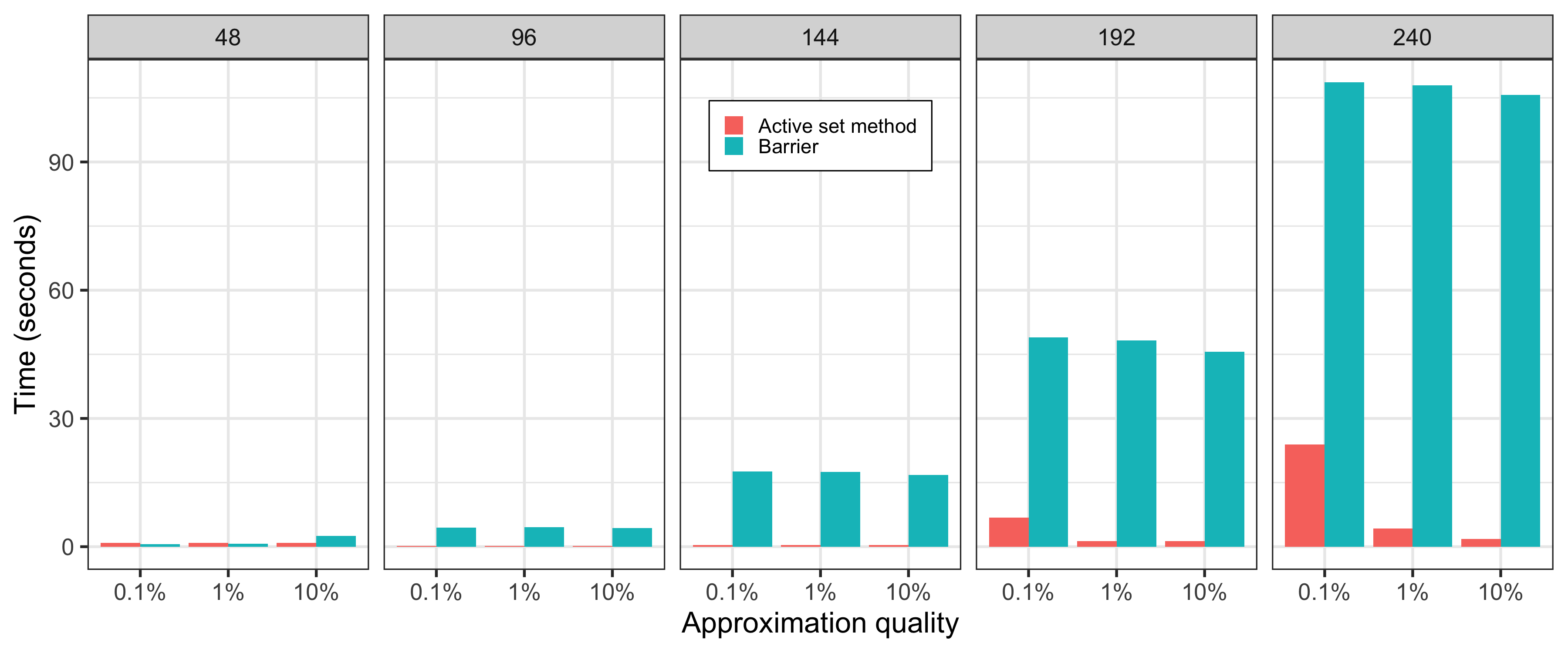} }
{Computation times for active set method and robust counterpart with $E = 5$ factories. \label{fig:active_set_fixed_E}}
   {Results shown for experiments with $E=5$ factories and $T \in \{48,96,144,192,240\}$ time periods. Blue bars show the computation times (in seconds) for the barrier method applied to the robust counterpart to have a primal dual gap that is within 10\%, 1\%, and 0.1\%. Barrier method is run without performing crossover, and so the barrier does not return a primal or dual feasible solution. Red bars show the computation time (in seconds) for the active set method to obtain a feasible linear decision rule that is within 10\%, 1\%, and 0.1\% of optimal. Optimality gaps are obtained by comparing the objective value of the linear decision rules obtained by the active set method to the optimal objective value obtained by solving the robust counterpart to optimality.  }
\end{figure}

\begin{figure}[t]
\FIGURE{\centering \includegraphics[width=1\linewidth]{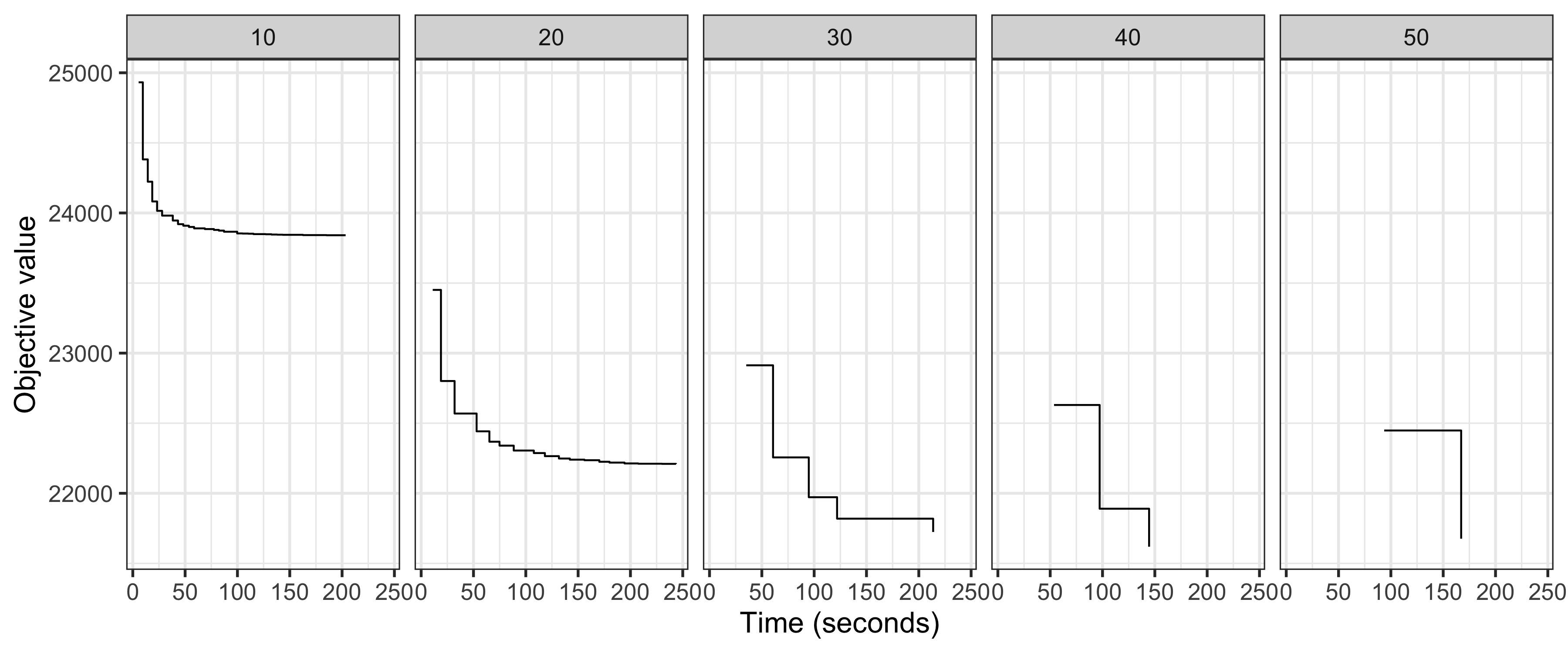} }
{Objective value and computation time for active set method with $T = 240$ time periods. \label{fig:active_set_fixed_T}}
   {Results shown for experiments with $T=240$ time periods and  $E \in \{10,20,30,40,50\}$ factories. Black lines show the objective value of the linear decision rule found by the active set method as a function of computation time.  }
\end{figure}

\section{Extensions and Limitations}\label{sec:extensions}

The proof techniques developed in this paper for establishing the sparsity of optimal linear decision rules (see \S\ref{sec:main_result}) are tailored for the class of production-inventory problems from \S\ref{sec:prodinv} with box uncertainty sets. In this section, we show that our proof techniques from \S\ref{sec:main_result} can also be extended to establish sparsity guarantees for a class of dynamic newsvendor problems with box uncertainty sets from \cite{ben2005retailer,bertsimas2010optimality,iancu2013supermodularity} and to the class of production-inventory problems from \S\ref{sec:prodinv} with non-box uncertainty sets.  We emphasize that while
the two extensions in this section (Theorems~\ref{thm:3} and \ref{thm:nonbox}) share the high-level proof roadmap of Theorem~\ref{thm:main} stated in Section \ref{sec:main_result}, the proof details, more specifically, the construction of the \eqref{line:type1} and \eqref{line:type2} constraints in Lemma \ref{lem:2} and the reformulation of these constraints to (S-1) to (S-3) in Lemma \ref{lem:zeros}, require individual analysis. In particular, the main technical difficulty in applying this analysis lies in the reformulation step that converts a large amount of the constraints \eqref{line:type2} to the constraints (S-2) type, which do not contribute to nonzero entries in the solution, and such a reformulation generally requires careful exploration of the problem structures. The broadest class of robust optimization problems that possess sparse optimal linear decision rules remains an open question, and we note that generalizations from this section and the production-inventory problem from \S\ref{sec:prodinv} all are limited to continuous decision variables and polyhedral cost functions.   Nevertheless, we hope the developments in this paper inspire and aid the robust optimization community in identifying more examples of robust optimization problems where the existence of sparse optimal solutions is guaranteed. 

%In this section, we show that our proof techniques from \S\ref{sec:main_result} can be extended  to establish  sparsity guarantees for other classes of robust optimization problems. These extensions thus demonstrate that the sparsity guarantees and proof techniques developed in this paper are not exclusive to the class of production-inventory problems from Ben-Tal et al.~\cite{ben2004adjustable}, and they provide a starting point for using our proof techniques to establish sparsity guarantees for other applications.   

\paragraph{Dynamic Newsvendor Problem. } 
A significant research effort in the robust optimization literature has been dedicated to {dynamic newsvendor problems} with box uncertainty sets; see  \cite{ben2005retailer,bertsimas2010optimality,iancu2013supermodularity}. This class of dynamic newsvendor problems is characterized by a single factory and nonlinear convex cost functions which capture the holding and backorder costs for inventory at the warehouse. Specifically, the cost function for these dynamic newsvendor problems is given by 
\begin{align} 
%\begin{aligned}
C(x_1,\ldots,x_T, \zeta_1,\ldots,\zeta_{T+1}) =  &  \sum_{t=1}^T \left(  c_{t} x_{t} +  h_t \left[ v_1 + \sum_{s=1}^t x_{s} - \sum_{s=2}^{t+1} \zeta_s \right]^+  + b_t \left[ -v_1 - \sum_{s=1}^t x_{s} + \sum_{s=2}^{t+1} \zeta_s \right]^+ \right) \tag{2a} \label{prob:example3:a} \\
\textnormal{subject to}\quad &0 \le x_{t} \le  p_{t} \quad \forall  t \in [T],  \tag{2b} \label{prob:example3:b}
%\end{aligned}
\end{align}
 where the firm begins at the start of the selling season with an initial inventory  of $v_1$ units of product, and, in each stage $t \in [T]$, the firm can decide to produce an additional $x_{t} \in [0,p_{t}]$ units of product from a single  factory at a cost of $c_{t}$ per unit. The customer demands for the product are denoted by  $\zeta_2,\ldots,\zeta_{T+1} \in \R$, and  the holding and backorder costs for the inventory at the end of each period $t \in [T]$ are given by $h_t$ and $b_t$. 

A fundamental result for this class of dynamic newsvendor problems  is that if Assumption~\ref{ass:1} holds, then the optimal objective values for \eqref{prob:main} and \eqref{prob:ldr} with  cost function~\eqref{prob:example3:a}-\eqref{prob:example3:b} are equal, that is, linear decision rules are optimal control policies  \cite[Theorem 3.1]{bertsimas2010optimality}. However, with the exception of guarantees on whether of the parameters of optimal linear decision rules are positive or negative (\cite[Proposition 5.1]{bertsimas2010optimality}, \cite[Lemma 5]{iancu2013supermodularity}), the structure of optimal linear decision rules for the production quantities in dynamic newsvendor problems has been unknown. Using our proof techniques from  \S\ref{sec:main_result}, we establish  for the first time that optimal linear decision rules for the class of dynamic newsvendor problems are not only optimal for \eqref{prob:main}; they are also sparse.
\begin{theorem}\label{thm:3}
Consider a cost function of the form \eqref{prob:example3:a}-\eqref{prob:example3:b} and let   Assumption~\ref{ass:1} hold. Then there exists an optimal solution $\bar{\by}$ for \eqref{prob:ldr} that satisfies $\| \bar{\by} \|_0 \le 10+12T.$
\end{theorem}

We emphasize that Theorem~\ref{thm:3} establishes that $10+12T$ is an upper bound on the \emph{total} number of nonzero parameters for optimal linear decision rules for the production quantities across \emph{all} of the time periods of the dynamic newsvendor problem. Theorem~\ref{thm:3} thus implies that there always exists an optimal linear decision rule  for the production quantities in the dynamic newsvendor problem in which the \emph{average} number of nonzero parameters for an optimal  linear decision rule in each time period is upper bounded by $(10+12T)/T = 12 + \frac{10}{T}.$ Since optimal linear decision rules are optimal control policies for the dynamic newsvendor problem  \cite[Theorem 3.1]{bertsimas2010optimality}, Theorem~\ref{thm:3} thus establishes for the first time  that the optimal production quantity decision in each time period can be made using an average of $12 + \frac{10}{T} = \mathcal{O}(1)$ of the demand realizations observed in the past time periods.\looseness=-1

\paragraph{Inventory Management with Non-Box  Uncertainty Sets.}

The previous results show that there exist sparse optimal linear decision rules in production-inventory problems with box uncertainty sets. Here we showcase how our proof techniques can be extended to production-inventory problems with non-box uncertainty sets. Specifically, our results here focus on the class of production-inventory problems from Ben-Tal et al.~\cite{ben2004adjustable} in which there are no lead times ($\delta_e = 0$ for all $e \in [E]$) and the demand is chosen from a non-separable budget uncertainty set of the form
\begin{align}\label{eq:newU}
  \mathcal{U} \triangleq \left \{\bzeta \equiv (\zeta_1,\ldots,\zeta_{T+1}): \quad \begin{aligned}
      &\zeta_1 = 1 \\
      &\ubar{D}_s \le \zeta_s \le \bar{D}_s & \forall s \in \{2,\ldots,T+1\}\\
      &\sum_{s=2}^{T+1} \frac{\zeta_s - \ubar{D}_s}{\bar{D}_s - \ubar{D}_s} \le T-k
  \end{aligned} \right \}. 
\end{align}
The difference compared with the original model is that the budget uncertainty  is no longer separable over time period $t$ due to the budget constraint $\sum_{s=2}^{T+1} \frac{\zeta_s - \ubar{D}_s}{\bar{D}_s - \ubar{D}_s} \le T-k$. This budget constraint means that among the $T$ demand variables $\zeta_t$, there are at most $T-k$ demands that choose their maximum values. Now we consider the linear decision rule approximation of the class of production-inventory problems with budget uncertainty set. % decision rule:
\begin{equation} \label{prob:ldr_4} \tag{\textnormal{LDR-4}}
    \begin{aligned}
   & \underset{\substack{\by_{t,1},\ldots,\by_{t,t} \in \R^E:\;  \forall t \in [T]}}{\textnormal{minimize}}&& \max_{\bzeta \equiv (\zeta_1,\ldots,\zeta_{T+1}) \in \mathcal{U} } \left \{ \sum_{t=1}^T \sum_{e=1}^E c_{te} \left( \sum_{s=1}^t    y_{t,s,e} \zeta_s\right) \right \} \\
    &\textnormal{subject to}&& \sum_{t=1}^T  \left( \sum_{s=1}^t y_{t,s,e} \zeta_s \right) \le  Q_e &&  \forall e \in[E] \\
   &&& 0 \le \left( \sum_{s=1}^t y_{t,s,e} \zeta_s \right) \le  p_{te} && \forall  e\in[E],\; t \in [T]\\
   &&&V_{\textnormal{min}} \le  v_1 + \sum_{\ell=1}^t \sum_{e=1}^E \left( \sum_{s=1}^\ell y_{\ell,s,e} \zeta_s \right)  - \sum_{s=2}^{t+1} \zeta_s  \le  V_{\textnormal{max}}&& \forall  t \in [T]\\
   &&& \quad \forall \bzeta \equiv (\zeta_1,\ldots,\zeta_{T+1})\in \mathcal{U}.
    \end{aligned}
  \end{equation}
In the following theorem, we show that our sparsity guarantees can be extended to optimal linear decision rules for the optimization problem~\eqref{prob:ldr_4}. More broadly, the following theorem establishes that our proof techniques for establishing sparsity guarantees for optimal linear decision rules can be extended to applications in which the uncertainty sets are dependent across time periods. 

\begin{theorem} \label{thm:nonbox}
If \eqref{prob:ldr_4} is feasible and has a finite optimal objective value, then there exists an optimal solution $\bar{\by}$ for \eqref{prob:ldr_4} that satisfies $\| \bar{\by} \|_0 \le  2 + 8E + 10T + 6ET + 12kT.$
\end{theorem}
The proof of Theorem \ref{thm:nonbox} can be found in Appendix~\ref{sec:proof_4}. The basic idea of the proof technique is to reduce the problem with a non-box uncertainty set to one with a box uncertainty set by dualizing the coupling constraint (i.e., the budget constraint for this example). We then  utilize the proof framework presented in \S\ref{sec:main_result}.

We emphasize that Theorem \ref{thm:nonbox} only establishes the existence of a sparse solution of \eqref{prob:ldr_4} in the case when $k$ is small. It is an open question whether sparsity guarantees hold when $k$ is large relative to $T$, as is typically recommended in the literature. 
Nevertheless, we hope the proof of Theorem \ref{thm:nonbox} can inspire the community to develop more examples of robust optimization problems with many time periods where the existence of sparse optimal linear decision rules is guaranteed.
  
\section{Conclusion and Future  Research} \label{sec:conclusion}

In this paper, we proved that there  exist sparse optimal linear decision rules for the-0 class of production-inventory problems from Ben-Tal et al.~\cite{ben2004adjustable}, and that the number of nonzero parameters in sparse optimal linear decision rules grows linearly in the number of time periods. 
Apart from their theoretical significance, our sparsity guarantees were shown to give rise to new practical algorithms for computing linear decision rules in robust optimization problems with huge numbers of time periods that were  out-of-reach by previous state-of-the-art algorithms. 
More broadly, we hope that this paper offers new motivations to firms for using linear decision rules in real-world large-scale applications, as well as opens up new research directions at the interface  between the practice and theory of robust optimization. These research directions include harnessing the structure of sparse optimal linear decision rules to design  effective algorithms for computing optimal linear decision rules in real-world large-scale applications, studying the tradeoffs between interpretability of linear decision rules and Pareto efficiency \cite{iancu2014pareto,bertsimas2020pareto}, analyzing the implications of sparsity of linear decision rules on time inconsistency in risk-averse planning problems, and investigating whether sparsity can be combined in Fourier-Motzkin elimination to obtain high-accuracy approximations in robust optimization problems with many recourse decisions \cite{zhen2018adjustable}.

\bibliographystyle{plainnat}
\bibliography{robust_ldr} 
%% Here starts the e-companion (EC)
%%%%%%%%%%%%%%%%%%%%%%%%%%%%%%%%%%%%%%%%%%%%%%%%%%%

\ECSwitch

%\ECDisclaimer
%%%%%%%%%%%%%%%%%%%%%%%%%%%%%%%%%%%%%%%%%%%%%%%%%%%%%%%%%%
%\OneAndAHalfSpacedXI

\AppendixTitle{Technical Proofs and Additional Results}
\setlength{\parskip}{0em}

%%% Main head for the e-companion
\begin{APPENDICES}

\section{Proofs of Lemmas~\ref{lem:1}, \ref{lem:2}, and \ref{lem:zeros}} \label{appx:3}
Throughout the proofs in Appendix~\ref{appx:3}, we will denote the set of feasible solutions to \eqref{prob:ldr_b} by 
\begin{align}\label{set}
\mathcal{Y} \triangleq \left \{ (\by,c_0): \;  \max_{\zeta_1 \in \mathcal{U}_1, \ldots,\zeta_T \in \mathcal{U}_T}  \left \{ \sum_{t=1}^T \ba_{i,t}^\intercal  \left(\sum_{s=1}^t \by_{t,s} \zeta_s \right) - \sum_{t=1}^T b_{i,t}   \zeta_t \right \}  \le c_i  \;\; \forall i \in \{0,\ldots,m\} \right \}.
\end{align}
We remark that it follows readily from Assumption~\ref{ass:1} that $\mathcal{Y}$ is a nonempty convex polyhedron. 
\begin{proof}{Proof of Lemma~\ref{lem:1}.}

Our proof of Lemma~\ref{lem:1} is based on contradiction. Indeed, for the sake of developing a contradiction, suppose that the set  $\mathcal{Y}$ does not have an extreme point. Since the set $\mathcal{Y}$ is a polyhedron, it follows from the supposition that $\mathcal{Y}$ has no extreme points that $\mathcal{Y}$ must contain  a line. In other words, there must exist a solution $(\by^0,c^0_0) \in \mathcal{Y}$  and a direction $(\bd, r) \neq (\bzero,0)$ such that $(\by^0,c_0^0)+\alpha (\bd,r) \in \mathcal{Y}$ for all $\alpha\in\R$. 

We begin by showing under the supposition that $\mathcal{Y}$ has no extreme points that $\bd \neq \bzero$. Indeed, we recall from Assumption~\ref{ass:1} that the optimal objective value of \eqref{prob:ldr} is finite. Moreover, we recall that $(\by^0,c_0^0)+\alpha (\bd,r) \in \mathcal{Y}$ for all $\alpha \in \R$. Since the objective value of $(\by^0,c_0^0)+\alpha (\bd,r)$   is equal to $c_0^0 + \alpha r$, it must be the case that $r = 0$. Therefore, it follows from the fact that $(\bd,r) \neq (\bzero,0)$ that $\bd \neq \bzero$.  

Next, it follows from Assumption~\ref{ass:2} that $(\by^0,c_0^0)+\alpha (\bd,r) \in \mathcal{Y}$ for all $\alpha\in\R$ must satisfy
\begin{align*}
    \sum_{s=1}^t \left( \by^0_{t,s} + \alpha \bd_{t,s}\right) \zeta_s \ge \bzero  \quad \quad \forall \alpha \in \R, \; t \in [T], \; \zeta_1 \in \mathcal{U}_1,\ldots,\zeta_T \in \mathcal{U}_T. 
\end{align*}
Since the above inequalities hold for all $\alpha \in \R$, it follows from algebra that
\begin{align}
        \sum_{s=1}^t  \bd_{t,s} \zeta_s = \bzero  \quad \quad \forall t \in [T], \; \zeta_1 \in \mathcal{U}_1,\ldots,\zeta_T \in \mathcal{U}_T.  \label{line:bradisbadatcomingupwithnames}
\end{align}
Recall that the uncertainty sets are intervals of the form $\mathcal{U}_1 \triangleq [\ubar{D}_1,\bar{D}_1],\ldots,\mathcal{U}_{T} \triangleq [\ubar{D}_{T}, \bar{D}_{T}]$, where $\ubar{D}_1 = \bar{D}_1 = 1$ and $\ubar{D}_t < \bar{D}_t$ for all $t \in \{2,\ldots,T\}$. Therefore, we observe  that the equalities on line~\eqref{line:bradisbadatcomingupwithnames} imply that the equality $\bd_{t,s} = \bzero$ must hold for all $s \in \{2,\ldots,T\}$ and $t \in \{s,\ldots,T\}$. Moreover, it follows from line~\eqref{line:bradisbadatcomingupwithnames}, from  the fact that $\ubar{D}_1 = \bar{D}_1 = 1$, and from the fact that $\bd_{t,s} = \bzero$ for all $s \in \{2,\ldots,T\}$ and $t \in \{s,\ldots,T\}$ that the equality $\bd_{t,1}=- \sum_{s=2}^t \bd_{t,s} \zeta_s = \bzero$ must hold for all $t \in \{1,\ldots,T\}$. We have thus shown that $\bd = \bzero$, which contradicts the supposition that the set of optimal solutions $\mathcal{Y}$ has a line. This concludes our proof that $\mathcal{Y}$ has at least one extreme point. \halmos
\end{proof} 

\begin{proof}{Proof of Lemma~\ref{lem:2}.}
Let $(\bar{\by},\bar{c}_0)$ denote an extreme point of  the set $\mathcal{Y}$ for \eqref{prob:ldr_b}. We first observe from the definitions of the uncertainty sets $\mathcal{U}_1 = [\ubar{D}_1,\bar{D}_1],\ldots,\mathcal{U}_T=[\ubar{D}_T,\bar{D}_T]$ and from algebra that $\mathcal{Y}$ can be written equivalently as
\begin{align*}%\label{set}
\mathcal{Y} &= \left \{ (\by,c_0): \;  \max_{\zeta_1 \in \mathcal{U}_1, \ldots,\zeta_T \in \mathcal{U}_T}  \left \{ \sum_{t=1}^T \ba_{i,t}^\intercal  \left(\sum_{s=1}^t \by_{t,s} \zeta_s \right) - \sum_{t=1}^T b_{i,t}   \zeta_t \right \}  \le c_i  \;\; \forall i \in \{0,\ldots,m\} \right \} \\
&=\left \{ (\by,c_0): \;  \max_{\zeta_1 \in \mathcal{U}_1, \ldots,\zeta_T \in \mathcal{U}_T}  \left \{ \sum_{s=1}^T\left( -b_{i,s} + \sum_{t=s}^T \ba_{i,t}^\intercal \by_{t,s}  \right) \zeta_s \right \}  \le c_i  \;\; \forall i \in \{0,\ldots,m\} \right \} \\
&=\left \{ (\by,c_0): \;  \sum_{s=1}^T  \max \left \{ \left( - b_{i,s}+  \sum_{t=s}^T \ba_{i,t}^\intercal \by_{t,s} \right) \ubar{D}_s,\;  \left( - b_{i,s}+  \sum_{t=s}^T \ba_{i,t}^\intercal \by_{t,s} \right) \bar{D}_s \right \} \le c_i   \;\; \forall i \in \{0,\ldots,m\} \right \} \\
&=\left \{ (\by,c_0): \;  \sum_{s=1}^T  \left( - b_{i,s}+  \sum_{t=s}^T \ba_{i,t}^\intercal \by_{t,s} \right) {D}_s^* \le c_i \; \; \forall D^*_s \in \left\{\ubar{D}_s, \bar{D}_s  \right \},\; s \in [T],  \textnormal{ and }i \in \{0,\ldots,m\} \right \}.
\end{align*}
Since $(\bar{\by},\bar{c}_0)$ is an extreme point of the set  $\mathcal{Y}$, and since the set $\mathcal{Y}$ is a polyhedron, we observe that $(\bar{\by},\bar{c}_0)$ is a basic feasible solution of $\mathcal{Y}$. In other words, $(\bar{\by}, \bar{c}_0)$ must be the unique solution to the system of constraints in $\mathcal{Y}$ which are active constraints at $(\bar{\by},\bar{c}_0)$. Let $\mathcal{I}$ denote the subset of $\{0,\ldots,m\}$ corresponding to the active constraints at the basic feasible solution $(\bar{\by},\bar{c}_0)$, that is, let 
\begin{align*}
\mathcal{I} \triangleq \{0\} \cup \left \{ i:  \begin{aligned}
 \textnormal{there exists } D^*_s \in \left\{\ubar{D}_s, \bar{D}_s \right \} \textnormal{ for each stage  }s \in [T] \textnormal{ such that } \sum_{s=1}^T  \left( - b_{i,s}+  \sum_{t=s}^T \ba_{i,t}^\intercal \bar{\by}_{t,s} \right) D^*_s  = c_i \end{aligned}
   \right \}.
\end{align*}
Since $(\bar{\by},\bar{c}_0)$ is an element of $\mathcal{Y}$, we observe for each $i \in \mathcal{I}$ that the equality
\begin{align*}
 \sum_{s=1}^T  \left( - b_{i,s}+  \sum_{t=s}^T \ba_{i,t}^\intercal \bar{\by}_{t,s} \right) D^*_s  = \begin{cases}c_i,&\text{if } i \neq 0,\\
 \bar{c}_0,&\text{if } i = 0
 \end{cases}
 \end{align*}
is satisfied whenever the following equality holds for each $s \in [T]$.
 \begin{align*}
 D^*_s = \begin{cases}
 \bar{D}_s,&\text{if } -b_{i,s} + \sum_{t=s}^T \ba_{i,t}^\intercal \bar{\by}_{t,s} > 0,\\
  \ubar{D}_s,&\text{if } -b_{i,s} + \sum_{t=s}^T \ba_{i,t}^\intercal \bar{\by}_{t,s} < 0,\\
 \ubar{D}_s \textnormal{ or } \bar{D}_s &\text{if }  -b_{i,s} + \sum_{t=s}^T \ba_{i,t}^\intercal \bar{\by}_{t,s} = 0.
 \end{cases}
 \end{align*}
Therefore, we conclude from the above observation that the set of active constraints at the basic feasible solution $(\bar{\by},\bar{c}_0)$ is given by the system of equalities
  \begin{align}
  &\sum_{s \in \mathcal{T}^>_i} \left( - b_{i,s}+  \sum_{t=s}^T \ba_{i,t}^\intercal \by_{t,s} \right) \bar{D}_s  +  \sum_{s \in \mathcal{T}^<_i} \left( - b_{i,s}+  \sum_{t=s}^T \ba_{i,t}^\intercal \by_{t,s} \right) \ubar{D}_s + \sum_{s \in \mathcal{T}^=_i} \left( - b_{i,s}+  \sum_{t=s}^T \ba_{i,t}^\intercal \by_{t,s} \right) {D}_s^*  = c_i \notag \\
 & \quad \quad \forall i \in \mathcal{I} \textnormal{ and }  D_s^* \in \left\{\ubar{D}_s,\bar{D}_s \right \} \forall s \in \mathcal{T}^=_i,\label{line:bind}
  \end{align}
where we define the disjoint index sets $\mathcal{T}^>_i$, $\mathcal{T}^=_i$, $\mathcal{T}^<_i$ for each $i \in \mathcal{I}$ as 
\begin{gather*}
\mathcal{T}^>_i \triangleq \left \{ s: \; -b_{i,s} + \sum_{t=s}^T \ba_{i,t}^\intercal \bar{\by}_{t,s} > 0  \right \},\quad  \mathcal{T}^=_i \triangleq \left \{ s: \; -b_{i,s} + \sum_{t=s}^T \ba_{i,t}^\intercal \bar{\by}_{t,s} = 0  \right \},\\
  \mathcal{T}^<_i \triangleq \left \{ s: \; -b_{i,s} + \sum_{t=s}^T \ba_{i,t}^\intercal \bar{\by}_{t,s} < 0  \right \}.
  \end{gather*}
We observe that the system of equalities~\eqref{line:bind} can be rewritten as 
 \begin{align*}
 \sum_{s \in \mathcal{T}^>_i} \left( - b_{0,s}+  \sum_{t=s}^T \ba_{0,t}^\intercal \by_{t,s} \right) \bar{D}_s  +  \sum_{s \in \mathcal{T}^<_i} \left( - b_{0,s}+  \sum_{t=s}^T \ba_{0,t}^\intercal \by_{t,s} \right) \ubar{D}_s  &= c_i &&\forall i \in \mathcal{I} \\
  \sum_{t=s}^T \ba_{i,t}^\intercal \by_{t,s} &=  b_{i,s}  && \forall i \in \mathcal{I}, s \in \mathcal{T}^=_i,
 \end{align*}
 which concludes our proof of Lemma~\ref{lem:2}.  \halmos \end{proof}
 
\begin{proof}{Proof of Lemma~\ref{lem:zeros}. }
The overarching idea of the proof of  Lemma~\ref{lem:zeros} is that the equations $\bbp_2 \bz=\textbf{0}$ from \eqref{line:system_lemma:2} can be eliminated by eliminating the corresponding variables. 

We begin by making several assumptions without loss of generality.  First, we recall that $\bbp_2$ is a matrix with more columns than rows. We henceforth assume without loss of generality that $\bbp_2$ has linearly independent rows  (otherwise we may remove the rows which are linearly dependent without changing the set of feasible solutions to the system of equations~\eqref{line:system_lemma:1}-\eqref{line:system_lemma:3}). We also assume for each row $i \in [m_2]$ of the matrix $\bbp_2$ that there exists a column $j \in [n] \setminus \mathcal{N}$ such that $p_{2,i,j} \neq 0$ (otherwise we may remove row $i$  without changing the set of feasible solutions to the system of equations~\eqref{line:system_lemma:1}-\eqref{line:system_lemma:3}).

  We will also use the following notation. For each row $i \in [m_2]$ of the matrix $\bbp_2$, let $j_i \in [n] \setminus \mathcal{N}$ denote the smallest column for which $p_{2,i,j} \neq 0$, and let $\mathcal{S}_i \triangleq \left \{ j \in [n] \setminus \{j_i\}: \; p_{2,i,j} \neq 0 \right \}$ denote the remaining columns for which the $i$-th row of $\bbp_2$ has nonzero entries. Finally, let $\mathcal{S} \triangleq \{ j_1,\ldots,j_{m_2}\}$, in which case it follows from the assumption that $\sum_{i=1}^{m_2} \mathbb{I} \left \{ p_{2,i,j} \neq 0 \right \} \le 1$ for each $j \in [n]$ that $\mathcal{S},\mathcal{S}_1,\ldots,\mathcal{S}_{m_2}$ are disjoint sets. 

We now perform a substitution of variables to eliminate the constraints~\eqref{line:system_lemma:2} from the system of equations. Specifically, by performing the  substitution $z_{j_i} = - \sum_{j \in \mathcal{S}_i} z_{j}$ for each $i \in [m_2]$, we conclude that $(\bar{z}_j: \; j \in [n] \setminus \mathcal{S})$ is the unique solution to the following system of equations. %equalities:
\begin{align*}
\sum_{j \in [n] \setminus \mathcal{S}} \bbp_{1,j} z_j + \sum_{i=1}^{m_2} \bbp_{1,j_i} \left(- \sum_{j \in \mathcal{S}_i} z_{j} \right) &= \bq\\
z_j &= 0 \quad \forall j \in \mathcal{N}. 
\end{align*}
Since $(\bar{z}_j: \; j \in [n] \setminus \mathcal{S})$ is the unique solution to the above system, and since $\bq \in \R^{m_1}$, we conclude that $(\bar{z}_j: \; j \in [n] \setminus \mathcal{S})$ must have at most $m_1$ nonzero entries. Moreover, it follows from the fact that $\mathcal{S},\mathcal{S}_1,\ldots,\mathcal{S}_{m_2}$ are disjoint sets and the fact that  $z_{j_i} = - \sum_{j \in \mathcal{S}_i} z_{j}$ for each $i \in [m_2]$ that at most $m_1$ of the entries $z_{j_1},\ldots,z_{j_{m_2}}$ are nonzero. This concludes our proof that $\| \bar{\bz} \|_0 \le 2 m_1$. 
\halmos
\end{proof}

\section{Proof of Theorem~\ref{thm:main}} \label{appx:proof_thm12}

\begin{proof}{Proof of Theorem~\ref{thm:main}. } Our proof is split into three steps, corresponding to Lemmas~\ref{lem:1}, \ref{lem:2}, and \ref{lem:zeros}. 

\vspace{0.5em}

\textbf{Step 1}: 
\noindent  We begin in the first step of our proof of Theorem~\ref{thm:main} by applying Lemma~\ref{lem:1} for the epigraph formulation of \eqref{prob:ldr} for cost functions given by \eqref{prob:example1:a}-\eqref{prob:example1:d}. To do this, we first notice that since the minimal lead time across the factories is $\delta \triangleq\min_{e\in [E]} \delta_e$  and since $c_{te}>0$, it must be the case that $x_{te}=0$ for every $t\ge T-\delta+1$ and $e\in [E]$ at every optimal solution. 
Next, we rewrite the cost function from  lines~\eqref{prob:example1:a}-\eqref{prob:example1:d} to match the format  used by Lemmas~\ref{lem:1} and \ref{lem:2}. Indeed, we recall that Lemmas~\ref{lem:1} and \ref{lem:2} involve cost functions in which the number of stages with decisions is equal to the number of stages with uncertain variables. Thus, by introducing dummy decision variables $\bx_{T+1}=0 \in \R^E$,  we  observe that the cost function on lines~\eqref{prob:example1:a}-\eqref{prob:example1:d} can be equivalently written  as 
\begin{align} 
%\begin{aligned}
C(\bx_1,\ldots,\bx_{T+1}, \zeta_1,\ldots,\zeta_{T+1}) =  &  \sum_{e=1}^E \sum_{t=1}^{T+1-\delta}  c_{te} x_{te}\tag{1a} \label{prob:example1:a} \\
\textnormal{subject to}\quad &\sum_{t=1}^{T+1-\delta} x_{te} \le  Q_e &&  \forall e  \in [E] \tag{1b} \label{prob:example1:b}\\
&0 \le x_{te} \le  p_{te} && \forall  e \in [E],  t \in [T +1-\delta] \tag{1c} \label{prob:example1:c}\\
&V_{\textnormal{min}} \le v_1 + \sum_{e=1}^E \sum_{\ell=1}^{t-\delta_e}   x_{\ell e} - \sum_{s=2}^{t+1} \zeta_s \le  V_{\textnormal{max}} && \forall  t \in [T],\tag{1d} \label{prob:example1:d}
%\end{aligned}
\end{align}
where we define $p_{T+1,e}\equiv 0$ and $c_{T+1,e} \equiv 0$ for each factory $e \in [E]$.  After adding the dummy decision variables $\bx_{T+1}$, we observe that the above cost function matches the format used by Lemmas~\ref{lem:1} and \ref{lem:2}.

Before proceeding onward, let us make two brief remarks about the convention used in the above cost function. First, we remark that the constraint~\eqref{prob:example1:c} implies that the dummy decision variables $\bx_{T+1}$ will always be equal to zero. Second, we remark that the decisions $x_{te}$ for each $t \ge T-\delta_e + 1$ are unnecessary in the above formulation, in the sense that these decisions  could always be set to zero without loss of generality. To see why the decisions are unnecessary, we notice that the decision  $x_{te}$ for each $t \ge T-\delta_e + 1$ does not appear in constraint~\eqref{prob:example1:d}; therefore, since $c_{te} \ge 0$, we observe that $x_{te}$ can at optimality take its minimum value that is allowed by constraints~\eqref{prob:example1:b} and \eqref{prob:example1:c}.  Although they are not necessary, the decisions  $x_{te}$ for each $t \ge T-\delta_e + 1$ are included in the above cost function to simplify the notation in the rest of the proof. % analysis. 

We next derive the optimization problem~\eqref{prob:ldr} corresponding to the above cost function~\eqref{prob:example1:a}-\eqref{prob:example1:d}. 
For the sake of clarity, we will show each of the intermediary algebra steps in this derivation. 
Indeed, we first observe from algebra that  line~\eqref{prob:example1:a}  can be written with linear decision rules as
\begin{align*}
  \sum_{e=1}^E \sum_{t=1}^{T+1-\delta}  c_{te} x_{te} &=  \sum_{e=1}^E \sum_{t=1}^{T+1-\delta}  c_{te} \left( \sum_{s=1}^t y_{t,s,e} \zeta_s \right)= \sum_{s=1}^{T +1-\delta}  \left( \sum_{t=s}^{T+1-\delta} \sum_{e=1}^E c_{te} y_{t,s,e} \right)  \zeta_s,
\end{align*}
where the first equality comes from using linear decision rules and the second equality follows from algebra. 
We observe that the left-hand side of line~\eqref{prob:example1:b} for each $e \in [E]$  can be written   with linear decision rules  as
\begin{align*}
   \sum_{t=1}^{T+1-\delta} x_{te} =  \sum_{t=1}^{T+1-\delta}  \left( \sum_{s=1}^t y_{t,s,e} \zeta_s \right) &= \sum_{s=1}^{T+1-\delta} \left( \sum_{t=s}^{T+1-\delta} y_{t,s,e} \right) \zeta_s, 
\end{align*}
where the first equality comes from using linear decision rules and the second equality follows from algebra. 
We observe that the decision in line~\eqref{prob:example1:c} for each $e \in [E]$ and $t \in [T+1-\delta]$ can be written with  linear decision rules as
\begin{align*}
    x_{te} = \sum_{s=1}^{t}  y_{t,s,e} \zeta_s + \sum_{s=t+1}^{T+1} 0 \zeta_s,
\end{align*}
where the equality comes from using linear decision rules. Finally, we observe that the inventory in line~\eqref{prob:example1:d} for each $t \in [T]$ can be written with linear decision rules as
\begin{align*}
    &v_1 + \sum_{e=1}^E \sum_{\ell=1}^{t-\delta_e}   x_{\ell e} - \sum_{s=2}^{t+1} \zeta_s \\
    &= v_1 + \sum_{e=1}^E \sum_{\ell=1}^{t-\delta_e}   \left( \sum_{s=1}^\ell y_{\ell,s,e} \zeta_s \right) - \sum_{s=2}^{t+1} \zeta_s\\
    &= v_1 + \sum_{s=1}^t \left( \sum_{\ell=s}^t \sum_{e \in [E]: \delta_e \le t - \ell} y_{\ell,s,e} \right)\zeta_s - \sum_{s=2}^{t+1} \zeta_s\\
    &= v_1 + \left( \sum_{\ell=1}^t \sum_{e \in [E]: \delta_e \le t - \ell} y_{\ell,1,e} \right) \zeta_1  + \sum_{s=2}^t \left(-1 + \sum_{\ell=s}^t \sum_{e \in [E]: \delta_e \le t - \ell} y_{\ell,s,e} \right) \zeta_s - \zeta_{t+1} + \sum_{s=t+2}^{T+1} 0 \zeta_s,
\end{align*}
where the first equality comes from using linear decision rules and the second and third equalities follow from algebra. Combining the above steps, we  conclude that the epigraph form of the optimization problem~\eqref{prob:ldr} with cost function~\eqref{prob:example1:a}-\eqref{prob:example1:d} is equivalent to 
\begin{equation} \label{prob:ldr_1} \tag{\textnormal{LDR-1}}
    \begin{aligned}
   & \underset{\substack{c_0 \in \R \\\by_{t,1},\ldots,\by_{t,t} \in \R^E:\;  \forall t \in [T+1]}}{\textnormal{minimize}}&&c_0\\ 
    &\textnormal{subject to}&& \sum_{s=1}^{T +1-\delta}  \left( \sum_{t=s}^{T+1-\delta} \sum_{e=1}^E c_{te} y_{t,s,e} \right)  \zeta_s \le c_0 \\
    &&&\sum_{s=1}^{T+1-\delta} \left( \sum_{t=s}^{T+1-\delta} y_{t,s,e} \right) \zeta_s \le  Q_e \quad \quad \quad   \forall e \in[E] \\
   &&& 0 \le  \sum_{s=1}^{t}  y_{t,s,e} \zeta_s + \sum_{s=t+1}^{T+1} 0 \zeta_s \le  p_{te} \quad \forall  e\in[E],\; t \in [T+1-\delta]\\
   &&&V_{\textnormal{min}} \le  v_1 + \left( \sum_{\ell=1}^t \sum_{e \in [E]: \delta_e \le t - \ell} y_{\ell,1,e} \right) \zeta_1 \\
   &&& \quad \quad \quad \quad + \sum_{s=2}^t \left(-1 + \sum_{\ell=s}^t \sum_{e \in [E]: \delta_e \le t - \ell} y_{\ell,s,e} \right) \zeta_s - \zeta_{t+1} + \sum_{s=t+2}^{T+1} 0 \zeta_s  \le  V_{\textnormal{max}}\quad  \forall  t \in [T]\\
   &&& \quad \forall \zeta_1 \in \mathcal{U}_1,\ldots,\zeta_{T+1} \in \mathcal{U}_{T+1}.
    \end{aligned}
  \end{equation} 
With the above notation, we are now ready to invoke Lemma~\ref{lem:1}. Indeed, we recall from the statement of Theorem~\ref{thm:main} that Assumption~\ref{ass:1} holds for cost function~\eqref{prob:example1:a}-\eqref{prob:example1:d}. {Furthermore, the constraint in \eqref{prob:example1:c} guarantees Assumption \ref{ass:2} holds.} Therefore, it follows from Lemma~\ref{lem:1} that the set of feasible solutions to \eqref{prob:ldr_1} is a nonempty polyhedron with at least one extreme point.  

\vspace{0.5em}
\textbf{Step 2}:   In the second step of our proof of Theorem~\ref{thm:main}, we use Lemma~\ref{lem:2} to characterize the structure of extreme points for the feasible set of \eqref{prob:ldr_1}. Indeed, let $(\bar{\by},\bar{c}_0)$ denote an extreme point of the set of feasible solutions of \eqref{prob:ldr_1}. Since Assumption~\ref{ass:1} holds, 
   it follows from Lemma~\ref{lem:2} that there exists 
\begin{itemize}
\item
index sets $\mathcal{I}^{\ref{prob:example1:b},\bar{\by}} \subseteq [E]$,   $\ubar{\mathcal{I}}^{\ref{prob:example1:c},\bar{\by}},\bar{\mathcal{I}}^{\ref{prob:example1:c},\bar{\by}} \subseteq [T+1-\delta]\times[E]$, and  $\ubar{\mathcal{I}}^{\ref{prob:example1:d},\bar{\by}},\bar{\mathcal{I}}^{\ref{prob:example1:d},\bar{\by}}  \subseteq [T]$; 
\item index sets $\mathcal{T}^{\ref{prob:example1:a},\bar{\by}} \subseteq [T+1-\delta]$, $\mathcal{T}^{\ref{prob:example1:b},\bar{\by}}_e \subseteq [T+1-\delta]$ for each $e \in \mathcal{I}^{\ref{prob:example1:b},\bar{\by}}$,  $\ubar{\mathcal{T}}^{\ref{prob:example1:c},\bar{\by}}_{t,e} \subseteq [T+1-\delta]$ for each $(t,e) \in \ubar{\mathcal{I}}^{\ref{prob:example1:c},\bar{\by}}$, 
$\bar{\mathcal{T}}^{\ref{prob:example1:c},\bar{\by}}_{t,e} \subseteq [T+1]$  for each $(t,e) \in \bar{\mathcal{I}}^{\ref{prob:example1:c},\bar{\by}}$, 
$\ubar{\mathcal{T}}^{\ref{prob:example1:d},\bar{\by}}_{t} \subseteq [T+1]$ for each $t \in \ubar{\mathcal{I}}^{\ref{prob:example1:d},\bar{\by}}$,  and  $\bar{\mathcal{T}}^{\ref{prob:example1:d},\bar{\by}}_{t} \subseteq [T+1]$ for each $t \in \bar{\mathcal{I}}^{\ref{prob:example1:d},\bar{\by}}$;
\item  hyperplanes $(\balpha^{\ref{prob:example1:a},\bar{\by}},\beta^{\ref{prob:example1:a},\bar{\by}})$, $(\balpha^{\ref{prob:example1:b},\bar{\by}}_e,\beta^{\ref{prob:example1:b},\bar{\by}}_e)$ for each $e \in {\mathcal{I}}^{\ref{prob:example1:b},\bar{\by}}$, $(\ubar{\balpha}^{\ref{prob:example1:c},\bar{\by}}_{t,e},\ubar{\beta}^{\ref{prob:example1:c},\bar{\by}}_{t,e})$ for each $(t,e) \in \ubar{\mathcal{I}}^{\ref{prob:example1:c},\bar{\by}}$, 
  $(\bar{\balpha}^{\ref{prob:example1:c},\bar{\by}}_{t,e},\bar{\beta}^{\ref{prob:example1:c},\bar{\by}}_{t,e})$ for each $(t,e) \in \bar{\mathcal{I}}^{\ref{prob:example1:c},\bar{\by}}$,  $(\ubar{\balpha}^{\ref{prob:example1:d},\bar{\by}}_{t},\ubar{\beta}^{\ref{prob:example1:d},\bar{\by}}_{t})$ for each $t \in \ubar{\mathcal{I}}^{\ref{prob:example1:d}, \bar{\by}}$, and  $(\bar{\balpha}^{\ref{prob:example1:d},\bar{\by}}_{t},\bar{\beta}^{\ref{prob:example1:d},\bar{\by}}_{t})$ for each $t \in \bar{\mathcal{I}}^{\ref{prob:example1:d}, \bar{\by}}$ 
  \end{itemize}
such that $\bar{\by}$ is the unique solution to  the following system of equations: 
\begin{align}
    %\begin{aligned}
&\balpha^{\ref{prob:example1:a},\bar{\by}} \cdot \by = \beta^{\ref{prob:example1:a},\bar{\by}}  \tag{HARD-1a} \label{line:hard_1a} \\
 &\balpha^{\ref{prob:example1:b},\bar{\by}}_e \cdot \by = \beta^{\ref{prob:example1:b},\bar{\by}}_e && \forall e \in \mathcal{I}^{\ref{prob:example1:b},\bar{\by}}\tag{HARD-1b} \label{line:hard_1b}\\
&\ubar{\balpha}^{\ref{prob:example1:c},\bar{\by}}_{t,e} \cdot \by = \ubar{\beta}^{\ref{prob:example1:c},\bar{\by}}_{t,e} && \forall (t,e) \in  \ubar{\mathcal{I}}^{\ref{prob:example1:c},\bar{\by}}  \tag{HARD-1c-LB} \label{line:hard_1c_lb}\\
&\bar{\balpha}^{\ref{prob:example1:c},\bar{\by}}_{t,e} \cdot \by = \bar{\beta}^{\ref{prob:example1:c},\bar{\by}}_{t,e} && \forall (t,e) \in  \bar{\mathcal{I}}^{\ref{prob:example1:c},\bar{\by}}  \tag{HARD-1c-UB} \label{line:hard_1c_ub}\\
&\ubar{\balpha}^{\ref{prob:example1:d},\bar{\by}}_{t} \cdot \by = \ubar{\beta}^{\ref{prob:example1:d},\bar{\by}}_{t} && \forall t \in  \ubar{\mathcal{I}}^{\ref{prob:example1:d},\bar{\by}}  \tag{HARD-1d-LB} \label{line:hard_1d_lb}\\
&\bar{\balpha}^{\ref{prob:example1:d},\bar{\by}}_{t} \cdot \by = \bar{\beta}^{\ref{prob:example1:d},\bar{\by}}_{t} && \forall t \in  \bar{\mathcal{I}}^{\ref{prob:example1:d},\bar{\by}}  \tag{HARD-1d-UB} \label{line:hard_1d_ub}\\
&\sum_{t=s}^{T+1} \sum_{e=1}^E  c_{te} y_{t,s,e} = 0 && \forall s \in \mathcal{T}^{\ref{prob:example1:a},\bar{\by}}\tag{EASY-1a} \label{line:easy_1a}\\
& \sum_{t=s}^{T+1} y_{t,s,e} =0 && \forall e \in[E], \; s \in \mathcal{T}^{\ref{prob:example1:b},\bar{\by}}_e\tag{EASY-1b} \label{line:easy_1b}\\
&y_{t,s,e} =0 && \forall e \in[E], \; t \in[T+1], \; s \in  \ubar{\mathcal{T}}^{\ref{prob:example1:c},\bar{\by}}_{t,e}  \cup \bar{\mathcal{T}}^{\ref{prob:example1:c},\bar{\by}}_{t,e} \text{ if } s \le t \tag{EASY-1c-i} \label{line:easy_1c_i}\\
&0 =0 && \forall e \in[E], \; t \in[T+1], \; s \in  \ubar{\mathcal{T}}^{\ref{prob:example1:c},\bar{\by}}_{t,e}  \cup \bar{\mathcal{T}}^{\ref{prob:example1:c},\bar{\by}}_{t,e} \text{ if } s \ge t+1 \tag{EASY-1c-ii} \label{line:easy_1c_ii}\\
& \sum_{\ell=s}^t \sum_{e \in [E]: \delta_e \le t - \ell} y_{\ell,s,e}    = 0 && \forall t \in [T], \; s \in \ubar{\mathcal{T}}^{\ref{prob:example1:d},\bar{\by}}_t \cup \bar{\mathcal{T}}^{\ref{prob:example1:d},\bar{\by}}_t \text{ if } s =1\tag{EASY-1d-i} \label{line:easy_1d_i}\\
& \sum_{\ell=s}^t \sum_{e \in [E]: \delta_e \le t - \ell} y_{\ell,s,e}    = 1 && \forall t \in [T], \; s \in \ubar{\mathcal{T}}^{\ref{prob:example1:d},\bar{\by}}_t \cup \bar{\mathcal{T}}^{\ref{prob:example1:d},\bar{\by}}_t \text{ if } s \in \{2,\ldots,t\} \tag{EASY-1d-ii} \label{line:easy_1d_ii}\\
& 0   = 1 && \forall t \in [T], \; s \in \ubar{\mathcal{T}}^{\ref{prob:example1:d},\bar{\by}}_t \cup \bar{\mathcal{T}}^{\ref{prob:example1:d},\bar{\by}}_t \text{ if } s = t+1 \tag{EASY-1d-iii} \label{line:easy_1d_iii}\\
& 0   = 0 && \forall t \in [T], \; s \in \ubar{\mathcal{T}}^{\ref{prob:example1:d},\bar{\by}}_t \cup \bar{\mathcal{T}}^{\ref{prob:example1:d},\bar{\by}}_t \text{ if } s \ge t+2. \tag{EASY-1d-iv} \label{line:easy_1d_iv}
 \end{align}
Let us make two observations about the above system of equations. First, we recall from its construction that there is exactly one solution to the above system, namely, $\bar{\by}$. Therefore, we readily observe that there must be no constraints of the form \eqref{line:easy_1d_iii}, i.e.,  the inequality $s \le t$ must hold for all $t \in [T]$ and all $s \in \ubar{\mathcal{T}}^{\ref{prob:example1:d},\bar{\by}}_t \cup \bar{\mathcal{T}}^{\ref{prob:example1:d},\bar{\by}}_t$. Second, we observe that we can remove the constraints of type \eqref{line:easy_1c_ii} and \eqref{line:easy_1d_iv} without loss of generality. Hence, we will drop the constraints  \eqref{line:easy_1d_iii},  \eqref{line:easy_1c_ii}, and \eqref{line:easy_1d_iv}  in our subsequent analysis.  %Finally, we observe without loss of generality that we can combine 
 \vspace{0.5em}
 
\textbf{Step 3}: In the third step of our proof of Theorem~\ref{thm:main}, we use Lemma~\ref{lem:zeros} to show that every extreme point $(\bar{\by},\bar{c}_0)$ of the feasible set of \eqref{prob:ldr_1} is sparse. 

To motivate our usage of Lemma~\ref{lem:zeros}, let us make several observations about the system of equations corresponding to the extreme point $(\bar{\by},\bar{c}_0)$. First, we observe that there are $\mathcal{O}(ET)$ equations in each of the lines~\eqref{line:hard_1a}, \eqref{line:hard_1b}, \eqref{line:hard_1c_lb},  \eqref{line:hard_1c_ub}, \eqref{line:hard_1d_lb},  \eqref{line:hard_1d_ub}, \eqref{line:easy_1a}, \eqref{line:easy_1b}, and \eqref{line:easy_1d_i}. Second, we recall from our discussion at the end of  Step 2 that the equations in lines~\eqref{line:easy_1c_ii}, \eqref{line:easy_1d_iii} and \eqref{line:easy_1d_iv}  can be dropped without loss of generality. Third, we observe that the equations in line~\eqref{line:easy_1c_i} will ultimately impose sparsity into the solution the system of equations.  Finally, we observe that there are up to $\mathcal{O}(E T^2)$ equations on line~\eqref{line:easy_1d_ii}.

With the goal of applying Lemma~\ref{lem:zeros} to the above system of equations, we perform algebraic manipulations on the equations in \eqref{line:easy_1d_ii}. Indeed, we first define the following index sets:
 \begin{align*}
 \mathscr{S}^{\bar{\by}} &\triangleq \bigcup_{t =1}^T \left( \left( \ubar{\mathcal{T}}^{\ref{prob:example1:d},\bar{\by}}_t  \cup \bar{\mathcal{T}}^{\ref{prob:example1:d},\bar{\by}}_t\right) \cap  \{2,\ldots,t\}\right) ,\\
 \mathscr{T}_s^{\bar{\by}} &{\triangleq \left \{ t \in [T]: s \in \left(\ubar{\mathcal{T}}^{\ref{prob:example1:d}, \bar{\by}}_t \cup \bar{\mathcal{T}}^{\ref{prob:example1:d},\bar{\by}}_t \right) \cap  \{2,\ldots,t\} \right \} \quad  \forall s \in \mathscr{S}^{\bar{\by}}}.
 \end{align*} 
 With the above notation, we readily observe that  \eqref{line:easy_1d_ii} can be written equivalently  as 
 \begin{align*}
 \sum_{\ell=s}^t \sum_{e\in[E]: \delta_e\le t-\ell} y_{\ell,s,e}   = 1 \quad\forall s \in \mathscr{S}^{\bar{\by}}, \; t \in \mathscr{T}^{\bar{\by}}_s. \tag{EASY-1d-ii}
 \end{align*}
For each $s \in \mathscr{S}^{\bar{\by}}$, let the elements of $\mathscr{T}_s^{\bar{\by}}$ be indexed in ascending order by $t_{s,1}^{\bar{\by}} < \ldots < t_{s,| \mathscr{T}_s^{\bar{\by}}|}^{\bar{\by}}$. 
With this notation, we observe that \eqref{line:easy_1d_ii}  can be written equivalently as  the following system of equations. %equivalently as % with the following equivalent constraints:
 \begin{align}
 &\sum_{\ell=s}^{t_{s,1}^{\bar{\by}}} \sum_{e\in[E]: \delta_e\le t_{s,1}^{\bar{\by}}-\ell} y_{\ell,s,e}  = 1 && \forall s \in \mathscr{S}^{\bar{\by}}  \tag{EASY-1d-ii'} \label{line:easy_1d_prime}\\
  &\sum_{\ell=s}^{t_{s,k+1}^{\bar{\by}}} \sum_{e\in [E]: \delta_e\le t_{s,k+1}^{\bar{\by}} -\ell} y_{\ell,s,e}  - \sum_{\ell=s}^{t_{s,k}^{\bar{\by}}} \sum_{e\in [E]: \delta_e\le t_{s,k}^{\bar{\by}} -\ell} y_{\ell,s,e} = 1-1=0 && \forall s \in \mathscr{S}^{\bar{\by}}, \; k \in \left\{1,\ldots,|\mathscr{T}_s^{\bar{\by}}| - 1 \right \}.  \tag{EASY-1d-ii''} \label{line:easy_1d_prime_prime}
 \end{align}
 In particular, it follows from algebra that  \eqref{line:easy_1d_prime_prime} is equivalent to
 \begin{align*}
       &\sum_{\ell=t_{s,k}^{\bar{\by}}+1}^{t_{s,k+1}^{\bar{\by}}} \sum_{e\in [E]: \delta_e\le t_{s,k+1}^{\bar{\by}} -\ell} y_{\ell,s,e}  +  \sum_{\ell=s}^{t_{s,k}^{\bar{\by}}} \sum_{e\in [E]:t_{s,k}^{\bar{\by}} -\ell + 1 \le \delta_e\le t_{s,k+1}^{\bar{\by}} -\ell} y_{\ell,s,e} = 0&& \forall s \in \mathscr{S}^{\bar{\by}}, \; k \in \left\{1,\ldots,|\mathscr{T}_s^{\bar{\by}}| - 1 \right \}.  \tag{EASY-1d-ii''} \label{line:easy_1d_prime_prime}
 \end{align*}

To simplify our notation, we now compactly represent the constraints from lines~\eqref{line:hard_1a}, \eqref{line:hard_1b}, \eqref{line:hard_1c_lb}, \eqref{line:hard_1c_ub} \eqref{line:hard_1d_lb}, \eqref{line:hard_1d_ub}, \eqref{line:easy_1a}, \eqref{line:easy_1b}, \eqref{line:easy_1d_i} and \eqref{line:easy_1d_prime} using the index set $\mathscr{I}^{\bar{\by}}$ and hyperplanes $(\balpha_i^{\bar{\by}}, \beta_i^{\bar{\by}})$ for each $i \in \mathscr{I}^{\bar{\by}}$, where 
 \begin{align*}
| \mathscr{I}^{\bar{\by}}|&= \underbrace{1}_{\eqref{line:hard_1a}} +  \underbrace{\left|\mathcal{I}^{\ref{prob:example1:b},\bar{\by}}\right|}_{\eqref{line:hard_1b}} + \underbrace{\left| \ubar{\mathcal{I}}^{\ref{prob:example1:c},\bar{\by}}   \right| }_{\eqref{line:hard_1c_lb}} + \underbrace{\left| \bar{\mathcal{I}}^{\ref{prob:example1:c},\bar{\by}}   \right| }_{\eqref{line:hard_1c_ub}} +  \underbrace{\left| \ubar{\mathcal{I}}^{\ref{prob:example1:d},\bar{\by}} \right|}_{\eqref{line:hard_1d_lb}}+  \underbrace{\left| \bar{\mathcal{I}}^{\ref{prob:example1:d},\bar{\by}} \right|}_{\eqref{line:hard_1d_ub}} \\
&\quad +\underbrace{\left|  \mathcal{T}^{\ref{prob:example1:a},\bar{\by}} \right| }_{\eqref{line:easy_1a}} + \underbrace{ \sum_{e=1}^E \left| \mathcal{T}^{\ref{prob:example1:b},\bar{\by}}_e \right|}_{\eqref{line:easy_1b}} +  \underbrace{\left| \left \{t \in [T]: \; 1 \in \ubar{\mathcal{T}}^{\ref{prob:example1:d},\bar{\by}}_t \cup \bar{\mathcal{T}}^{\ref{prob:example1:d},\bar{\by}}_t \right \} \right| }_{\eqref{line:easy_1d_i}}+ \underbrace{ \left| \mathscr{S}^{\bar{\by}} \right|}_{\eqref{line:easy_1d_prime}}\\
&\le \underbrace{1}_{\eqref{line:hard_1a}} +  \underbrace{E}_{\eqref{line:hard_1b}} + \underbrace{(T+1-\delta)E}_{\eqref{line:hard_1c_lb}}  + \underbrace{(T+1-\delta)E}_{\eqref{line:hard_1c_ub}} + \underbrace{T}_{\eqref{line:hard_1d_lb}}+ \underbrace{T}_{\eqref{line:hard_1d_ub}}\\
&\quad + \underbrace{T+1}_{\eqref{line:easy_1a}} + \underbrace{(T+1-\delta)E }_{\eqref{line:easy_1b}} +  \underbrace{T }_{\eqref{line:easy_1d_i}} + \underbrace{T-1}_{\eqref{line:easy_1d_prime}}  \\
&= 1 + 4E + 5T + 3ET -3E\delta.
 \end{align*}
With the above notation, we have shown that $\bar{\by}$ is the unique solution to the following system of equations.
\begin{align}
    %\begin{aligned}
&\balpha_i^{\bar{\by}} \cdot \by = \beta_i^{\bar{\by}} \quad  \forall i \in \mathscr{I}^{\bar{\by}} \tag{HARD-combined} \label{line:hard_combined} \\
&\sum_{\ell=t_{s,k}^{\bar{\by}}+1}^{t_{s,k+1}^{\bar{\by}}} \sum_{e\in [E]: \delta_e\le t_{s,k+1}^{\bar{\by}} -\ell} y_{\ell,s,e}  +  \sum_{\ell=s}^{t_{s,k}^{\bar{\by}}} \sum_{e\in [E]:t_{s,k}^{\bar{\by}} -\ell + 1 \le \delta_e\le t_{s,k+1}^{\bar{\by}} -\ell} y_{\ell,s,e} = 0\quad  \forall s \in \mathscr{S}^{\bar{\by}}, \; k \in \left\{1,\ldots,|\mathscr{T}_s^{\bar{\by}}| - 1 \right \}.  \tag{EASY-1d-ii''} \label{line:easy_1d_prime_prime}\\
&y_{t,s,e} =0 \quad \forall e \in[E], \; t \in[T+1], \; s \in  \ubar{\mathcal{T}}^{\ref{prob:example1:c},\bar{\by}}_{t,e}  \cup \bar{\mathcal{T}}^{\ref{prob:example1:c},\bar{\by}}_{t,e} \text{ if } s \le t \tag{EASY-1c-i} \label{line:easy_1c_i}
 \end{align}
{We next apply Lemma~\ref{lem:zeros} to the above system of equations by noticing \eqref{line:hard_combined}, \eqref{line:easy_1d_prime_prime}, and \eqref{line:easy_1c_i} follow the same structure of \eqref{line:system_lemma:1}, \eqref{line:system_lemma:2} and \eqref{line:system_lemma:3}, respectively. Furthermore, we notice that there is no overlapping index of $y$ in \eqref{line:easy_1d_prime_prime}, thus each column of the corresponding $\bbp_2$ has at most one nonzero entry.}
Since $\bar{\by}$ is the unique solution to the above system of equations, it follows from Lemma~\ref{lem:zeros} that the number of nonzero entries in $\bar{\by}$ satisfies
\begin{align*}
\| \bar{\by} \|_0 \le 2 \left| \mathcal{I} \right| &\le 2 \left( 1 + 4E + 5T + 3ET -3E\delta \right),
\end{align*}

Since the above inequality holds for every extreme point $(\bar{\by},\bar{c}_0)$, since we have proven that \eqref{prob:ldr_1} has at least one extreme point, and since we observe from Assumption~\ref{ass:1} that \eqref{prob:ldr_1} has an optimal solution which is an extreme point, our proof of Theorem~\ref{thm:main} is complete. \halmos 
\end{proof}

\section{Proofs of Theorem~\ref{thm:3}}
\begin{proof}{Proof of Theorem~\ref{thm:3}.}
Our proof of Theorem~\ref{thm:3} follows a similar organization to that of Theorem~\ref{thm:main}. Specifically, the proof of Theorem~\ref{thm:main} is split into three steps which correspond to Lemmas~\ref{lem:1}, \ref{lem:2}, and \ref{lem:zeros}. In contrast to the proof of Theorem~\ref{thm:main}, the proof of Theorem~\ref{thm:3} requires the introduction of auxiliary decision rules to account for nonlinear cost functions. 

\vspace{0.5em}
\textbf{Step 1}: We begin in the first step of our proof of Theorem~\ref{thm:3} by applying Lemma~\ref{lem:1} to an epigraph formulation of \eqref{prob:ldr} for cost function~\eqref{prob:example3:a}-\eqref{prob:example3:b}. 
Indeed, we observe that \eqref{prob:main} with cost function \eqref{prob:example3:a}-\eqref{prob:example3:b} can be reformulated as  
\begin{align*}
    \begin{aligned}
&\underset{\bx, \bz}{\textnormal{minimize}}&& \max_{\zeta_1 \in \mathcal{U}_1,\ldots ,\zeta_{T+1} \in \mathcal{U}_{T+1}} \left \{ \sum_{t=1}^{T+1} c_t  x_{t}(\zeta_1,\ldots,\zeta_t) + \sum_{t=1}^{T+1} z_{t}(\zeta_1,\ldots,\zeta_{t}) \right \}  \\
&\textnormal{subject to}&& z_{t+1}(\zeta_1,\ldots,\zeta_{t+1})   \ge h_t \left( v_1 + \sum_{s=1}^t x_{s}(\zeta_1,\ldots,\zeta_s)  - \sum_{s=2}^{t+1} \zeta_s \right)  && \forall t \in [T]\\
&&&z_{t+1}(\zeta_1,\ldots,\zeta_{t+1})   \ge - b_t \left( v_1 + \sum_{s=1}^t x_{s}(\zeta_1,\ldots,\zeta_s)  - \sum_{s=2}^{t+1} \zeta_s \right)  &&  \forall  t \in [T]\\
&&&0 \le x_t(\zeta_1,\ldots,\zeta_t) \le  p_{t} && \forall t\in[T + 1]\\
&&& z_1(\zeta_1) \ge 0 \\
&&& \quad \forall \zeta_1 \in \mathcal{U}_1,\ldots,\zeta_{T+1} \in \mathcal{U}_{T+1},
\end{aligned}
\end{align*}
where the auxiliary decision rule $z_{t+1}(\zeta_1,\ldots,\zeta_{t+1})$ captures the holding and backorder costs in each period $t \in [T]$, $p_{T+1} = 0$, $c_{T+1} = 1$, and  $x_{T+1}(\zeta_1,\ldots,\zeta_{T+1})$ and $z_1(\zeta_1)$ are dummy decision rules that  can always be identically equal to zero at optimality.  
We note that the dummy decision rules $x_{T+1}(\zeta_1,\ldots,\zeta_{T+1})$ and $z_1(\zeta_1)$ have been introduced into the above optimization problem to match the setting of Lemmas~\ref{lem:1} and \ref{lem:2}, in which the number of decisions is constant in each stage (in this case, the number of decisions in each stage is $n = 2$).  We also readily observe from inspection that any feasible solution of the above optimization problem will satisfy  $x_t(\zeta_1,\ldots,\zeta_{t}), z_t(\zeta_1,\ldots,\zeta_t) \ge 0$ for all stages $t \in [T+1]$ and all realizations $\zeta_1 \in \mathcal{U}_1,\ldots,\zeta_{T+1} \in \mathcal{U}_{T+1}$.  

It follows from \cite[Theorem 3.1]{bertsimas2010optimality} and Assumption~\ref{ass:1}  that the auxiliary decision rules  $z_{t}(\zeta_1,\ldots,\zeta_{t})$ and the production decision rules $x_t(\zeta_1,\ldots,\zeta_t)$ for all $t \in [T+1]$ in the above optimization problem can be replaced with linear decision rules without any loss of optimality. Hence,  we conclude that the optimization problems~\eqref{prob:main} and \eqref{prob:ldr} with cost function \eqref{prob:example3:a}-\eqref{prob:example3:b} are equivalent to one another and can be written as  
\begin{align*}%\label{eq:example3}
\begin{aligned}
   & \underset{\substack{y_{t,1},\ldots,y_{t,t} \in \R \; \forall t \in [T+1]\\
    w_{t,1},\ldots,w_{t,t} \in \R \; \forall t \in [T+1]}}{\textnormal{minimize}}\quad &&\max_{\zeta_1 \in \mathcal{U}_1,\ldots,\zeta_{T+1} \in \mathcal{U}_{T+1}} \left \{ \sum_{t=1}^{T+1} c_t  \left( \sum_{s=1}^t y_{t,s} \zeta_s \right) + \sum_{t=1}^{T+1} \left( \sum_{s=1}^{t} w_{t,s} \zeta_s \right) \right \} % \tag{1a} \label{prob:example1:a} 
\\
&\textnormal{subject to}&&  \sum_{s=1}^{t+1} w_{t+1,s} \zeta_s \ge h_t \left( v_1 + \sum_{\ell=1}^t   \left( \sum_{s=1}^\ell y_{\ell,s} \zeta_s \right)  - \sum_{s=2}^{t+1} \zeta_s \right)  && \forall t \in [T]\\
&&& \sum_{s=1}^{t+1} w_{t+1,s} \zeta_s  \ge - b_t \left( v_1 + \sum_{\ell=1}^t \left( \sum_{s=1}^\ell y_{\ell,s} \zeta_s \right)   - \sum_{s=2}^{t+1} \zeta_s \right)  &&  \forall  t \in [T]\\
&&&0 \le   \sum_{s=1}^t y_{t,s} \zeta_s  \le  p_{t} && \forall t\in[T+1]\\
&&& w_{1,1} \zeta_1 \ge 0 \\
&&& \quad \forall \zeta_1 \in \mathcal{U}_1,\ldots,\zeta_{T+1} \in \mathcal{U}_{T+1}.
\end{aligned}
\end{align*}
By rearranging terms in the above optimization problem, by adding an epigraph decision variable $c_0 \in \R$, and by noticing that $\sum_{\ell=t+1}^t h_t y_{\ell,s} = -\sum_{\ell=t+1}^t b_t y_{\ell,s}=0$ for each period $t \in [T]$, 
we observe that the above optimization problem can be written equivalently as
\begin{align} 
\underset{\substack{c_0 \in \R\\y_{t,1},\ldots,y_{t,t} \in \R \; \forall t \in [T+1]\\
w_{t,1},\ldots,w_{t,t} \in \R \; \forall t \in [T+1]}}{\textnormal{minimize}}\quad &c_0 \notag \\
\textnormal{subject to}\quad &   \sum_{s=1}^{T+1} \left(\sum_{t=s}^{T+1} \left( c_t y_{t,s} +   w_{t,s} \right) \right) \zeta_s  \le c_0 \tag{2a} \label{line:ldr_3a} \\
&  \sum_{s=1}^{t+1} \left(\sum_{\ell=s}^t h_t y_{\ell,s} - h_t - w_{t+1,s} \right) \zeta_s + \sum_{s=t+2}^{T+1} 0 \zeta_s \le -h_t v_1  &&  \forall  t \in [T] \tag{2b} \label{prob:example3:b}\\
& \sum_{s=1}^{t+1} \left(\sum_{\ell=s}^t -b_t y_{\ell,s} + b_t - w_{t+1,s} \right) \zeta_s + \sum_{s=t+2}^{T+1} 0 \zeta_s  \le b_t v_1  &&  \forall t \in [T] \tag{2c} \label{prob:example3:c}\\
 &0 \le  \sum_{s=1}^t y_{t,s} \zeta_s + \sum_{s=t+1}^{T+1} 0 \zeta_s \le p_t && \forall  t \in [T+1] \tag{2d} \label{prob:example3:d}\\
  &\sum_{s=1}^t w_{t,s} \zeta_s  + \sum_{s=t+1}^{T+1} 0 \zeta_s \ge 0 \tag{2e} && \forall t \in [1] \label{prob:example3:e}\\
& \quad \forall \zeta_1 \in \mathcal{U}_1,\ldots,\zeta_{T+1} \in \mathcal{U}_{T+1}.\notag \end{align}
With the above notation, we are now ready to invoke Lemma~\ref{lem:1}. Indeed, we recall from the statement of Theorem~\ref{thm:3} that Assumption~\ref{ass:1} holds for cost function~\eqref{prob:example3:a}-\eqref{prob:example3:b}. Therefore,  it follows from the above reasoning, the constraints in \eqref{prob:example3:b}-\eqref{prob:example3:e}, and Lemma~\ref{lem:1} that the feasible set of the above optimization problem is a nonempty polyhedron with at least one extreme point.  

\vspace{0.5em}
\textbf{Step 2}:   In the second step of our proof of Theorem~\ref{thm:3}, we use Lemma~\ref{lem:2} to characterize the structure of extreme points for the feasible set of the above optimization problem. Indeed, let $(\bar{\by},\bar{\bw},\bar{c}_0)$ denote an extreme point of the feasible set of the above optimization problem. Then it follows readily from Lemma~\ref{lem:2} that there exists 
\begin{itemize}
\item index sets $\mathcal{I}^{\ref{prob:example3:b},\bar{\by},\bar{\bw}}, \mathcal{I}^{\ref{prob:example3:c},\bar{\by},\bar{\bw}} \subseteq [T]$, $\ubar{\mathcal{I}}^{\ref{prob:example3:d},\bar{\by},\bar{\bw}},\bar{\mathcal{I}}^{\ref{prob:example3:d},\bar{\by},\bar{\bw}}  \subseteq [T + 1]$, and $\mathcal{I}^{\ref{prob:example3:e},\bar{\by},\bar{\bw}}  \subseteq [1]$; 
\item index sets $\mathcal{T}^{\ref{prob:example3:a},\bar{\by},\bar{\bw}} \subseteq [T+1]$, $\mathcal{T}^{\ref{prob:example3:b},\bar{\by},\bar{\bw}}_t \subseteq [T+1]$ for each $t \in \mathcal{I}^{\ref{prob:example3:b},\bar{\by},\bar{\bw}}$, $\mathcal{T}^{\ref{prob:example3:c},\bar{\by},\bar{\bw}}_t \subseteq [T+1]$ for each $t \in \mathcal{I}^{\ref{prob:example3:c},\bar{\by},\bar{\bw}}$, $\ubar{\mathcal{T}}^{\ref{prob:example3:d},\bar{\by},\bar{\bw}}_{t} \subseteq [T+1]$ for each $t \in \ubar{\mathcal{I}}^{\ref{prob:example3:d},\bar{\by},\bar{\bw}}$, $\bar{\mathcal{T}}^{\ref{prob:example3:d},\bar{\by},\bar{\bw}}_{t} \subseteq [T+1]$ for each $t \in \bar{\mathcal{I}}^{\ref{prob:example3:d},\bar{\by},\bar{\bw}}$, and  $\mathcal{T}_t^{\ref{prob:example3:e},\bar{\by},\bar{\bw}} \subseteq [T+1]$ for each $t \in \mathcal{I}^{\ref{prob:example3:e},\bar{\by},\bar{\bw}}$; 
\item  hyperplanes $(\balpha^{\ref{prob:example3:a},\bar{\by},\bar{\bw}},\bgamma^{\ref{prob:example3:a},\bar{\by},\bar{\bw}},\beta^{\ref{prob:example3:a},\bar{\by},\bar{\bw}})$, $(\balpha^{\ref{prob:example3:b},\bar{\by},\bar{\bw}}_t,\bgamma^{\ref{prob:example3:b},\bar{\by},\bar{\bw}}_t,\beta^{\ref{prob:example3:b},\bar{\by},\bar{\bw}}_t)$ for each $t \in {\mathcal{I}}^{\ref{prob:example3:b},\bar{\by},\bar{\bw}}$,
$(\balpha^{\ref{prob:example3:c},\bar{\by},\bar{\bw}}_t,\bgamma^{\ref{prob:example3:c},\bar{\by},\bar{\bw}}_t,\beta^{\ref{prob:example3:c},\bar{\by},\bar{\bw}}_t)$ for each $t \in {\mathcal{I}}^{\ref{prob:example3:c},\bar{\by},\bar{\bw}}$, 
$(\ubar{\balpha}^{\ref{prob:example3:d},\bar{\by},\bar{\bw}}_{t},\ubar{\bgamma}^{\ref{prob:example3:d},\bar{\by},\bar{\bw}}_{t},\ubar{\beta}^{\ref{prob:example3:d},\bar{\by},\bar{\bw}}_{t})$ for each $t \in \ubar{\mathcal{I}}^{\ref{prob:example3:d},\bar{\by},\bar{\bw}}$, $(\bar{\balpha}^{\ref{prob:example3:d},\bar{\by},\bar{\bw}}_{t},\bar{\bgamma}^{\ref{prob:example3:d},\bar{\by},\bar{\bw}}_{t}, \bar{\beta}^{\ref{prob:example3:d},\bar{\by},\bar{\bw}}_{t})$ for each $t \in \bar{\mathcal{I}}^{\ref{prob:example3:d},\bar{\by},\bar{\bw}}$, and  $(\balpha_t^{\ref{prob:example3:e},\bar{\by},\bar{\bw}},\bgamma_t^{\ref{prob:example3:e},\bar{\by},\bar{\bw}},\beta_t^{\ref{prob:example3:e},\bar{\by},\bar{\bw}})$ for each $t \in {\mathcal{I}}^{\ref{prob:example3:e},\bar{\by},\bar{\bw}}$
\end{itemize}
such that $(\bar{\by},\bar{\bw})$ is the unique solution to  the following system of equalities. % optimization problem:
\begin{align}
    %\begin{aligned}
&\balpha^{\ref{prob:example3:a},\bar{\by},\bar{\bw}} \cdot \by + \bgamma^{\ref{prob:example3:a},\bar{\by},\bar{\bw}} \cdot \bw = \beta^{\ref{prob:example3:a},\bar{\by},\bar{\bw}}  \tag{HARD-2a} \label{prob:example3:a1} \\
 &\balpha^{\ref{prob:example3:b},\bar{\by},\bar{\bw}}_t \cdot \by + \bgamma^{\ref{prob:example3:b},\bar{\by},\bar{\bw}}_t \cdot \bw = \beta^{\ref{prob:example3:b},\bar{\by},\bar{\bw}}_t && \forall t \in \mathcal{I}^{\ref{prob:example3:b},\bar{\by},\bar{\bw}}\tag{HARD-2b} \label{prob:example3:b1}\\
 &\balpha^{\ref{prob:example3:c},\bar{\by},\bar{\bw}}_t \cdot \by + \bgamma^{\ref{prob:example3:c},\bar{\by},\bar{\bw}}_t \cdot \bw = \beta^{\ref{prob:example3:c},\bar{\by},\bar{\bw}}_t && \forall t \in \mathcal{I}^{\ref{prob:example3:c},\bar{\by},\bar{\bw}}\tag{HARD-2c} \label{prob:example3:c1}\\
&\bar{\balpha}^{\ref{prob:example3:d},\bar{\by},\bar{\bw}}_{t} \cdot \by + \bar{\bgamma}^{\ref{prob:example3:d},\bar{\by},\bar{\bw}}_{t} \cdot \bw = \bar{\beta}^{\ref{prob:example3:d},\bar{\by},\bar{\bw}}_{t} && \forall t \in  \bar{\mathcal{I}}^{\ref{prob:example3:d},\bar{\by},\bar{\bw}}  \tag{HARD-2d-UB} \label{prob:example3:d1UB}\\
&\ubar{\balpha}^{\ref{prob:example3:d},\bar{\by},\bar{\bw}}_{t} \cdot \by + \ubar{\bgamma}^{\ref{prob:example3:d},\bar{\by},\bar{\bw}}_{t} \cdot \bw = \ubar{\beta}^{\ref{prob:example3:d},\bar{\by},\bar{\bw}}_{t} && \forall t \in  \ubar{\mathcal{I}}^{\ref{prob:example3:d},\bar{\by},\bar{\bw}}   \tag{HARD-2d-LB} \label{prob:example3:d1LB}\\
 &\balpha^{\ref{prob:example3:e},\bar{\by},\bar{\bw}}_t \cdot \by + \bgamma^{\ref{prob:example3:e},\bar{\by},\bar{\bw}}_t \cdot \bw = \beta^{\ref{prob:example3:e},\bar{\by},\bar{\bw}}_t && \forall t \in {\mathcal{I}}^{\ref{prob:example3:e},\bar{\by},\bar{\bw}} \tag{HARD-2e} \label{prob:example3:e1}\\
&\sum_{t=s}^{T+1} \left( c_t y_{t,s} +   w_{t,s} \right)= 0 && \forall s \in \mathcal{T}^{\ref{prob:example3:a},\bar{\by},\bar{\bw}}\tag{EASY-2a} \label{prob:example3:a2}\\
& \sum_{\ell=s}^t h_t y_{\ell,s} - h_t - w_{t+1,s} =0 && \forall t \in[T], s \in \mathcal{T}^{\ref{prob:example3:b},\bar{\by},\bar{\bw}}_t\text{ if } s \le t + 1\tag{EASY-2b-i} \label{prob:example3:b2_i}\\
& 0=0 && \forall t \in[T], s \in \mathcal{T}^{\ref{prob:example3:b},\bar{\by},\bar{\bw}}_t\text{ if } s \ge t + 2\tag{EASY-2b-ii} \label{prob:example3:b2_ii}\\
& - \sum_{\ell=s}^t b_t y_{\ell,s} + b_t - w_{t+1,s} =0 && \forall t \in[T], s \in \mathcal{T}^{\ref{prob:example3:c},\bar{\by},\bar{\bw}}_t\text{ if } s \le t + 1\tag{EASY-2c-i} \label{prob:example3:c2_i}\\
&0=0 && \forall t \in[T], s \in \mathcal{T}^{\ref{prob:example3:c},\bar{\by},\bar{\bw}}_t\text{ if } s \ge t + 2\tag{EASY-2c-ii} \label{prob:example3:c2_ii}\\
&y_{t,s} =0 && \forall t \in[T+1], \; s \in \bar{\mathcal{T}}^{\ref{prob:example3:d},\bar{\by},\bar{\bw}}_{t} \cup \ubar{\mathcal{T}}^{\ref{prob:example3:d},\bar{\by},\bar{\bw}}_{t} \text{ if } s \le t \tag{EASY-2d-i} \label{prob:example3:d2_i}\\
&0=0 && \forall t \in[T+1], \; s \in \bar{\mathcal{T}}^{\ref{prob:example3:d},\bar{\by},\bar{\bw}}_{t} \cup \ubar{\mathcal{T}}^{\ref{prob:example3:d},\bar{\by},\bar{\bw}}_{t} \text{ if } s \ge t+1   \tag{EASY-2d-ii} \label{prob:example3:d2_ii}\\
&w_{t,s} = 0 && \forall t \in [1], s \in {\mathcal{T}}^{\ref{prob:example3:e},\bar{\by},\bar{\bw}}_{t} \text{ if } s  \le t    \tag{EASY-2e-i} \label{prob:example3:e2_i}\\
&0 = 0 && \forall  t \in [1], s \in {\mathcal{T}}^{\ref{prob:example3:e},\bar{\by},\bar{\bw}}_{t} \text{ if } s \ge t + 1.    \tag{EASY-2e-ii} \label{prob:example3:e2_ii}
 \end{align}

 \vspace{0.5em}
 
 \textbf{Step 3}: In the third step of our proof of Theorem~\ref{thm:3}, we apply Lemma~\ref{lem:zeros} to every extreme point $(\bar{\by},\bar{\bw},\bar{c}_0)$. With the goal of Lemma~\ref{lem:zeros} in mind, we first observe from \eqref{prob:example3:b2_i} and \eqref{prob:example3:c2_i} that  any solution to the above system of equations must satisfy the following equalities:
 \begin{align*}
     w_{t+1,s} &= \sum_{\ell=s}^t h_t y_{\ell,s} - h_t && \forall t \in [T], \; s \in\mathcal{T}^{\ref{prob:example3:b},\bar{\by},\bar{\bw}}_t \text{ if } s \le t + 1\\
     w_{t+1,s} &= - \sum_{\ell=s}^t b_t y_{\ell,s} + b_t && \forall t \in [T], \; s \in\mathcal{T}^{\ref{prob:example3:c},\bar{\by},\bar{\bw}}_t  \text{ if } s \le t + 1. 
 \end{align*}

%Motivated by the above equality, 
We can thus eliminate the variables $w_{t+1,s}$ for each $t\in [T]$ and $s\in \mathcal{T}^{\ref{prob:example3:b},\bar{\by},\bar{\bw}}_t \cup \mathcal{T}^{\ref{prob:example3:c},\bar{\by},\bar{\bw}}_t$ with $s \le t + 1$ in the above system of equations by the substitution 
\begin{align*}
    w_{t+1,s} \leftarrow \begin{cases}
    \sum_{\ell=s}^t h_t y_{\ell,s} - h_t,&\text{if } t \in [T] \text{ and }s \in\mathcal{T}^{\ref{prob:example3:b},\bar{\by},\bar{\bw}}_t,\\
    - \sum_{\ell=s}^t b_t y_{\ell,s} + b_t,&\text{if } t \in [T] \text{ and }  s \in  \mathcal{T}^{\ref{prob:example3:c},\bar{\by},\bar{\bw}}_t  \setminus \mathcal{T}^{\ref{prob:example3:b},\bar{\by},\bar{\bw}}_t.
    \end{cases}
\end{align*}
With this substitution, we observe that \eqref{prob:example3:b2_i} and \eqref{prob:example3:c2_i} can be replaced with
\begin{align}
    \sum_{\ell=s}^t y_{\ell,s}=1 ~~~~\forall t\in [T], s\in \mathcal{T}^{\ref{prob:example3:c},\bar{\by},\bar{\bw}}_t \cap \mathcal{T}^{\ref{prob:example3:b},\bar{\by},\bar{\bw}}_t, \; s \le t + 1 \tag{EASY-2bc-i} \label{prob:example3:bc'2}
\end{align}

We next perform algebraic manipulations on the equations in \eqref{prob:example3:bc'2}. 
Indeed, we first define the following sets:
 \begin{align*}
 \mathscr{S}^{\bar{\by},\bar{\bw}} &\triangleq \bigcup_{t =1}^T  \left( \left( \mathcal{T}^{\ref{prob:example3:c},\bar{\by},\bar{\bw}}_t \cap \mathcal{T}^{\ref{prob:example3:b},\bar{\by},\bar{\bw}}_t\right) \cap \{1,\ldots,t+1\} \right),\\
  \mathscr{T}^{\bar{\by},\bar{\bw}}_s &\triangleq \left \{ t \in [T]:\; s \in \mathcal{T}^{\ref{prob:example3:c},\bar{\by},\bar{\bw}}_t \cap \mathcal{T}^{\ref{prob:example3:b},\bar{\by},\bar{\bw}}_t {\cap \{1,\ldots,t+1\}} \right \} \quad  \forall s \in \mathscr{S}^{\bar{\by},\bar{\bw}}.
 \end{align*} 
 With the above notation, we observe that the constraints \eqref{prob:example3:bc'2} can be rewritten as 
 \begin{align*}
 &\sum_{\ell=s}^t  y_{\ell,s}   = 1 && \forall s \in \mathscr{S}^{\bar{\by},\bar{\bw}}, \; t \in \mathscr{T}^{\bar{\by},\bar{\bw}}_s.% \tag{1d-2} \label{prob:example1:d2}
 \end{align*}
Moreover, let the elements of $\mathscr{T}^{\bar{\by},\bar{\bw}}_s$ be denoted by $t^{\bar{\by},\bar{\bw}}_{s,1} < \ldots < t^{\bar{\by},\bar{\bw}}_{s,| \mathscr{T}^{\bar{\by},\bar{\bw}}_s|}$. 
With this notation, we readily observe using algebra that the constraints \eqref{prob:example3:bc'2} can be replaced with the following equivalent constraints:
 \begin{align}
 &\sum_{\ell=s}^{t_{s,1}^{\bar{\by},\bar{\bw}}}  y_{\ell,s}   = 1 && \forall s \in \mathscr{S}^{\bar{\by},\bar{\bw}}  \tag{EASY-2bc-i'} \label{prob:example3:bc2_prime}\\
  &\sum_{\ell={t^{\bar{\by},\bar{\bw}}_{s,k}}+1}^{t^{\bar{\by},\bar{\bw}}_{s,k+1}}  y_{\ell,s}   = 0 && \forall s \in \mathscr{S}^{\bar{\by},\bar{\bw}}, \; k \in \left\{1,\ldots,|\mathscr{T}^{\bar{\by},\bar{\bw}}_s| - 1 \right \}.  \tag{EASY-2bc-i''} \label{prob:example3:bc2_prime_prime}
 \end{align}

To simplify our notation, we now compactly represent the constraints from lines~\eqref{prob:example3:a1}, \eqref{prob:example3:b1}, \eqref{prob:example3:c1}, \eqref{prob:example3:d1UB},  \eqref{prob:example3:d1LB},  \eqref{prob:example3:e1},  \eqref{prob:example3:a2},  and \eqref{prob:example3:bc2_prime} using the index set $\mathscr{I}^{\bar{\by},\bar{\bw}}$ and hyperplanes $(\balpha^{\bar{\by},\bar{\bw}}_i, \bgamma_i^{\bar{\by},\bar{\bw}}, \beta^{\bar{\by},\bar{\bw}}_i)$   for each $i \in \mathscr{I}^{\bar{\by},\bar{\bw}}$\footnote{The vector $\bgamma^{\bar{\by},\bar{\bw}}_i$ in each hyperplane $i \in \mathscr{I}^{\bar{\by},\bar{\bw}}$ contains an element corresponding to the variable $w_{t+1,s}$ for each $t \in [T]$ and $s \notin \mathcal{T}^{\ref{prob:example3:b},\bar{\by},\bar{\bw}}_t \cup \mathcal{T}^{\ref{prob:example3:c},\bar{\by},\bar{\bw}}_t$.}, where
 \begin{align*}
|  \mathscr{I}^{\bar{\by},\bar{\bw}}|&= \underbrace{1}_{\eqref{prob:example3:a1}} 
+  \underbrace{\left|\mathcal{I}^{\ref{prob:example3:b},\bar{\by},\bar{\bw}}\right|}_{\eqref{prob:example3:b1}} 
+  \underbrace{\left|\mathcal{I}^{\ref{prob:example3:c},\bar{\by},\bar{\bw}}\right|}_{\eqref{prob:example3:c1}}
+ \underbrace{\left| \bar{\mathcal{I}}^{\ref{prob:example3:d},\bar{\by},\bar{\bw}}   \right| }_{\eqref{prob:example3:d1UB}}
+  \underbrace{\left| \ubar{\mathcal{I}}^{\ref{prob:example3:d},\bar{\by},\bar{\bw}} \right| }_{\eqref{prob:example3:d1LB}} +  \underbrace{\left| {\mathcal{I}}^{\ref{prob:example3:e},\bar{\by},\bar{\bw}} \right| }_{\eqref{prob:example3:e1}} 
+ \underbrace{ \left| \mathcal{T}^{\ref{prob:example3:a},\bar{\by},\bar{\bw}} \right|}_{\eqref{prob:example3:a2}} 
+ \underbrace{ \left| \mathscr{S}^{\bar{\by},\bar{\bw}} \right|}_{\eqref{prob:example3:bc2_prime}}\\
&\le \underbrace{1}_{\eqref{prob:example3:a1}} +  \underbrace{T}_{\eqref{prob:example3:b1}} + 
\underbrace{T}_{\eqref{prob:example3:c1}} + 
\underbrace{T+1}_{\eqref{prob:example3:d1UB}} +  \underbrace{T+1}_{\eqref{prob:example3:d1LB}} +  \underbrace{1 }_{\eqref{prob:example3:e1}} + \underbrace{T+1}_{\eqref{prob:example3:a2}}+   \underbrace{T}_{\eqref{prob:example3:bc2_prime}}  \\
&\le 5 + 6T.
 \end{align*}
It follows from the above notation, from the fact that we have used substitution to eliminate  the variable $w_{t+1,s}$ for each $t \in [T]$ and $s \in  \mathcal{T}^{\ref{prob:example3:c},\bar{\by},\bar{\bw}}_t  \cup \mathcal{T}^{\ref{prob:example3:b},\bar{\by},\bar{\bw}}_t$ with $s \le t + 1$, and from the fact that \eqref{prob:example3:b2_ii}, \eqref{prob:example3:c2_ii},  \eqref{prob:example3:d2_ii}, and  \eqref{prob:example3:e2_ii} can be eliminated without loss of generality that $\left(\bar{\by}, \left( \bar{w}_{t+1,s}: t \in [T], \; s \notin \left( \mathcal{T}^{\ref{prob:example3:c},\bar{\by},\bar{\bw}}_t  \cup \mathcal{T}^{\ref{prob:example3:b},\bar{\by},\bar{\bw}}_t \right) \cap \{1,\ldots,t+1\} \right ) \right)$ is the unique solution to the following system of equations. 
\begin{align}
    %\begin{aligned}
&\balpha_i^{\bar{\by},\bar{\bw}} \cdot \by + \sum_{t=1}^T \sum_{s \notin \left(\mathcal{T}^{\ref{prob:example3:c},\bar{\by},\bar{\bw}}_t  \cup \mathcal{T}^{\ref{prob:example3:b},\bar{\by},\bar{\bw}}_t {\cap \{1,\ldots,t+1\}} \right )}  \gamma_i^{\bar{\by},\bar{\bw}} w_{t+1,s} = \beta_i^{\bar{\by},\bar{\bw}} && \forall i \in \mathscr{I}^{\bar{\by},\bar{\bw}} \notag \\
 &y_{t,s} =0 && \forall t \in[T+1], \; s \in \bar{\mathcal{T}}^{\ref{prob:example3:d},\bar{\by},\bar{\bw}}_{t} \cup \ubar{\mathcal{T}}^{\ref{prob:example3:d},\bar{\by},\bar{\bw}}_{t} \text{ if } s \le t \tag{EASY-2d-i} \label{prob:example3:d2_i}\\
&w_{1,1} = 0 && \forall t \in [1], s \in {\mathcal{T}}^{\ref{prob:example3:e},\bar{\by},\bar{\bw}}_{t} \text{ if } s  \le t    \tag{EASY-2e-i} \label{prob:example3:e2_i}\\
  &\sum_{\ell={t^{\bar{\by},\bar{\bw}}_{s,k}}+1}^{t^{\bar{\by},\bar{\bw}}_{s,k+1}}  y_{\ell,s}   = 0 && \forall s \in \mathscr{S}^{\bar{\by},\bar{\bw}}, \; k \in \left\{1,\ldots,|\mathscr{T}^{\bar{\by},\bar{\bw}}_s| - 1 \right \}.  \tag{EASY-2bc''}\label{line:EASY-2bc''}
  \end{align}
Notice that any $y_{\ell,s}$ only appear once in \eqref{line:EASY-2bc''}. It thus follows from Lemma~\ref{lem:zeros} that the number of nonzero entries in $\bar{\by}$ satisfies
\begin{align*}
\| \bar{\by} \|_0 \le
\left\| \left(\bar{\by}, \left( \bar{w}_{t+1,s}: t \in [T], \; s \notin \left( \mathcal{T}^{\ref{prob:example3:c},\bar{\by},\bar{\bw}}_t  \cup \mathcal{T}^{\ref{prob:example3:b},\bar{\by},\bar{\bw}}_t \right) \cap \{1,\ldots,t+1\} \right ) \right) \right\|_0 \le 2 \left| \mathscr{I}^{\bar{\by},\bar{\bw}} \right| &\le 10 + 12T,
\end{align*}
which concludes our proof of Theorem~\ref{thm:3}.
\halmos \end{proof}

% \section{Proof fro Theorem }

%\clearpage

\section{Proof of Theorem~\ref{thm:nonbox}\label{sec:proof_4}}
\begin{proof}{Proof of Theorem~\ref{thm:nonbox}.}

The basic idea of the proof is reducing the problem with non-separable uncertainty set to a problem with box uncertainty set, by dualizing the linking constraints, and then utilizing the same arguments as in the proof of Theorem \ref{thm:main}.

With the same argument as the Step 1 in the proof of Theorem~\ref{thm:main}, we know that the epigraph form of the problem is equivalent to
    \begin{align}\label{budget} \tag{3}
   & \underset{\substack{c_0 \in \R \\\by_{t,1},\ldots,\by_{t,t} \in \R^E:\;  \forall t \in [T+1]}}{\textnormal{minimize}}&&c_0  \notag \\ 
    &\textnormal{subject to}&& \sum_{s=1}^{T +1}  \left( \sum_{t=s}^{T+1} \sum_{e=1}^E c_{te} y_{t,s,e} \right)  \zeta_s \le c_0 \tag{3a} \label{4-a} \\
    &&&\sum_{s=1}^{T+1} \left( \sum_{t=s}^{T+1} y_{t,s,e} \right) \zeta_s \le  Q_e \quad \quad \quad   \forall e \in[E] \tag{3b} \label{4-b} \\
   &&& 0 \le  \sum_{s=1}^{t}  y_{t,s,e} \zeta_s + \sum_{s=t+1}^{T+1} 0 \zeta_s \le  p_{te} \quad \forall  e\in[E],\; t \in [T+1] \tag{3c} \label{4-c}\\
   &&&V_{\textnormal{min}} \le  v_1 + \left( \sum_{\ell=1}^t \sum_{e \in [E]} y_{\ell,1,e} \right) \zeta_1 \nonumber \\
   &&& \quad \quad \quad \quad + \sum_{s=2}^t \left(-1 + \sum_{\ell=s}^t \sum_{e \in [E]} y_{\ell,s,e} \right) \zeta_s - \zeta_{t+1} + \sum_{s=t+2}^{T+1} 0 \zeta_s  \le  V_{\textnormal{max}} \tag{3d} \label{4-d}\quad  \forall  t \in [T]\\
   &&& \quad \forall \bzeta \equiv (\zeta_1,\ldots,\zeta_{T+1})\in \mathcal{U}.\nonumber
    \end{align}
  % \end{equation} 

Next, we denote $f_s=\frac{1}{\bar{D}_s-\ubar{D}_s}$ for $s=2,...,T+1$, $f_s=0$ for $s=1$,  and $g=T-k+ \sum_{s=2}^{T+1} \frac{ \ubar{D}_s}{\bar{D}_s - \ubar{D}_s}$. Then, we can rewrite the budget uncertainty set $\mathcal{U}$ as  $\{\zeta_1 \in \mathcal{U}_1, \ldots,\zeta_T \in \mathcal{U}_T: \bff^\intercal \bzeta \le g\}$. Now we consider the generic form of the problem with the notation of \eqref{set}, where the feasible region becomes
\begin{align*}
\mathcal{Y} \triangleq \left \{ (\by,c_0): \;  \max_{\zeta_1 \in \mathcal{U}_1, \ldots,\zeta_T \in \mathcal{U}_T, \bff^\intercal\bzeta\le g}  \left \{ \sum_{t=1}^T \ba_{i,t}^\intercal  \left(\sum_{s=1}^t \by_{t,s} \zeta_s \right) - \sum_{t=1}^T b_{i,t}   \zeta_t \right \}  \le c_i  \;\; \forall i \in \{0,\ldots,m\} \right \}.
\end{align*}
Furthermore, we denote $\mu_i\ge 0$ as the dual variable for $\bff^\intercal \bzeta \le g$ in the $i$-th constraint, and denote
\begin{align*}%\label{set}
\mathcal{Y}_{\mu} \triangleq \left \{ (\by,c_0): \;  \max_{\zeta_1 \in \mathcal{U}_1, \ldots,\zeta_T \in \mathcal{U}_T}  \left \{ \sum_{t=1}^T \ba_{i,t}^\intercal  \left(\sum_{s=1}^t \by_{t,s} \zeta_s \right) - \sum_{t=1}^T b_{i,t}   \zeta_t + \mu_i (g-\bff^\intercal \bzeta) \right \}  \le c_i  \;\; \forall i \in \{0,\ldots,m\} \right \}.
\end{align*}
The next lemma presents the connection between $\mathcal{Y}$ and $\mathcal{Y}_{\bmu}$:
\begin{lemma}\label{lem:equi}
$\mathcal{Y}=\cup_{\bmu\ge \bzero} \mathcal{Y}_{\bmu}$.
\end{lemma}

\begin{proof}{Proof.}
Consider any $\bmu\ge \bzero$, and let $(\by,c_0)\in \mathcal{Y}_{\bmu}$. Then it holds for each $i \in \{0,\ldots,m\}$ that
\begin{align*}
&\max_{\zeta_1 \in \mathcal{U}_1, \ldots,\zeta_T \in \mathcal{U}_T, \bff^\intercal\bzeta\le g}  \left \{ \sum_{t=1}^T \ba_{i,t}^\intercal  \left(\sum_{s=1}^t \by_{t,s} \zeta_s \right) - \sum_{t=1}^T b_{i,t}   \zeta_t \right \}\\
&\le \max_{\zeta_1 \in \mathcal{U}_1, \ldots,\zeta_T \in \mathcal{U}_T}  \left \{  \sum_{t=1}^T \ba_{i,t}^\intercal  \left(\sum_{s=1}^t \by_{t,s} \zeta_s \right) - \sum_{t=1}^T b_{i,t}   \zeta_t + \mu_i (g-\bff^\intercal\bzeta) \right \}\\
&\le c_i,
\end{align*}
where the first inequality uses weak duality and the second inequality is from the definition of $\mathcal{Y}_{\bmu}$. Thus, it holds for all $\mu_i\ge 0$ that $\mathcal{Y}_{\bmu}\subseteq \mathcal{Y}$, whereby
\begin{align}\label{eq:equi_1}
    \cup_{\bmu\ge \bzero}\mathcal{Y}_{\bmu}\subseteq \mathcal{Y}.
\end{align}

On the other hand, for any $(\by,c_0)\in \mathcal{Y}$, it holds by strong duality of linear optimization that there exists $\bmu\ge \bzero$ such that
\begin{align*}
    &\max_{\zeta_1 \in \mathcal{U}_1, \ldots,\zeta_T \in \mathcal{U}_T}  \left \{ \sum_{t=1}^T \ba_{i,t}^\intercal  \left(\sum_{s=1}^t \by_{t,s} \zeta_s \right) - \sum_{t=1}^T b_{i,t}   \zeta_t + \mu_i (g-\bff^\intercal\xi) \right \}\\
    =& \max_{\zeta_1 \in \mathcal{U}_1, \ldots,\zeta_T \in \mathcal{U}_T, \bff^\intercal\zeta\le g}  \left \{ \sum_{t=1}^T \ba_{i,t}^\intercal  \left(\sum_{s=1}^t \by_{t,s} \zeta_s \right) - \sum_{t=1}^T b_{i,t}   \zeta_t \right \} \le c_i, \forall i\in \{0,...,m\} \ ,
\end{align*}
thus it holds that $y\in \mathcal{Y}_{\bmu}$, whereby
\begin{align}\label{eq:equi_2}
    \mathcal{Y} \subseteq \cup_{\bmu\ge \bzero}\mathcal{Y}_{\bmu}\ .
\end{align}
We finish the proof by combining \eqref{eq:equi_1} and \eqref{eq:equi_2}.\halmos
\end{proof}

The next lemma shows that an extreme point of $\mathcal{Y}$ is also an extreme point of $\mathcal{Y}_{\bmu}$.
\begin{lemma}\label{lem:extreme_point}
For any extreme point $(\by,c_0)$ of $\mathcal{Y}$, there exists $\bmu\ge \bzero$, such that $(\by,c_0)$ is an extreme point of $\mathcal{Y}_{\bmu}$.
\end{lemma}
\begin{proof}{Proof.}
It follows from Lemma \ref{lem:equi} that there exists $\bmu\ge \bzero$ such that $(\by,c_0)\in \mathcal{Y}_{\bmu}$. Suppose $(\by,c_0)$ is not an extreme point of $\mathcal{Y}_{\bmu}$, then there exist $(\by^1,c_0^1), (\by^2,c_0^2)\in \mathcal{Y}_{\bmu}$ and $0<\lambda<1$ such that $(\by,c_0)=\lambda (\by^1,c_0^1) + (1-\lambda) (\by^2,c_0^2)$. Notice that it follows from Lemma~\ref{lem:equi} that $(\by^1,c_0^1), (\by^2,c_0^2)\in \mathcal{Y}$. This shows that $(\by,c_0)$ is not an extreme point of $\mathcal{Y}$, which leads to a contradiction. Therefore, it must be the case that $(\by,c_0)$ is an extreme point of $\mathcal{Y}_{\bmu}$. \halmos
\end{proof}

Notice that the uncertainty set for $\mathcal{Y}_{\bmu}$ only involves box constraint. Now we can use the above argument to reduce \eqref{budget} with non-box constraint to a similar problem with box constraint by dualizing the budget constraint. 

Another key observation is that the constraint \eqref{4-c} for $t\le T+1-k$ only involves $\zeta_1,...,\zeta_t$ and is independent of $\zeta_{T+1-k+1},\zeta_{T+1-k+2},...,\zeta_{T+1}$. Furthermore, notice that for any $\zeta_1\in \mathcal{U}_1, \zeta_2\in \mathcal{U}_2, ..., \zeta_{T+1-k}\in \mathcal{U}_{T+1-k}$, we can find $\bzeta \equiv (\zeta_1,...,\zeta_{T+1})\in \mathcal{U}$ with the same value of $\zeta_1,\ldots,\zeta_{T+1-k}$ due to the definition of $\mathcal U$. This means that the effective uncertainty set for \eqref{4-c} with $t\le T+1-k$ is exactly a box constraint. The same argument works for \eqref{4-d} for $t\le T+1-k$.

Therefore, for any extreme point $\bar \by$ of the constraint set in \eqref{budget}, it follows from Lemma \ref{lem:extreme_point} that there exists $\mu^a\ge 0$, $\mu^b_e\ge 0$ for $e\in [E]$, $\ubar{\mu}^c_{et}, \bar{\mu}^c_{et}\ge 0$ for $e\in[E]$, $t\in\{T+2-k,...,T+1\}$, $\ubar{\mu}_t^d,\bar{\mu}_t^d\ge 0$ for $t\in \{T+2-k,...,T+1\}$, such that $\bar \by$ is an extreme point to the following set:

 \begin{align}
    & && \sum_{s=1}^{T +1}  \left( -\mu^a f_s +\sum_{t=s}^{T+1} \sum_{e=1}^E c_{te} y_{t,s,e} \right)  \zeta_s \le c_0 -\mu^a g \tag{4a} \label{budget-2a} \\
    &&&\sum_{s=1}^{T+1} \left( -\mu^b_e f_s +\sum_{t=s}^{T+1} y_{t,s,e} \right) \zeta_s \le  Q_e -\mu^b_e g \quad \quad \quad   \forall e \in[E] \tag{4b} \label{budget-2b} \\
   &&& 0 \le  \sum_{s=1}^{t}  y_{t,s,e} \zeta_s  \le p_{te} \quad \forall  e\in[E],\; t \in \{1,...,T+1-k\} \tag{4c-i} \label{budget-2c-i}\\
   &&& -\ubar{\mu}^c_{et} g \le  \sum_{s=1}^{t}  \left(y_{t,s,e}-\ubar{\mu}^c_{et}f_s\right) \zeta_s   \quad \forall  e\in[E],\; t \in \{T+2-k,..., T+1\} \tag{4c-ii-LB} \label{budget-2c-ii-LB}\\
   &&&  \sum_{s=1}^{t}  \left(y_{t,s,e}-\bar{\mu}^c_{et}f_s\right) \zeta_s  \le  p_{te} -\bar{\mu}^c_{et} g \quad \forall  e\in[E],\; t \in \{T+2-k,..., T+1\} \tag{4c-ii-UB} \label{budget-2c-ii-UB}\\
   &&&V_{\textnormal{min}} \le  v_1 + \left( \sum_{\ell=1}^t \sum_{e \in [E]} y_{\ell,1,e} \right) \zeta_1 \nonumber \\
   &&& \quad \quad \quad \quad + \sum_{s=2}^t \left(-1 + \sum_{\ell=s}^t \sum_{e \in [E]} y_{\ell,s,e} \right) \zeta_s - \zeta_{t+1} \le  V_{\textnormal{max}}   \tag{4d-i} \label{budget-2d-i}\quad  \forall  t \in \{1,...,T+1-k\}\\
   &&&V_{\textnormal{min}} -\ubar{\mu}_t^d g \le  v_1 + \left(-\ubar{\mu}_t^d f_1+ \sum_{\ell=1}^t \sum_{e \in [E]} y_{\ell,1,e} \right) \zeta_1 \nonumber \\
   &&& \quad \quad \quad \quad + \sum_{s=2}^t \left(-1 -\ubar{\mu}_t^d f_s + \sum_{\ell=s}^t \sum_{e \in [E]} y_{\ell,s,e} \right) \zeta_s - \zeta_{t+1}    \tag{4d-ii-LB} \label{budget-2d-ii-LB}\quad  \forall  t \in \{T+2-k,..., T\}\\
   &&&  v_1 + \left( -\bar{\mu}_t^d f_1 +\sum_{\ell=1}^t \sum_{e \in [E]} y_{\ell,1,e} \right) \zeta_1 \nonumber \\
   &&& \quad \quad \quad \quad + \sum_{s=2}^t \left(-1 -\bar{\mu}_t^d f_s + \sum_{\ell=s}^t \sum_{e \in [E]} y_{\ell,s,e} \right) \zeta_s - \zeta_{t+1}   \le  V_{\textnormal{max}} -\bar{\mu}_t^d g \tag{4d-ii-UB} \label{budget-2d-ii-UB}\quad  \forall  t \in\{T+2-k,..., T\}\\
   &&& \quad \forall \zeta_1 \in \mathcal{U}_1,\ldots,\zeta_{T+1} \in \mathcal{U}_{T+1} \nonumber.
    \end{align}

Now notice that the uncertainty set is a box constraint, and we can utilize the same arguments as in the proof of Theorem \ref{thm:main} for the proof. We can create a similar set of equality systems as in \eqref{line:hard_1a}-\eqref{line:easy_1d_iv}. Consider an extreme point $\bar y$ to the constraint set \eqref{budget-2a}-\eqref{budget-2d-ii-UB}. It follows from Lemma~\ref{lem:2} that there exists 
\begin{itemize}
\item
index sets $\mathcal{I}^{\ref{budget-2b},\bar{\by}} \subseteq [E]$,   $\ubar{\mathcal{I}}^{\ref{budget-2c-i},\bar{\by}},\bar{\mathcal{I}}^{\ref{budget-2c-i},\bar{\by}} \subseteq [T+1-k]\times[E]$, ${\mathcal{I}}^{\ref{budget-2c-ii-LB},\bar{\by}},{\mathcal{I}}^{\ref{budget-2c-ii-UB},\bar{\by}} \subseteq \{T+2-k,...,T+1\}\times[E]$, $\ubar{\mathcal{I}}^{\ref{budget-2d-i},\bar{\by}},\bar{\mathcal{I}}^{\ref{budget-2d-i},\bar{\by}}  \subseteq [T+1-k]$ and  ${\mathcal{I}}^{\ref{budget-2d-ii-LB},\bar{\by}},{\mathcal{I}}^{\ref{budget-2d-ii-UB},\bar{\by}}  \subseteq \{T+2-k,...,T+1\}$; 

\item index sets $\mathcal{T}^{\ref{budget-2a},\bar{\by}} \subseteq [T+1]$, $\mathcal{T}^{\ref{budget-2b},\bar{\by}}_e \subseteq [T+1]$ for each $e \in \mathcal{I}^{\ref{budget-2b},\bar{\by}}$,  
$\ubar{\mathcal{T}}^{\ref{budget-2c-i},\bar{\by}}_{t,e} \subseteq [T+1-k]$ for each $(t,e) \in \ubar{\mathcal{I}}^{\ref{budget-2c-i},\bar{\by}}$, 
$\bar{\mathcal{T}}^{\ref{budget-2c-i},\bar{\by}}_{t,e} \subseteq [T+1-k]$  for each $(t,e) \in \bar{\mathcal{I}}^{\ref{budget-2c-i},\bar{\by}}$, 
${\mathcal{T}}^{\ref{budget-2c-ii-LB},\bar{\by}}_{t,e} \subseteq \{T+2-k,...,T+1\}$ for each $(t,e) \in {\mathcal{I}}^{\ref{budget-2c-ii-LB},\bar{\by}}$, 
${\mathcal{T}}^{\ref{budget-2c-ii-UB},\bar{\by}}_{t,e} \subseteq \{T+2-k,...,T+1\}$  for each $(t,e) \in {\mathcal{I}}^{\ref{budget-2c-ii-UB},\bar{\by}}$,
$\ubar{\mathcal{T}}^{\ref{budget-2d-i},\bar{\by}}_{t} \subseteq [T+1-k]$ for each $t \in \ubar{\mathcal{I}}^{\ref{budget-2d-i},\bar{\by}}$,   $\bar{\mathcal{T}}^{\ref{budget-2d-i},\bar{\by}}_{t} \subseteq [T+1-k]$ for each $t \in \bar{\mathcal{I}}^{\ref{budget-2d-i},\bar{\by}}$, ${\mathcal{T}}^{\ref{budget-2d-ii-LB},\bar{\by}}_{t} \subseteq \{T+2-k,...,T+1\}$ for each $t \in {\mathcal{I}}^{\ref{budget-2d-ii-LB},\bar{\by}}$,   ${\mathcal{T}}^{\ref{budget-2d-ii-UB},\bar{\by}}_{t} \subseteq \{T+2-k,...,T+1\}$ for each $t \in {\mathcal{I}}^{\ref{budget-2d-ii-UB},\bar{\by}}$;

\item  hyperplanes $(\balpha^{\ref{budget-2a},\bar{\by}},\beta^{\ref{budget-2a},\bar{\by}})$, $(\balpha^{\ref{budget-2b},\bar{\by}}_e,\beta^{\ref{budget-2b},\bar{\by}}_e)$ for each $e \in {\mathcal{I}}^{\ref{budget-2b},\bar{\by}}$, 
$(\ubar{\balpha}^{\ref{budget-2c-i},\bar{\by}}_{t,e},\ubar{\beta}^{\ref{budget-2c-i},\bar{\by}}_{t,e})$ for each $(t,e) \in \ubar{\mathcal{I}}^{\ref{budget-2c-i},\bar{\by}}$, 
  $(\bar{\balpha}^{\ref{budget-2c-i},\bar{\by}}_{t,e},\bar{\beta}^{\ref{budget-2c-i},\bar{\by}}_{t,e})$ for each $(t,e) \in \bar{\mathcal{I}}^{\ref{budget-2c-i},\bar{\by}}$,
  $({\balpha}^{\ref{budget-2c-ii-LB},\bar{\by}}_{t,e},{\beta}^{\ref{budget-2c-ii-LB},\bar{\by}}_{t,e})$ for each $(t,e) \in {\mathcal{I}}^{\ref{budget-2c-ii-LB},\bar{\by}}$, 
  $({\balpha}^{\ref{budget-2c-ii-UB},\bar{\by}}_{t,e},{\beta}^{\ref{budget-2c-ii-UB},\bar{\by}}_{t,e})$ for each $(t,e) \in {\mathcal{I}}^{\ref{budget-2c-ii-UB},\bar{\by}}$,
  $(\ubar{\balpha}^{\ref{budget-2d-i},\bar{\by}}_{t},\ubar{\beta}^{\ref{budget-2d-i},\bar{\by}}_{t})$ for each $t \in \ubar{\mathcal{I}}^{\ref{budget-2d-i}, \bar{\by}}$, 
  $(\bar{\balpha}^{\ref{budget-2d-i},\bar{\by}}_{t},\bar{\beta}^{\ref{budget-2d-i},\bar{\by}}_{t})$ for each $t \in \bar{\mathcal{I}}^{\ref{budget-2d-i}, \bar{\by}}$ ,
  $({\balpha}^{\ref{budget-2d-ii-LB},\bar{\by}}_{t},{\beta}^{\ref{budget-2d-ii-LB},\bar{\by}}_{t})$ for each $t \in {\mathcal{I}}^{\ref{budget-2d-ii-LB}, \bar{\by}}$, 
  $(\bar{\balpha}^{\ref{budget-2d-ii-UB},{\by}}_{t},{\beta}^{\ref{budget-2d-ii-UB},\bar{\by}}_{t})$ for each $t \in {\mathcal{I}}^{\ref{budget-2d-ii-UB}, \bar{\by}}$ ,
  \end{itemize}
such that $\bar{\by}$ is the unique solution to  the following system of equations. % optimization problem:
\begin{align}
    %\begin{aligned}
&\balpha^{\ref{budget-2a},\bar{\by}} \cdot \by = \beta^{\ref{budget-2a},\bar{\by}}  \tag{HARD-4a} \label{line:budget_hard_5a} \\
 &\balpha^{\ref{budget-2b},\bar{\by}}_e \cdot \by = \beta^{\ref{budget-2b},\bar{\by}}_e && \forall e \in \mathcal{I}^{\ref{budget-2b},\bar{\by}}\tag{HARD-4b} \label{line:budget_hard_5b}\\
&\ubar{\balpha}^{\ref{budget-2c-i},\bar{\by}}_{t,e} \cdot \by = \ubar{\beta}^{\ref{budget-2c-i},\bar{\by}}_{t,e} && \forall (t,e) \in  \ubar{\mathcal{I}}^{\ref{budget-2c-i},\bar{\by}}  \tag{HARD-4c-i-LB} \label{line:budget_hard_5c_lb}\\
&\bar{\balpha}^{\ref{budget-2c-i},\bar{\by}}_{t,e} \cdot \by = \bar{\beta}^{\ref{budget-2c-i},\bar{\by}}_{t,e} && \forall (t,e) \in  \bar{\mathcal{I}}^{\ref{budget-2c-i},\bar{\by}}  \tag{HARD-4c-i-UB} \label{line:budget_hard_5c_ub}\\
&{\balpha}^{\ref{budget-2c-ii-UB},\bar{\by}}_{t,e} \cdot \by = {\beta}^{\ref{budget-2c-ii-UB},\bar{\by}}_{t,e} && \forall (t,e) \in  {\mathcal{I}}^{\ref{budget-2c-ii-UB},\bar{\by}}  \tag{HARD-4c-ii-LB} \label{line:budget_hard_5c_ii_lb}\\
&{\balpha}^{\ref{budget-2c-ii-LB},\bar{\by}}_{t,e} \cdot \by = {\beta}^{\ref{budget-2c-ii-LB},\bar{\by}}_{t,e} && \forall (t,e) \in  {\mathcal{I}}^{\ref{budget-2c-ii-LB},\bar{\by}}  \tag{HARD-4c-ii-UB} \label{line:budget_hard_5c_ii_ub}\\
&\ubar{\balpha}^{\ref{budget-2d-i},\bar{\by}}_{t} \cdot \by = \ubar{\beta}^{\ref{budget-2d-i},\bar{\by}}_{t} && \forall t \in  \ubar{\mathcal{I}}^{\ref{budget-2d-i},\bar{\by}}  \tag{HARD-4d-i-LB} \label{line:budget_hard_5d_lb}\\
&\bar{\balpha}^{\ref{budget-2d-i},\bar{\by}}_{t} \cdot \by = \bar{\beta}^{\ref{budget-2d-i},\bar{\by}}_{t} && \forall t \in  \bar{\mathcal{I}}^{\ref{budget-2d-i},\bar{\by}} 
\tag{HARD-4d-i-UB} \label{line:budget_hard_5d_ub}\\
&{\balpha}^{\ref{budget-2d-ii-UB},\bar{\by}}_{t} \cdot \by = {\beta}^{\ref{budget-2d-ii-UB},\bar{\by}}_{t} && \forall t \in  {\mathcal{I}}^{\ref{budget-2d-ii-UB},\bar{\by}}  \tag{HARD-4d-ii-LB} \label{line:budget_hard_5d_ii_lb}\\
&{\balpha}^{\ref{budget-2d-ii-LB},\bar{\by}}_{t} \cdot \by = {\beta}^{\ref{budget-2d-ii-LB},\bar{\by}}_{t} && \forall t \in  {\mathcal{I}}^{\ref{budget-2d-ii-LB},\bar{\by}} 
\tag{HARD-4d-ii-UB} \label{line:budget_hard_5d_ii_ub}\\
&\sum_{t=s}^{T+1} \sum_{e=1}^E  c_{te} y_{t,s,e} = \mu^a f_s && \forall s \in \mathcal{T}^{\ref{budget-2a},\bar{\by}}\tag{EASY-4a} \label{line:budget_easy_5a}\\
& \sum_{t=s}^{T+1} y_{t,s,e} =\mu^b_e g && \forall e \in[E], \; s \in \mathcal{T}^{\ref{budget-2b},\bar{\by}}_e\tag{EASY-4b} \label{line:budget_easy_5b}\\
&y_{t,s,e} =0 && \forall e \in[E], \; t \in[T+1-k], \; s \in  \ubar{\mathcal{T}}^{\ref{budget-2c-i},\bar{\by}}_{t,e}  \cup \bar{\mathcal{T}}^{\ref{budget-2c-i},\bar{\by}}_{t,e} \text{ if } s \le t \tag{EASY-4c-i} \label{line:budget_easy_5c_i}\\
&y_{t,s,e} =\ubar{\mu}^c_{et}f_s && \forall e \in[E], \; t \in\{T+2-k,T+1\}, \; s \in  {\mathcal{T}}^{\ref{budget-2c-ii-LB},\bar{\by}}_{t,e}   \text{ if } s \le t \tag{EASY-4c-ii-LB} \label{line:budget_easy_5c_ii_LB}\\
&y_{t,s,e} =\bar{\mu}^c_{et}f_s && \forall e \in[E], \; t \in\{T+2-k,T+1\}, \; s \in   {\mathcal{T}}^{\ref{budget-2c-ii-UB},\bar{\by}}_{t,e} \text{ if } s \le t \tag{EASY-4c-ii-UB} \label{line:budget_easy_5c_ii_UB}\\
& \sum_{\ell=s}^t \sum_{e \in [E]} y_{\ell,s,e}    = 0 && \forall t \in [T+1-k], \; s \in \ubar{\mathcal{T}}^{\ref{budget-2d-i},\bar{\by}}_t \cup \bar{\mathcal{T}}^{\ref{budget-2d-i},\bar{\by}}_t \text{ if } s =1\tag{EASY-4d-i-i} \label{line:budget_easy_5d_i_i}\\
& \sum_{\ell=s}^t \sum_{e \in [E]} y_{\ell,s,e}    = 1 && \forall t \in [T+1-k], \; s \in \ubar{\mathcal{T}}^{\ref{budget-2d-i},\bar{\by}}_t \cup \bar{\mathcal{T}}^{\ref{budget-2d-i},\bar{\by}}_t \text{ if } s \in \{2,\ldots,t\} \tag{EASY-4d-i-ii} \label{line:budget_easy_5d_i_ii}\\
& \sum_{\ell=s}^t \sum_{e \in [E]} y_{\ell,s,e}    = \ubar{\mu}_t^d f_1 && \forall t \in \{T+2-k,T+1\}, \; s \in {\mathcal{T}}^{\ref{budget-2d-ii-LB},\bar{\by}}_t  \text{ if } s =1\tag{EASY-4d-ii-LB-i} \label{line:budget_easy_5d_ii_LB_i}\\
& \sum_{\ell=s}^t \sum_{e \in [E]} y_{\ell,s,e}    = 1 +\ubar{\mu}_t^d f_s && \forall t \in \{T+2-k,T+1\}, \; s \in {\mathcal{T}}^{\ref{budget-2d-ii-LB},\bar{\by}}_t  \text{ if } s \in \{2,\ldots,t\} \tag{EASY-4d-ii-LB-ii} \label{line:budget_easy_5d_ii_LB_ii}\\
& \sum_{\ell=s}^t \sum_{e \in [E]} y_{\ell,s,e}    = \bar{\mu}_t^d f_1 && \forall t \in \{T+2-k,T+1\}, \; s \in {\mathcal{T}}^{\ref{budget-2d-ii-UB},\bar{\by}}_t  \text{ if } s =1\tag{EASY-4d-ii-UB-i} \label{line:budget_easy_5d_ii_UB_i}\\
& \sum_{\ell=s}^t \sum_{e \in [E]} y_{\ell,s,e}    = 1+\bar{\mu}_t^d f_s && \forall t \in \{T+2-k,T+1\}, \; s \in {\mathcal{T}}^{\ref{budget-2d-ii-UB},\bar{\by}}_t \text{ if } s \in \{2,\ldots,t\}. \tag{EASY-4d-ii-UB-ii} \label{line:budget_easy_5d_ii_UB_ii}
 \end{align}

With exactly the same argument as in the step 3 of the proof of Theorem \ref{thm:main}, we have
 \begin{align*}
| \mathcal I|
&\le \underbrace{1}_{\eqref{line:budget_hard_5a}} +  \underbrace{E}_{\eqref{line:budget_hard_5b}} + \underbrace{(T+1-k)E}_{\eqref{line:budget_hard_5c_lb}}  + \underbrace{(T+1-k)E}_{\eqref{line:budget_hard_5c_ub}} + 
\underbrace{kE}_{\eqref{line:budget_hard_5c_ii_lb}}  + \underbrace{kE}_{\eqref{line:budget_hard_5c_ii_ub}}\\
& +\underbrace{T+1-k}_{\eqref{line:budget_hard_5d_lb}}+ \underbrace{T+1-k}_{\eqref{line:budget_hard_5d_ub}}+ \underbrace{k-1}_{\eqref{line:budget_hard_5d_ii_lb}}+ \underbrace{k-1}_{\eqref{line:budget_hard_5d_ii_ub}}
\quad + \underbrace{T+1}_{\eqref{line:budget_easy_5a}} + \underbrace{(T+1)E }_{\eqref{line:budget_easy_5b}} \\
&+ \underbrace{kT }_{\eqref{line:budget_easy_5c_ii_LB}} + 
\underbrace{kT }_{\eqref{line:budget_easy_5c_ii_UB}} +
\underbrace{T-k}_{\eqref{line:budget_easy_5d_i_i}}  +
\underbrace{T-k}_{\eqref{line:budget_easy_5d_i_ii}}  \\
& +
\underbrace{kT}_{\eqref{line:budget_easy_5d_ii_LB_i}} +
\underbrace{kT}_{\eqref{line:budget_easy_5d_ii_LB_ii}} +
\underbrace{kT}_{\eqref{line:budget_easy_5d_ii_UB_i}} +
\underbrace{kT}_{\eqref{line:budget_easy_5d_ii_UB_ii}} \\
&= 1 + 4E + 5T + 3ET + 6kT\ ,
 \end{align*}
 where we treat \eqref{line:budget_easy_5a}, \eqref{line:budget_easy_5b}, \eqref{line:budget_easy_5c_ii_LB}, \eqref{line:budget_easy_5c_ii_UB}, \eqref{line:budget_easy_5d_ii_LB_i}, \eqref{line:budget_easy_5d_ii_LB_ii} \eqref{line:budget_easy_5d_ii_UB_i}, \eqref{line:budget_easy_5d_ii_LB_ii}, \eqref{line:budget_easy_5d_ii_UB_ii} as hard constraints, and we utilize the same argument as in the proof of Theorem \ref{thm:main} for \eqref{line:budget_easy_5d_i_ii}.
Similar to the proof of Theorem \ref{thm:main}, it follows from Lemma \ref{lem:zeros} that
$$
\|\bar \by\|_0\le 2 | \mathcal I|\le 2\pran{1 + 4E + 5T + 3ET + 6kT},
$$
which finishes the proof. 
\halmos
\end{proof}

\section{Proof of Proposition~\ref{prop:productioninventory_reformulation}}

\begin{proof}{Proof of Proposition~\ref{prop:productioninventory_reformulation}.}
Consider any active set $\mathcal{A}$. We observe for each $s \in \{1,\ldots,T+1\}$ that the quantities
\begin{align*}
    - b_{i,s} + \sum_{t,j: (t,s,j) \in \mathcal{A}} a_{i,t,j} y_{t,s,j}
\end{align*}
in the constraints of the optimization problem~\eqref{prob:ldr_1} are given by %corresponding to the 
\begin{align}
    &0 + \sum_{t,e: (t,s,e) \in \mathcal{A}} c_{te} y_{tse} \label{cons:1a} \tag{CONS-$s$-1a} \\
    &0 + \sum_{t: (t,s,e) \in \mathcal{A}}  y_{tse} && \forall e \in [E] \label{cons:1b} \tag{CONS-$s$-1b} \\
    &0 + y_{tse} && \forall t \in \{s,\ldots,T+1\}, e \in [E]: (t,s,e) \in \mathcal{A} \label{cons:1c-UB} \tag{CONS-$s$-1c-UB}\\
    &0 - y_{tse} && \forall t \in \{s,\ldots,T+1\}, e \in [E]: (t,s,e) \in \mathcal{A}\label{cons:1c-LB} \tag{CONS-$s$-1c-LB}\\
    &0 + \sum_{\ell=1}^t \sum_{e \in [E]: (t,s,e) \in \mathcal{A}} y_{\ell s e} && \forall t \in [T] \quad \left[ \textnormal{if } s = 1 \right]\label{cons:1d-UB-i} \tag{CONS-$s$-1d-UB-i}\\
     &0 - \sum_{\ell=1}^t \sum_{e \in [E]: (t,s,e) \in \mathcal{A}} y_{\ell s e} && \forall t \in [T] \quad \left[ \textnormal{if } s = 1 \right] \label{cons:1d-LB-i} \tag{CONS-$s$-1d-LB-i} \\
    &-1 + \sum_{\ell=s}^t \sum_{e \in [E]: (\ell,s,e) \in \mathcal{A}} y_{\ell se} && \forall t \in \{s,\ldots,T\} \quad \left[ \textnormal{if } s \in \{2,\ldots,T\} \right] \label{cons:1d-UB-ii} \tag{CONS-$s$-1d-UB-ii}\\
    &1 - \sum_{\ell=s}^t \sum_{e \in [E]: (\ell,s,e) \in \mathcal{A}} y_{\ell se} && \forall t \in \{s,\ldots,T\} \quad \left[ \textnormal{if } s \in \{2,\ldots,T\} \right] \label{cons:1d-LB-ii} \tag{CONS-$s$-1d-LB-ii} \\
    &1 && \left[\textnormal{if } s \in \{2,\ldots,T+1 \} \right] \label{cons:1d-UB-iii} \tag{CONS-$s$-1d-UB-iii} \\
     &-1 && \left[\textnormal{if } s \in \{2,\ldots,T+1 \} \right]\label{cons:1d-LB-iii} \tag{CONS-$s$-1d-LB-iii}
\end{align}
We now count the number of unique quantities across the periods $s \in \{1,\ldots,T+1\}$ of the optimization problem.
First, we observe that the number of unique quantities among \eqref{cons:1a}, \eqref{cons:1b}, \eqref{cons:1c-UB}, and \eqref{cons:1c-LB} across periods $s \in \{1,\ldots,T+1\}$ is upper bounded by 
\begin{align}
\sum_{s=1}^{T+1} \left(  \underbrace{1}_{\eqref{cons:1a}} +  \underbrace{\sum_{e \in [E]} 1}_{\eqref{cons:1b}} + \underbrace{  \sum_{t=s}^{T+1} \sum_{e \in [E]: (t,s,e) \in \mathcal{A} }2 }_{\eqref{cons:1c-UB} + \eqref{cons:1c-LB}} \right) &= (E+1) (T+1) + 2 | \mathcal{A}|. \label{line:counting:1}
\end{align}
Second, we observe for each $s \in \{2,\ldots,T\}$ and $t \in \{s,\ldots,T\}$ that \begin{align*}
\left[ (t,s,e) \notin \mathcal{A} \textnormal{ for all } e \in [E]\right] \implies  \left[ \sum_{\ell=s}^{t-1} \sum_{e \in [E]: (\ell,s,e) \in \mathcal{A}} y_{\ell se}  =  \sum_{\ell=s}^{t} \sum_{e \in [E]: (\ell,s,e) \in \mathcal{A}} y_{\ell se}  \right].
\end{align*}
Therefore, we observe that the number of unique quantities among \eqref{cons:1d-UB-i}, \eqref{cons:1d-LB-i}, \eqref{cons:1d-UB-ii}, and \eqref{cons:1d-LB-ii} across periods  $s \in \{1,\ldots,T\}$ is upper bounded by
\begin{align}
   \sum_{s=1}^T \underbrace{\left(2 + \sum_{t \in \{s,\ldots,T\}: \textnormal{there exists } e \in [E] \textnormal{ such that } (t,s,e) \in \mathcal{A}} 2 \right)}_{\substack{\eqref{cons:1d-UB-i} + \eqref{cons:1d-LB-i} \\+ \eqref{cons:1d-UB-ii} +\eqref{cons:1d-LB-ii}}} &\le     \sum_{s=1}^T \left(2 + \sum_{t =s}^T \sum_{e \in [E]: (t,s,e) \in \mathcal{A}} 2 \right) \notag \\
    &\le 2T + 2 | \mathcal{A}|. \label{line:counting:2}
\end{align}
Finally, we observe that the number of unique quantities among  \eqref{cons:1d-UB-iii} and \eqref{cons:1d-LB-iii} is given by 
\begin{align}
\sum_{s=2}^{T+1}  \underbrace{  2 }_{\eqref{cons:1d-UB-iii} + \eqref{cons:1d-LB-iii}}  =  2T.  \label{line:counting:3}
\end{align}
Combining lines~\eqref{line:counting:1}, \eqref{line:counting:2}, and \eqref{line:counting:3}, we conclude for the production-inventory problem that 
\begin{align*}
K^{\mathcal{A},1} + \cdots + K^{\mathcal{A},T+1} &\le \underbrace{(E+1)(T+1) + 2 | \mathcal{A}|}_{\eqref{line:counting:1}} \;+\; \underbrace{ 2T + 2 | \mathcal{A}| }_{\eqref{line:counting:2}}\; + \; \underbrace{2T}_{\eqref{line:counting:3}} \\
&= 4 | \mathcal{A}| + ET + 5T + E + 1. 
\end{align*}

\halmos \end{proof}

\begin{proof}{Proof of Lemma~\ref{lem:reform}.} Consider any  $s \in [T]$ and $i,i' \in \{0,\ldots,m\}$ that satisfy the equality
    \begin{align*}
\left(b_{i,s}, \left( a_{i,t,j} \right)_{t,j: (t,s,j) \in \mathcal{A}}\right) = \left(b_{i',s}, \left( a_{i',t,j} \right)_{t,j: (t,s,j) \in \mathcal{A}}\right).
    \end{align*}
  It follows from the above equality that the following equality holds for all $y_{t,s,j} \in \R$, $(t,s,j) \in \mathcal{A}$:   
    \begin{align*}
    \left( - b_{i,s} + \sum_{t,j: (t,s,j) \in \mathcal{A}} a_{i',t,j} y_{t,s,j} \right) = \left( - b_{i',s} + \sum_{t,j: (t,s,j) \in \mathcal{A}} a_{i,t,j} y_{t,s,j} \right),
    \end{align*}
    and the above equality implies that the following two optimization problems are identical:
     \begin{gather*}
%&\max_{\zeta_s \in \mathcal{U}_s} \left \{ \left( - b_{i,s} + \sum_{t,j: (t,s,j) \in \mathcal{A}} a_{i,t,j} y_{t,s,j} \right) \zeta_s \right \} = 
\left[ \begin{aligned}
&\underset{\bar{\omega}_{i,s},\ubar{\omega}_{i,s} \in \R}{\text{minimize}}&& \bar{D}_s  \bar{\omega}_{i,s}  - \ubar{D}_s \ubar{\omega}_{i,s}  \\
    &\textnormal{subject to}&& \bar{\omega}_{i,s} - \ubar{\omega}_{i,s} =  -b_{i,s} + \sum_{t,j: (t,s,j) \in \mathcal{A}}  a_{i,t,j}  y_{t,s,j} \\
        &&& \bar{\omega}_{i,s}, \ubar{\omega}_{i,s} \ge 0 
    \end{aligned} \right]; \\
    \left[ \begin{aligned}
&\underset{\bar{\omega}_{i',s},\ubar{\omega}_{i',s} \in \R}{\text{minimize}}&& \bar{D}_s  \bar{\omega}_{i',s}  - \ubar{D}_s \ubar{\omega}_{i',s}  \\
    &\textnormal{subject to}&& \bar{\omega}_{i',s} - \ubar{\omega}_{i',s} =  -b_{i',s} + \sum_{t,j: (t,s,j) \in \mathcal{A}}  a_{i',t,j}  y_{t,s,j} \\
        &&& \bar{\omega}_{i',s}, \ubar{\omega}_{i',s} \ge 0 
    \end{aligned} \right]. 
\end{gather*}
Because those two optimization problems are identical, they have the same optimal solution. Said another way, the fact that those two optimization problems are identical implies that there are optimal solutions $(\bar{\omega}^*_{i,s}, \ubar{\omega}^*_{i,s}),(\bar{\omega}^*_{i',s}, \ubar{\omega}^*_{i',s})$ for those two problems that satisfy $(\bar{\omega}^*_{i,s}, \ubar{\omega}^*_{i,s}) = (\bar{\omega}^*_{i',s}, \ubar{\omega}^*_{i',s})$. 
    Hence, we conclude that the optimal objective value of the linear optimization problem~\eqref{prob:P_A_1} would not change if we added the constraint
    $$(\bar{\omega}_{i,s}, \ubar{\omega}_{i,s}) = (\bar{\omega}_{i',s}, \ubar{\omega}_{i',s}).$$
    Because the above reasoning can be applied to all $s \in [T]$ and all $i,i' \in \{0,\ldots,m\}$ that map to the same tuple in $\mathcal{K}^{\mathcal{A},s}$, our proof of Lemma~\ref{lem:reform} is complete. 
\halmos \end{proof}
\section{Proof of Propositions~\ref{prop:terminate} and \ref{prop:monotone} in Section \ref{sec:algorithm:activeset}} \label{appx:D}
Our proofs of Propositions~\ref{prop:terminate} and \ref{prop:monotone} and Lemma~\ref{lem:D_A_feas} make use of several intermediary observations and results. First,  we observe using the robust counterpart technique that the dual linear optimization reformulation of \eqref{prob:ldr} is stated as follows:
 \begin{equation} \label{prob:D} \tag{D}
\begin{aligned}
&\underset{ \substack{
\lambda_0,\ldots,\lambda_m \in \R, \\
\zeta_{0,s},\ldots,\zeta_{m,s} \in \R \; \forall s \in [T]}}{\textnormal{maximize}}&& - \sum_{i=1}^m c_i \lambda_i - \sum_{i=0}^m \sum_{s=1}^T  b_{i,s} \zeta_{i,s}   \\
&\textnormal{subject to}&& \begin{aligned}[t]
&\sum_{i=0}^m   a_{i,t,j}  \zeta_{i,s}= 0 && \forall 1 \le s \le t \le T, j \in [n] \\
&  \ubar{D}_s \lambda_i  \le \zeta_{i,s} \le \bar{D}_s \lambda_i && \forall i \in \{0,\ldots,m\}, \; s \in [T]\\
& \lambda_0 = 1 \\
&\lambda_i \ge 0 && \forall i \in \{1,\ldots,m\}.  
\end{aligned}
\end{aligned} 
\end{equation} 
The optimal objective value of \eqref{prob:D} is equal to the optimal objective value of \eqref{prob:ldr}, which implies that the optimal objective value of~\eqref{prob:D} is always less than or equal to the optimal objective value of \eqref{prob:D_A}. Moreover, we observe that the dual of the linear optimization problem~\eqref{prob:P_A_1} is given by
 \begin{equation} \label{prob:D_A_1} \tag{D-$\mathcal{A}$'}
\begin{aligned}
&\underset{ \substack{
\lambda_0,\ldots,\lambda_m \in \R, \\
\zeta_{0,s},\ldots,\zeta_{m,s} \in \R \; \forall s \in [T]}}{\textnormal{maximize}}&& - \sum_{i=1}^m c_i \lambda_i - \sum_{i=0}^m \sum_{s=1}^T  b_{i,s} \zeta_{i,s}   \\
&\textnormal{subject to}&& \begin{aligned}[t]
&\sum_{i=0}^m   a_{i,t,j}  \zeta_{i,s}= 0 && \forall (t,s,j) \in \mathcal{A} \\
&  \ubar{D}_s \lambda_i  \le \zeta_{i,s} \le \bar{D}_s \lambda_i && \forall i \in \{0,\ldots,m\}, \; s \in [T]\\
& \lambda_0 = 1 \\
&\lambda_i \ge 0 && \forall i \in \{1,\ldots,m\}.  
\end{aligned}
\end{aligned} 
\end{equation} 
Finally, for the sake of convenience, we repeat the optimization problem~\eqref{prob:D_A} below:
\begin{equation} \tag{\ref{prob:D_A}}
\begin{aligned}
&\underset{ \substack{
\lambda_0,\ldots,\lambda_m \in \R, \\
\zeta^{\mathcal{A},s}_{k,s} \in \R \; \forall s \in [T], k \in [K^{\mathcal{A},s}]}}{\textnormal{maximize}}&& - \sum_{i=1}^m c_i \lambda_i -  \sum_{s=1}^T  \sum_{k=1}^{K^{\mathcal{A},s}} b^{\mathcal{A},s}_{k,s} \zeta^{\mathcal{A},s}_{k,s}  \\
&\textnormal{subject to}&& \begin{aligned}[t]
&\sum_{k=1}^{K^{\mathcal{A},s}}  a^{\mathcal{A},s}_{k,t,j}  \zeta^{\mathcal{A},s}_{k,s}   = 0 && \forall (t,s,j) \in \mathcal{A} \\
&  \ubar{D}_s \left( \sum_{i \in \{0,\ldots,m\}: \pi^{\mathcal{A},s}(i) = k} \lambda_{i} \right)   \le \zeta^{\mathcal{A},s}_{k,s}  \le \bar{D}_s \left( \sum_{i \in \{0,\ldots,m\}: \pi^{\mathcal{A},s}(i) = k} \lambda_{i} \right)  && \forall  s \in [T], k \in [ K^{\mathcal{A},s}]\\
& \lambda_0 = 1 \\
&\lambda_i \ge 0 && \forall i \in \{1,\ldots,m\}. 
\end{aligned}
\end{aligned} 
\end{equation}

In the following lemma, we provide a transformation of a feasible solution of the optimization problem \eqref{prob:D_A} into a (possibly-infeasible) solution  for the optimization problem~\eqref{prob:D} with the same objective value. 
\begin{lemma} \label{lem:transformation}
Consider any feasible solution of \eqref{prob:D_A}. For each $i \in \{0,\ldots,m\}$ and $s \in [T]$, let
\begin{align*}
\zeta_{i,s} &\triangleq 
\left( \frac{\lambda_i}{\sum_{i' \in \{0,\ldots,m\}: \pi^{\mathcal{A},s}(i') = \pi^{\mathcal{A},s}(i)} \lambda_{i'}} \right) \zeta^{\mathcal{A},s}_{\pi^{\mathcal{A},s}(i),s}
\end{align*}
where we define $0/0$ to be equal to $0$. 
Then the following conditions are satisfied:
\begin{subequations}
\begin{align}
 - \sum_{i=1}^m c_i \lambda_i - \sum_{i=0}^m \sum_{s=1}^T  b_{i,s} \zeta_{i,s}  & = - \sum_{i=1}^m c_i \lambda_i -  \sum_{s=1}^T  \sum_{k=1}^{K^{\mathcal{A},s}} b^{\mathcal{A},s}_{k,s} \zeta^{\mathcal{A},s}_{k,s} \label{line:obj_same} \\
\sum_{i=0}^m   a_{i,t,j}  \zeta_{i,s}&= 0 && \forall (t,s,j) \in \mathcal{A} \label{line:cons:lem:1}\\
  \ubar{D}_s \lambda_i  \le \zeta_{i,s} &\le \bar{D}_s \lambda_i && \forall i \in \{0,\ldots,m\}, \; s \in [T] \label{line:cons:lem:2}\\
 \lambda_0& = 1\label{line:cons:lem:3} \\
\lambda_i & \ge 0 && \forall i \in \{1,\ldots,m\}.\label{line:cons:lem:4}  
\end{align}
\end{subequations}
\end{lemma}
Before presenting the proof of the above  Lemma~\ref{lem:transformation}, let us offer an interpretation of its conditions. 
Condition~\eqref{line:obj_same} shows that the transformation from the above lemma yields a solution for the optimization problem~\eqref{prob:D} with the same objective value as the feasible solution for \eqref{prob:D_A}. Conditions~\eqref{line:cons:lem:1}-\eqref{line:cons:lem:4} show that the transformation yields a solution that satisfies all of the constraints of the optimization problem~\eqref{prob:D} with the possible exception of equality constraints of the form $\sum_{i=0}^m   a_{i,t,j}  \zeta_{i,s} = 0$ for some subset of tuples $ (t,s,j) \notin \mathcal{A}$. Having interpreted those conditions, we now present the proof of Lemma~\ref{lem:transformation}.

\begin{proof}{Proof of Lemma~\ref{lem:transformation}.}
Consider any feasible solution of \eqref{prob:D_A}. For each $i \in \{0,\ldots,m\}$ and $s \in [T]$, let
\begin{align}
\zeta_{i,s} &\triangleq 
\left( \frac{\lambda_i}{\sum_{i' \in \{0,\ldots,m\}: \pi^{\mathcal{A},s}(i') = \pi^{\mathcal{A},s}(i)} \lambda_{i'}} \right) \zeta^{\mathcal{A},s}_{\pi^{\mathcal{A},s}(i),s} \label{line:transformation_proof_zeta}
\end{align}
where we define $0/0$ to be equal to $0$. 

We first show that condition~\eqref{line:obj_same} holds. Indeed, 
\begin{align*}
 - \sum_{i=1}^m c_i \lambda_i - \sum_{i=0}^m \sum_{s=1}^T  b_{i,s} \zeta_{i,s} &=\sum_{i=1}^m c_i \lambda_i - \sum_{s=1}^T \sum_{k=1}^{K^{\mathcal{A},s}}  \sum_{i \in \{0,\ldots,m\}: \pi^{\mathcal{A},s}(i) = k}  b_{i,s} \zeta_{i,s}  \\
 &=\sum_{i=1}^m c_i \lambda_i - \sum_{s=1}^T \sum_{k=1}^{K^{\mathcal{A},s}}  b^{\mathcal{A},s}_{k,s} \sum_{i \in \{0,\ldots,m\}: \pi^{\mathcal{A},s}(i) = k} \zeta_{i,s} \\
  &=\sum_{i=1}^m c_i \lambda_i - \sum_{s=1}^T \sum_{k=1}^{K^{\mathcal{A},s}}  b^{\mathcal{A},s}_{k,s} \sum_{i \in \{0,\ldots,m\}: \pi^{\mathcal{A},s}(i) = k} \left( \frac{\lambda_i}{\sum_{i' \in \{0,\ldots,m\}: \pi^{\mathcal{A},s}(i') = \pi^{\mathcal{A},s}(i)} \lambda_{i'}} \right) \zeta^{\mathcal{A},s}_{\pi^{\mathcal{A},s}(i),s} \\
    &=\sum_{i=1}^m c_i \lambda_i - \sum_{s=1}^T \sum_{k=1}^{K^{\mathcal{A},s}}  b^{\mathcal{A},s}_{k,s} \sum_{i \in \{0,\ldots,m\}: \pi^{\mathcal{A},s}(i) = k} \left( \frac{\lambda_i}{\sum_{i' \in \{0,\ldots,m\}: \pi^{\mathcal{A},s}(i') =  k} \lambda_{i'}} \right) \zeta^{\mathcal{A},s}_{k,s} \\
        &=\sum_{i=1}^m c_i \lambda_i - \sum_{s=1}^T \sum_{k=1}^{K^{\mathcal{A},s}}  b^{\mathcal{A},s}_{k,s} \zeta^{\mathcal{A},s}_{k,s}. 
\end{align*}
The first and second equalities follow from algebra. The third equality follows from line~\eqref{line:transformation_proof_zeta}. The fourth and fifth equalities follow from algebra. 

We first show that condition~\eqref{line:cons:lem:1} holds. Indeed,  we observe for each $(t,s,j) \in \mathcal{A}$ that 
\begin{align*}
\sum_{i=0}^m a_{i,t,j} \zeta_{i,s} &= \sum_{k=1}^{K^{\mathcal{A},s}} \sum_{i \in \{0,\ldots,m\}: \pi^{\mathcal{A},s}(i) = k}  a_{i,t,j} \zeta_{i,s}\\
&=  \sum_{k=1}^{K^{\mathcal{A},s}}a^{\mathcal{A},s}_{k,t,j} \sum_{i \in \{0,\ldots,m\}: \pi^{\mathcal{A},s}(i) = k} \zeta_{i,s}\\
&=  \sum_{k=1}^{K^{\mathcal{A},s}}a^{\mathcal{A},s}_{k,t,j} \sum_{i \in \{0,\ldots,m\}: \pi^{\mathcal{A},s}(i) = k} \left( \frac{\lambda_i}{\sum_{i' \in \{0,\ldots,m\}: \pi^{\mathcal{A},s}(i') = \pi^{\mathcal{A},s}(i)} \lambda_{i'}} \right) \zeta^{\mathcal{A},s}_{\pi^{\mathcal{A},s}(i),s}\\
&=  \sum_{k=1}^{K^{\mathcal{A},s}}a^{\mathcal{A},s}_{k,t,j} \sum_{i \in \{0,\ldots,m\}: \pi^{\mathcal{A},s}(i) = k} \left( \frac{\lambda_i}{\sum_{i' \in \{0,\ldots,m\}: \pi^{\mathcal{A},s}(i') = k} \lambda_{i'}} \right) \zeta^{\mathcal{A},s}_{k,s}\\
&=  \sum_{k=1}^{K^{\mathcal{A},s}}a^{\mathcal{A},s}_{k,t,j} \zeta^{\mathcal{A},s}_{k,s}\\
&= 0. 
\end{align*}
The first and second equalities follow from algebra. The third equality follows from line~\eqref{line:transformation_proof_zeta}. The fourth and fifth equalities follow from algebra. The sixth equality follows from the fact that we are considering a feasible solution of the linear optimization problem~\eqref{prob:D_A}. 

We next show that condition~\eqref{line:cons:lem:2} holds. Indeed, we observe for each $i \in \{0,\ldots,m\}$ and $s \in [T]$ that
\begin{align*}
\zeta_{i,s} &= \left( \frac{\lambda_i}{\sum_{i' \in \{0,\ldots,m\}: \pi^{\mathcal{A},s}(i') = \pi^{\mathcal{A},s}(i)} \lambda_{i'}} \right) \zeta^{\mathcal{A},s}_{\pi^{\mathcal{A},s}(i),s} \\
&\in \left[ \left( \frac{\lambda_i}{\sum_{i' \in \{0,\ldots,m\}: \pi^{\mathcal{A},s}(i') = \pi^{\mathcal{A},s}(i)} \lambda_{i'}} \right) \ubar{D}_s \left( \sum_{i \in \{0,\ldots,m\}: \pi^{\mathcal{A},s}(i) = k} \lambda_{i} \right) , \right.  \\
&\phantom{\in \Big[}\;  \left.   \left( \frac{\lambda_i}{\sum_{i' \in \{0,\ldots,m\}: \pi^{\mathcal{A},s}(i') = \pi^{\mathcal{A},s}(i)} \lambda_{i'}} \right) \bar{D}_s \left( \sum_{i \in \{0,\ldots,m\}: \pi^{\mathcal{A},s}(i) = k} \lambda_{i} \right) \right] \\
&= \left[ \ubar{D}_s \lambda_i, \bar{D}_s,\lambda_i \right]. 
\end{align*}
The first equality follows from line~\eqref{line:transformation_proof_zeta}. The inclusion follows from the fact that  we are considering a feasible solution of the linear optimization problem~\eqref{prob:D_A}. The second equality follows from algebra. 

Because  conditions~\eqref{line:cons:lem:3} and \eqref{line:cons:lem:4} follow from the  fact that  we are considering a feasible solution of the linear optimization problem~\eqref{prob:D_A}, our proof of Lemma~\ref{lem:transformation} is complete. 
\halmos \end{proof}
Equipped with Lemma~\ref{lem:transformation}, we are now ready for the proofs of Propositions~\ref{prop:terminate} and \ref{prop:monotone} and Lemma~\ref{lem:D_A_feas}.
\begin{proof}{Proof of Proposition~\ref{prop:terminate}.}
Consider any optimal solution for \eqref{prob:D_A}, and suppose that the solution satisfies
\begin{align}
 \sum_{i=0}^m   a_{i,t,j} \left( \frac{\lambda_i}{\sum_{i' \in \{0,\ldots,m\}: \pi^{\mathcal{A},s}(i') = \pi^{\mathcal{A},s}(i)} \lambda_{i'}} \right) \zeta^{\mathcal{A},s}_{\pi^{\mathcal{A},s}(i),s} = 0 \quad \forall (t,s,j) \notin \mathcal{A}. \tag{\ref{line:terminate}}
 \end{align} 
 Now, for each $i \in \{0,\ldots,m\}$ and $s \in [T]$, let
\begin{align*}
\zeta_{i,s} &\triangleq 
\left( \frac{\lambda_i}{\sum_{i' \in \{0,\ldots,m\}: \pi^{\mathcal{A},s}(i') = \pi^{\mathcal{A},s}(i)} \lambda_{i'}} \right) \zeta^{\mathcal{A},s}_{\pi^{\mathcal{A},s}(i),s}. 
\end{align*}
In that case, it follows from \eqref{line:terminate} that 
\begin{align}
 \sum_{i=0}^m   a_{i,t,j} \zeta_{i,s} = 0 \quad \forall (t,s,j) \notin \mathcal{A}. \label{line:terminate:reprise}
 \end{align} 
 Therefore, it follows from line~\eqref{line:terminate:reprise} and Lemma~\ref{lem:transformation} that there exists a feasible solution for the optimization problem~\eqref{prob:D} with the same objective value as the optimal objective value of \eqref{prob:D_A}. Hence, we have shown that the optimal objective value of \eqref{prob:D} is equal to the optimal objective value of \eqref{prob:D_A}. Because the optimal objective value of \eqref{prob:D_A} is equal to the optimal objective value of \eqref{prob:ldr}, and because the optimal objective value of \eqref{prob:D} is equal to the optimal objective value of \eqref{prob:ldr}, our proof of Proposition~\ref{prop:terminate} is complete. 
\halmos \end{proof}
\begin{proof}{Proof of Proposition~\ref{prop:monotone}.}
Let $\hat{\lambda}_0,\ldots,\hat{\lambda}_m \in \R$ and $
\hat{\zeta}^{\mathcal{A},s}_{k,s} \in \R \; \forall s \in [T], k \in [K^{\mathcal{A},s}]$ denote an optimal solution  for the optimization problem~\eqref{prob:D_A}. Therefore, it follows from Lemma~\ref{lem:transformation} and from the fact that the optimal objective value of \eqref{prob:D_A} is equal to the optimal objective value of \eqref{prob:D_A_1} that an optimal solution for the optimization problem~\eqref{prob:D_A_1} is given by
$\bar{\lambda}_0,\ldots,\bar{\lambda}_m \in \R$ and $
\bar{\zeta}_{i,s} \in \R \; \forall s \in [T], k \in [K^{\mathcal{A},s}]$, where
\begin{align*}
\bar{\lambda}_i &\triangleq \hat{\lambda}_i && \forall i \in \{0,\ldots,m\};\\
\bar{\zeta}_{i,s} &\triangleq 
\left( \frac{\hat{\lambda}_i}{\sum_{i' \in \{0,\ldots,m\}: \pi^{\mathcal{A},s}(i') = \pi^{\mathcal{A},s}(i)} \hat{\lambda}_{i'}} \right) \hat{\zeta}^{\mathcal{A},s}_{\pi^{\mathcal{A},s}(i),s} && \forall i \in \{0,\ldots,m\},s \in \{1,\ldots,T\}.
\end{align*}
With the above notation, let us define
\begin{align*}
  \mathcal{A}^{=} &\triangleq \left \{(t,s,j) \notin \mathcal{A}:  \quad  \sum_{i=0}^m   a_{i,t,j} \left( \frac{\hat{\lambda}_i}{\sum_{i' \in \{0,\ldots,m\}: \pi^{\mathcal{A},s}(i') = \pi^{\mathcal{A},s}(i)} \hat{\lambda}_{i'}} \right) \hat{\zeta}^{\mathcal{A},s}_{\pi^{\mathcal{A},s}(i),s} = 0 \right \} \\
  &= \left \{(t,s,j) \notin \mathcal{A}:  \quad  \sum_{i=0}^m   a_{i,t,j} \bar{\zeta}_{i,s}= 0 \right \}.
\end{align*}
If $\widetilde{\mathcal{A}} \subseteq \mathcal{A}^=$, then we observe that  $\bar{\lambda}_0,\ldots,\bar{\lambda}_m \in \R$ and $
\bar{\zeta}_{i,s} \in \R \; \forall s \in [T], k \in [K^{\mathcal{A},s}]$ must also be  an optimal solution for the linear optimization problem \textnormal{(D-$\mathcal{A} \cup \widetilde{\mathcal{A}}$')}. Since the optimal objective value of \textnormal{(D-$\mathcal{A} \cup \widetilde{\mathcal{A}}$')} is equal to the optimal objective value of \textnormal{(LDR-$\mathcal{A} \cup \widetilde{\mathcal{A}}$)}, our proof is complete.   
\halmos \end{proof}

\section{Practical Improvements to Algorithm \ref{al:active_set}} \label{sec:algorithm:improvements}
We summarize here several improvements and omitted details for our active set method from \S\ref{sec:algorithm:activeset}. 

\subsection{Efficiently solving \eqref{prob:D_A} by removing redundant constraints}\label{app:solveDA} 
In Step 1 of each iteration of our algorithm, we must solve the linear optimization problem~\eqref{prob:D_A}. To improve the practical efficiency of solving this linear optimization problem~\eqref{prob:D_A}, we can remove inequality constraints that are not going to be binding at optimality. In particular, we remove constraints from \eqref{prob:D_A} using the following lemma:\looseness=-1
\begin{lemma} \label{lem:remove_unnecessary}
Define the following set of tuples:
$$\mathcal{C} \triangleq \left \{(k,s): s \in [T], k \in [K^{\mathcal{A},s}], \left({b}^{\mathcal{A},s}_{ k,s}, \left( a^{\mathcal{A},s}_{k,t,j} \right)_{t,j: (t,s,j) \in \mathcal{A}}\right) = \left(0,(0)_{t,j: (t,s,j) \in \mathcal{A}} \right)  \right \}. $$
Then the following constraints can be simultaneously removed from the linear optimization problem~\eqref{prob:D_A} without changing the optimal objective value:
\begin{align}
    &  \ubar{D}_s \left( \sum_{i \in \{0,\ldots,m\}: \pi^{\mathcal{A},s}(i) = k} \lambda_{i} \right)   \le \zeta^{\mathcal{A},s}_{k,s}  \le \bar{D}_s \left( \sum_{i \in \{0,\ldots,m\}: \pi^{\mathcal{A},s}(i) = k} \lambda_{i} \right)  \quad  \forall  (k,s) \in \mathcal{C}.  \label{line:remove_unnecessary} 
\end{align}
\end{lemma}
\begin{proof}{Proof.}
Suppose for the sake of developing a contradiction that the optimal objective value of the linear optimization problem~\eqref{prob:D_A} is changed after removing the constraints from line~\eqref{line:remove_unnecessary}. Then it follows from strong duality for linear optimization that there must exist an optimal solution for \eqref{prob:P_A} and a corresponding tuple $(k,s) \in \mathcal{C}$ such that the optimal solution either satisfies $\bar{\omega}^{\mathcal{A},s}_{k,s} > 0$ or satisfies $\ubar{\omega}^{\mathcal{A},s}_{k,s}> 0$, but not both. However, every   feasible solution of the linear optimization problem~\eqref{prob:P_A} must satisfy  the following constraints. 
\begin{align*}
\begin{aligned}
        &&&\bar{\omega}^{\mathcal{A},s}_{k,s}  - \ubar{\omega}^{\mathcal{A},s}_{k,s}  =   - b^{\mathcal{A},s}_{k} + \sum_{t,j: (t,s,j) \in \mathcal{A}} a_{k,t,j}^{\mathcal{A},s} y_{t,s,j} && \forall (k,s) \in \mathcal{C}  \\
    &&&  \bar{\omega}^{\mathcal{A},s}_{k,s} , \ubar{\omega}^{\mathcal{A},s}_{k,s}  \ge 0 &&  \forall (k,s) \in \mathcal{C}
    \end{aligned}
\end{align*}
Since the equality $ ({b}^{\mathcal{A},s}_{ k,s}, ( a^{\mathcal{A},s}_{k,t,j} )_{t,j: (t,s,j) \in \mathcal{A}}) = (0,(0)_{t,j: (t,s,j) \in \mathcal{A}} )$ holds for all $(k,s) \in \mathcal{C}$, it follows from the above reasoning that every feasible solution of the linear optimization problem~\eqref{prob:P_A} must satisfy the equality $\bar{\omega}^{\mathcal{A},s}_{k,s} = \ubar{\omega}^{\mathcal{A},s}_{k,s} $ for all $(k,s) \in \mathcal{C}$. Therefore, we conclude that there must exist an optimal solution of the linear optimization problem~\eqref{prob:P_A} and a corresponding tuple $(k,s) \in \mathcal{C}$ such that the optimal solution  satisfies $\bar{\omega}^{\mathcal{A},s}_{k,s}, \ubar{\omega}^{\mathcal{A},s}_{k,s} > 0$. We have thus obtained a contradiction, which completes the proof of Lemma~\ref{lem:remove_unnecessary}. 
\halmos \end{proof}
% \subsubsection{Updating the active set.}\label{sec:algorithm:add_to_active_set}

% \subsubsection{Initial active set.}  \label{sec:improvements:markov} In the first iteration of our algorithm, we use  the following initial active set:
% \begin{align*}
% \mathcal{A} &= \left \{ (t,1,j): t \in [T], j \in [n] \right \} \cup  \left \{ (t,t,j): t \in [T], j \in [n] \right \}. 
% \end{align*} 
% We refer to the initial active set defined above as the \emph{Markovian active set}. In the case of the Markovian active set,   the optimization problem~\eqref{prob:ldr_A} will yield linear decision rules where the decisions in each period $t \in [T]$ have the form 
% $\bx_{t} = \by_{t1}  + \by_{tt} \zeta_t$. 

\subsection{Addressing infeasibility }  \label{sec:algorithm:infeas}
In Step 1 of our algorithm in  \S\ref{sec:algorithm:activeset}, we assumed that the optimal objective value of the linear optimization problem~\eqref{prob:D_A} was finite. Here we propose a simple modification to Step 1 to address settings in which the optimal objective value of \eqref{prob:D_A} is not finite.\looseness=-1

To motivate our modification, we begin by making the following observation.
\begin{lemma}\label{lem:D_A_feas}
If Assumption~\ref{ass:1} holds, then \eqref{prob:D_A} always has a feasible solution. 
\end{lemma}

\begin{proof}{Proof.}
We recall from  Assumption~\ref{ass:1} that the optimal objective value of \eqref{prob:ldr} is finite.  Moreover, we observe that the optimal objective value of \eqref{prob:ldr} is equal to the optimal objective value of \eqref{prob:D}, which implies that the optimal objective value of \eqref{prob:ldr} is finite. Finally, we observe that the feasible set of the linear optimization problem~\eqref{prob:D_A} for any active set $\mathcal{A}$ is a superset of the feasible set for the linear optimization problem~\eqref{prob:D}. Hence, we conclude that \eqref{prob:D_A} always has at least one feasible solution. 
\halmos \end{proof}

The above Lemma~\ref{lem:D_A_feas} shows that the optimal objective value of \eqref{prob:D_A} is not finite if and only if the optimal objective value of \eqref{prob:D_A} is unbounded. Therefore, for every active set $\mathcal{A}$, we observe that there  always exists a sufficiently large constant $M \ge 0$ such that the optimal objective value of the following optimization problem is finite:
{\small
\begin{equation}  \label{prob:D_A_M} \tag{D-$\mathcal{A}$-$M$}
\begin{aligned}
&\underset{ \substack{
\lambda_0,\ldots,\lambda_m \in \R, \\
\zeta^{\mathcal{A},s}_{k,s} \in \R \; \forall s \in [T], k \in [K^{\mathcal{A},s}]}}{\textnormal{maximize}}&& - \sum_{i=1}^m c_i \lambda_i -  \sum_{s=1}^T  \sum_{k=1}^{K^{\mathcal{A},s}} b^{\mathcal{A},s}_{k,s} \zeta^{\mathcal{A},s}_{k,s}  \\
&\textnormal{subject to}&& \begin{aligned}[t]
&\sum_{k=1}^{K^{\mathcal{A},s}}  a^{\mathcal{A},s}_{k,t,j}  \zeta^{\mathcal{A},s}_{k,s}   = 0 && \forall (t,s,j) \in \mathcal{A} \\
&  \ubar{D}_s \left( \sum_{i \in \{0,\ldots,m\}: \pi^{\mathcal{A},s}(i) = k} \lambda_{i} \right)   \le \zeta^{\mathcal{A},s}_{k,s}  \le \bar{D}_s \left( \sum_{i \in \{0,\ldots,m\}: \pi^{\mathcal{A},s}(i) = k} \lambda_{i} \right)  && \forall  s \in [T], k \in [ K^{\mathcal{A},s}]\\
& \lambda_0 = 1 \\
&0 \le \lambda_i \le M && \forall i \in \{1,\ldots,m\}. 
\end{aligned}
\end{aligned} 
\end{equation}
}%

In view of the above observation, we now present our simple modification to Step 1 to address settings in which the optimal objective value of \eqref{prob:D_A} is not finite. In our modification of Step 1, we first solve the linear optimization problem~\eqref{prob:D_A}. If the optimal objective value of \eqref{prob:D_A} is finite, then we proceed to Step 2. Otherwise, we solve \eqref{prob:D_A_M} with greater and greater values of $M \ge 0$ until the linear optimization problem~\eqref{prob:D_A_M}  has a feasible solution, at which point we proceed to Step 2. This modification thus ensures that the algorithm can eventually reach Step 3 (in which new tuples will be added to the active set) regardless of the choice of the active set at the beginning of each iteration. 

\subsection{Evaluating the termination condition} \label{sec:evaluating_termination_condition} In Step 2 of each iteration of our algorithm, we decide whether to terminate the algorithm by determining whether the optimal solution obtained from solving \eqref{prob:D_A} satisfies line~\eqref{line:terminate}. In the following Lemma~\ref{lem:evaluate_terminate}, we show that line~\eqref{line:terminate} can be evaluated efficiently, both with respect to computation time and computer memory. Note that the following lemma makes use of the quantity $Z \triangleq \sum_{i=0}^m \sum_{t=1}^T \sum_{s=1}^t \sum_{j=1}^n \mathbb{I} \{ a_{i,t,j} \neq 0 \}$, which is defined here as the total number of nonzeros among the vectors $\ba_{i,t} \in \R^n$ for $i \in \{0,\ldots,m\}$ and $t \in [T]$ in the robust optimization problem. 
\begin{lemma} \label{lem:evaluate_terminate}
The quantities $\sum_{i=0}^m   a_{i,t,j} \left( \frac{\lambda_i}{\sum_{i' \in \{0,\ldots,m\}: \pi^{\mathcal{A},s}(i') = \pi^{\mathcal{A},s}(i)} \lambda_{i'}} \right) \zeta^{\mathcal{A},s}_{\pi^{\mathcal{A},s}(i),s}$  for each $(t,s,j) \notin \mathcal{A}$ can be computed in a total of $\mathcal{O}(T(Z + n K^{\mathcal{A}}))$ time and  $\mathcal{O}(Z + nT^2)$ space. 
\end{lemma}

\begin{proof}{Proof.}
To efficiently determine whether line~\eqref{line:terminate} is satisfied, let the matrix  $\bba_i \triangleq (\ba_{i,1}^\intercal,\ldots,\ba_{i,T}^\intercal) \in \R^{T \times n}$ be defined for each $i \in \{0,\ldots,m\}$. We observe from algebra that the following equalities hold for each $(t,s,j) \notin \mathcal{A}$:
\begin{align*}
     \sum_{i=0}^m   a_{i,t,j} \left( \frac{\lambda_i}{\sum_{i' \in \{0,\ldots,m\}: \pi^{\mathcal{A},s}(i') = \pi^{\mathcal{A},s}(i)} \lambda_{i'}} \right) \zeta^{\mathcal{A},s}_{\pi^{\mathcal{A},s}(i),s}  
              &= \sum_{k=1}^{K^{ \mathcal{A},s}} \underbrace{\left( \frac{ \zeta^{\mathcal{A},s}_{k,s}}{\sum_{i \in \{0,\ldots,m\}: \pi^{\mathcal{A},s}(i) =k} \lambda_{i}}\right)}_{\alpha_{k,s}^{\mathcal{A},s}} \left(  \sum_{i \in \{0,\ldots,m\}: \pi^{\mathcal{A},s}(i) = k}   a_{i,t,j} \lambda_i \right)  \\
                  &= \sum_{k=1}^{K^{ \mathcal{A},s}} \alpha_{k,s}^{\mathcal{A},s} \left(  \sum_{i \in \{0,\ldots,m\}: \pi^{\mathcal{A},s}(i) = k}   a_{i,t,j} \lambda_i \right)  \\
                    &= \sum_{k=1}^{K^{ \mathcal{A},s}} \alpha_{k,s}^{\mathcal{A},s}\left[\sum_{i \in \{0,\ldots,m\}: \pi^{\mathcal{A},s}(i) = k}  \lambda_i \bba_{i} \right]_{t,j}.
 \alpha^{\mathcal{A},s}_{k,s} \beta^{\mathcal{A},s}_{k,t,j}. 
\end{align*}
Therefore, determining whether line~\eqref{line:terminate} is satisfied can be split into the following three parts.
\begin{enumerate}[(a)]
\item We compute the quantity $\alpha_{k,s}^{\mathcal{A},s}$ for each $s \in [T]$ and $k \in [K^{\mathcal{A},s}]$.
\item For each period $s \in [T]$:
\begin{enumerate}[(i)]
\item We compute the sparse matrix $\sum_{i \in \{0,\ldots,m\}: \pi^{\mathcal{A},s}(i) = k}  \lambda_i \bba_{i}$  for each $k \in [K^{\mathcal{A},s}]$. 
\item We compute $\sum_{k=1}^{K^{ \mathcal{A},s}} \alpha_{k,s}^{\mathcal{A},s}\left[\sum_{i \in \{0,\ldots,m\}: \pi^{\mathcal{A},s}(i) = k}  \lambda_i \bba_{i} \right]_{t,j}$ for each $t,j$ such that $(t,s,j) \notin \mathcal{A}$. 
\end{enumerate}
\end{enumerate}
We observe that the computation time and space for performing part (a) is $\mathcal{O}(K^{\mathcal{A}})$.  We observe that each iteration of part (b-i) requires a total of $\mathcal{O}(Z)$ time and space to compute and store the sparse matrices $\sum_{i \in \{0,\ldots,m\}: \pi^{\mathcal{A},s}(i) = k}  \lambda_i \bba_{i}$  for each $k \in [K^{\mathcal{A},s}]$. Since the memory for storing the sparse matrices can be reused across iterations of part (b-i), and since part (b-i) will be performed in $T$ iterations, we conclude that a total of $\mathcal{O}(TZ)$ time and $\mathcal{O}(Z)$ space is needed to perform part (b-i) across all iterations.  Finally, we observe  part (b-ii) for each $(t,s,j) \notin \mathcal{A}$ requires $\mathcal{O}(K^{\mathcal{A},s})$ time and $\mathcal{O}(1)$ space. Since part (b-ii) will be performed once for each $(t,s,j) \notin \mathcal{A}$, we observe that the total computation time for performing part (b-ii) across all iterations is 
\begin{align*}
    \mathcal{O} \left(\sum_{(t,s,j) \notin \mathcal{A}} K^{\mathcal{A},s} \right) =    \mathcal{O} \left(\sum_{j=1}^n \sum_{s=1}^T \sum_{t=s}^T  K^{\mathcal{A},s} \right) =  \mathcal{O} \left(\sum_{j=1}^n \sum_{s=1}^T T  K^{\mathcal{A},s} \right)  =   \mathcal{O}\left(nT K^{\mathcal{A}} \right),
\end{align*}
and the total space for performing part (b-ii) across all iterations is
\begin{align*}
    \mathcal{O} \left(\sum_{(t,s,j) \notin \mathcal{A}} K^{\mathcal{A},s} \right)  = \mathcal{O}(n T^2). 
\end{align*}
Combining the above analysis, we conclude that the total computation time for computing the quantities $\sum_{i=0}^m   a_{i,t,j} \left( \frac{\lambda_i}{\sum_{i' \in \{0,\ldots,m\}: \pi^{\mathcal{A},s}(i') = \pi^{\mathcal{A},s}(i)} \lambda_{i'}} \right) \zeta^{\mathcal{A},s}_{\pi^{\mathcal{A},s}(i),s}$  for each $(t,s,j) \notin \mathcal{A}$ is
\begin{align*}
    \mathcal{O}\left(\underbrace{K^{\mathcal{A}}}_{\textnormal{(a)}} +  \underbrace{TZ}_{\textnormal{(b-i)}} +   \underbrace{nT K^{\mathcal{A}}}_{\textnormal{(b-ii)}}\right) = \mathcal{O} \left( T \left( Z + n K^{\mathcal{A}} \right) \right),
\end{align*}
and the total space for computing the quantities $\sum_{i=0}^m   a_{i,t,j} \left( \frac{\lambda_i}{\sum_{i' \in \{0,\ldots,m\}: \pi^{\mathcal{A},s}(i') = \pi^{\mathcal{A},s}(i)} \lambda_{i'}} \right) \zeta^{\mathcal{A},s}_{\pi^{\mathcal{A},s}(i),s}$  for each $(t,s,j) \notin \mathcal{A}$ is
\begin{align*}
    \mathcal{O}\left(\underbrace{K^{\mathcal{A}}}_{\textnormal{(a)}} +  \underbrace{Z}_{\textnormal{(b-i)}} +   \underbrace{nT^2}_{\textnormal{(b-ii)}}\right) = \mathcal{O} \left( Z + nT^2\right).
\end{align*}
Our proof of Lemma~\ref{lem:evaluate_terminate} is complete. 
\halmos \end{proof}

In the production-inventory problem, we observe that $Z = \Theta(ET^2)$ due to constraint~\eqref{prob:example1:d}. Therefore, it follows from Proposition~\ref{prop:productioninventory_reformulation} that Lemma~\ref{lem:evaluate_terminate} requires a total of $\mathcal{O}(nT^2) = \mathcal{O}(| \{ (t,s,j): (t,s,j) \notin \mathcal{A} \}|)$ space for the production-inventory problem. Although this memory usage is quadratic in the number of periods $T$, we note that the memory usage can be further reduced from $\mathcal{O}(Z + nT^2)$  to $\mathcal{O}(Z)$ by reusing the allocated memory for representing each of the quantities $\sum_{i=0}^m   a_{i,t,j} \left( \frac{\lambda_i}{\sum_{i' \in \{0,\ldots,m\}: \pi^{\mathcal{A},s}(i') = \pi^{\mathcal{A},s}(i)} \lambda_{i'}} \right) \zeta^{\mathcal{A},s}_{\pi^{\mathcal{A},s}(i),s}$  for each $(t,s,j) \notin \mathcal{A}$.  Hence, we conclude that the computer memory required for performing Step 2 and Step 3 in our algorithm can scale linearly in the computer memory required to encode the constraints of the underlying dynamic robust optimization problem.

% \section{Proofs from \S\ref{sec:algorithm:improvements}}

\section{Additional Numerical Results} \label{appx:numerical}

\subsection{Practical Implications of Theorem~\ref{thm:main}} \label{appx:numerical:implications}

In this appendix, we perform numerical experiments to investigate the practical implications of Theorem~\ref{thm:main} in predicting the sparsity of optimal linear decision rules. Our numerical experiments in this appendix focus on the same problem setup  of production-inventory problems that is described at the beginning of  \S\ref{sec:experiment}. That is, our numerical experiments in this subsection focus on instances of \eqref{prob:ldr_1} in which the customer demand and production constraints follow the same setup as given in lines~\eqref{prob:experiment_setup:1}-\eqref{prob:experiment_setup:3}.  Similarly as in \S\ref{sec:experiment}, we are   particularly interested in settings where  $T$ grows large, in which case the robust optimization problem serves as an approximation of a continuous-review ordering system.   For each value of $T$ and $E$, we compute the optimal linear decision rules by solving the linear optimization formulation from \cite[Equation (39)]{ben2004adjustable} using primal simplex method.  

In  Figure~\ref{fig:plot1}, we present the results of numerical experiments with $E=3$ factories and varying numbers of periods $T$. The results of our numerical experiments provide two key takeaways. 

 Our first takeaway from Figure~\ref{fig:plot1} is that the level of sparsity of the optimal linear decision rules obtained in the numerical experiments is very significant when the number of periods is large. Indeed, the left plot in Figure~\ref{fig:plot1} shows that the number of nonzero parameters in the optimal linear decision rules  decreases to $3\%$ of the total number of parameters when inventory levels and production decisions are made twice per week over a selling horizon of one year. From a practical perspective, such a low density level of nonzeros in the optimal linear decision rules provides a clear motivation  for the growing stream of research that harnesses sparsity to develop faster and more memory-efficient algorithms in various applications~\cite{bandi2019sustainable,lorca2016multistage,ardestani2018value,delage2021column}. 

Our second takeaway from Figure~\ref{fig:plot1} is that Theorem~\ref{thm:main} is found to be predictive of the level of sparsity in the  optimal linear decision rules. In the right plot of Figure~\ref{fig:plot1}, we  present  the number of nonzero parameters in the optimal linear decision rules obtained using primal simplex method, the upper bound from Theorem~\ref{thm:main}, and the number of parameters in static decision rules  (which serves as a natural lower bound). Indeed, the multiplicative gap between the upper bound from Theorem~\ref{thm:main} and the lower bound from the static rule is very narrow (it is roughly a factor of $11$). As we can see, the numbers of nonzero parameters in the optimal linear decision rules stay in this a narrow band and grow linearly in the number of stages. This result not only validates our theory but also showcases that the number of nonzero parameters in optimal linear decision rules does not grow sublinearly in $T$.

\begin{figure}[t]
\FIGURE{\centering \includegraphics[width=0.47\linewidth]{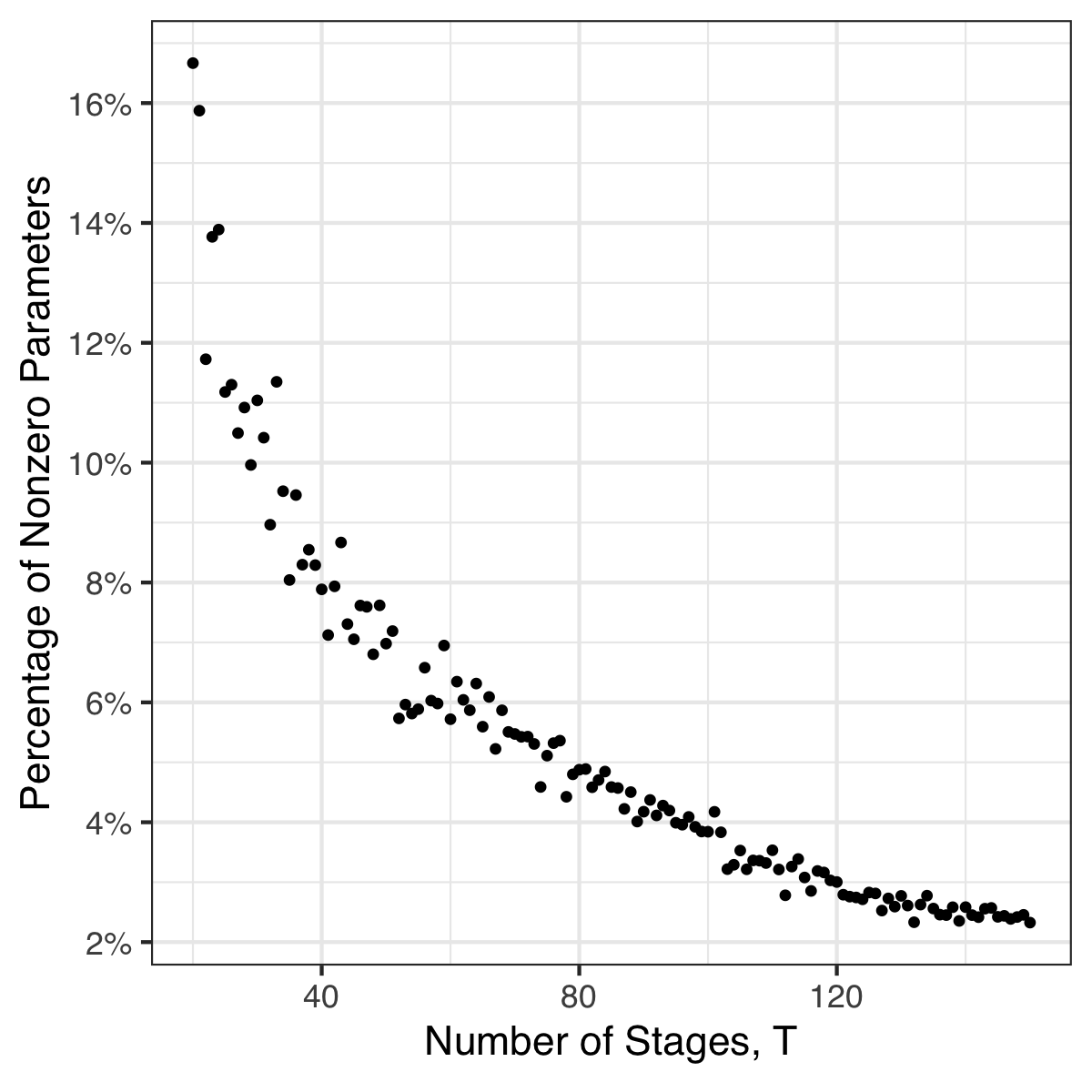} \includegraphics[width=0.47\linewidth]{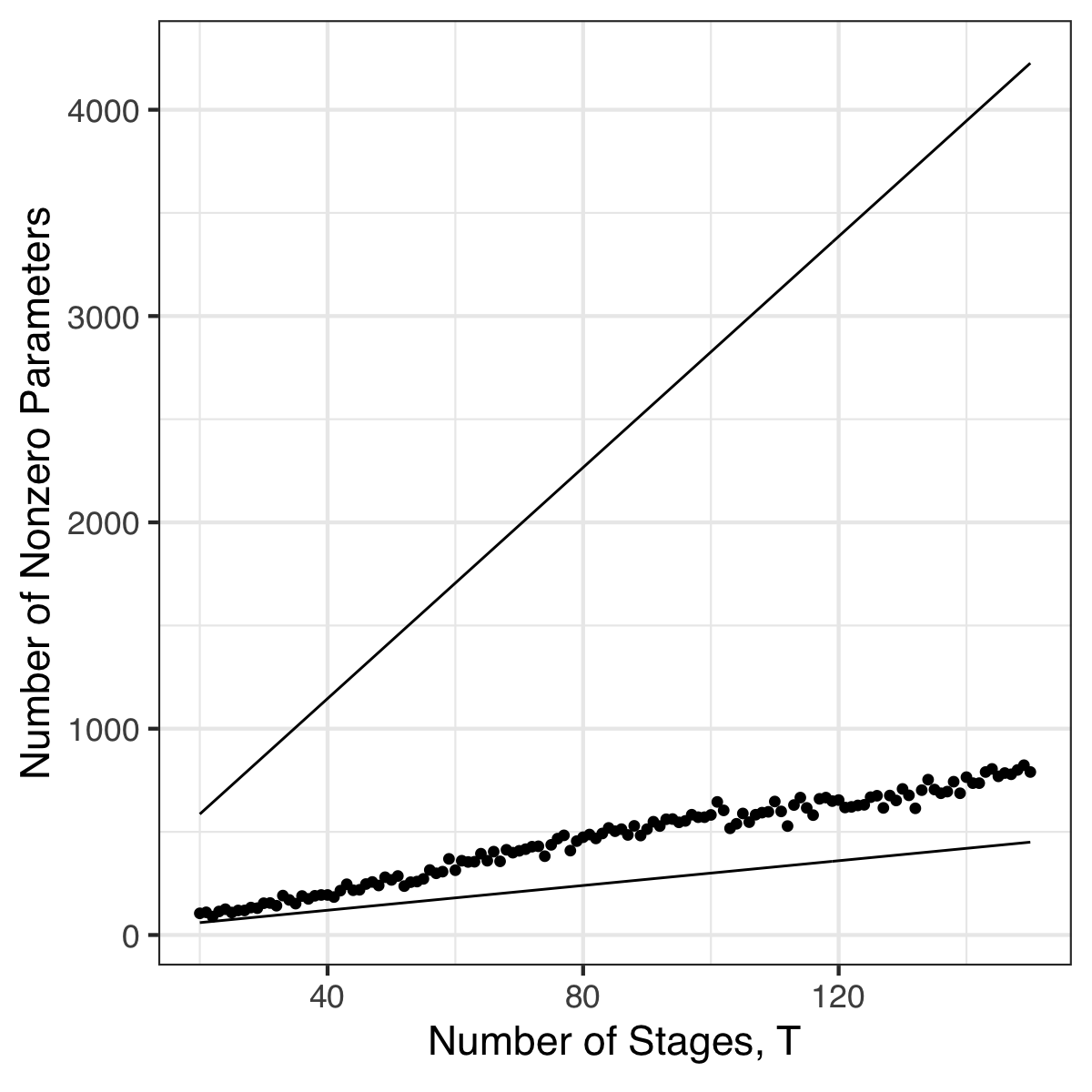} }
{Sparsity of optimal linear decision rules for production-inventory problem, $E = 3$. \label{fig:plot1}}
   {Each point represents the optimal linear decision rules computed for the corresponding number of stages $T$ and for $E=3$ factories. Left figure shows the percentage of parameters of optimal linear decision rules which are nonzero. Right figure shows the number of nonzero parameters in optimal linear decision rules compared to the upper bound from Theorem~\ref{thm:main} (top solid black line) and the number of parameters in static decision rules (bottom solid black line).   }
\end{figure}

In Figures~\ref{fig:plot2} and \ref{fig:plot3}, we present  numerical results which are similar to those from Figure~\ref{fig:plot1} for cases where the number of factories is $E=4$ and $E=5$. The results show that the findings from Figure~\ref{fig:plot1} are not exclusive to the case with $E = 3$ factories.

\begin{figure}[h]
\FIGURE{\centering \includegraphics[width=0.47\linewidth]{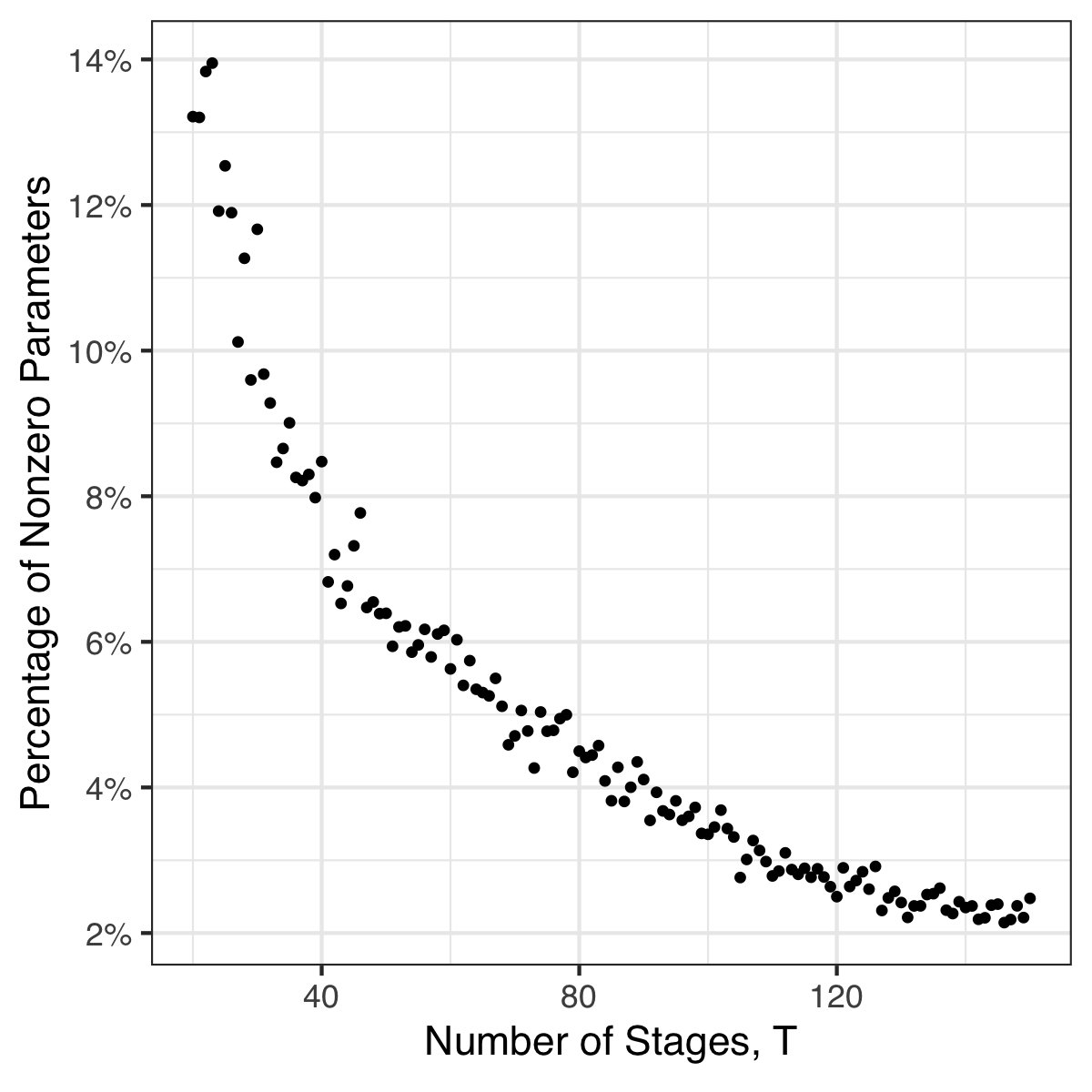} \includegraphics[width=0.47\linewidth]{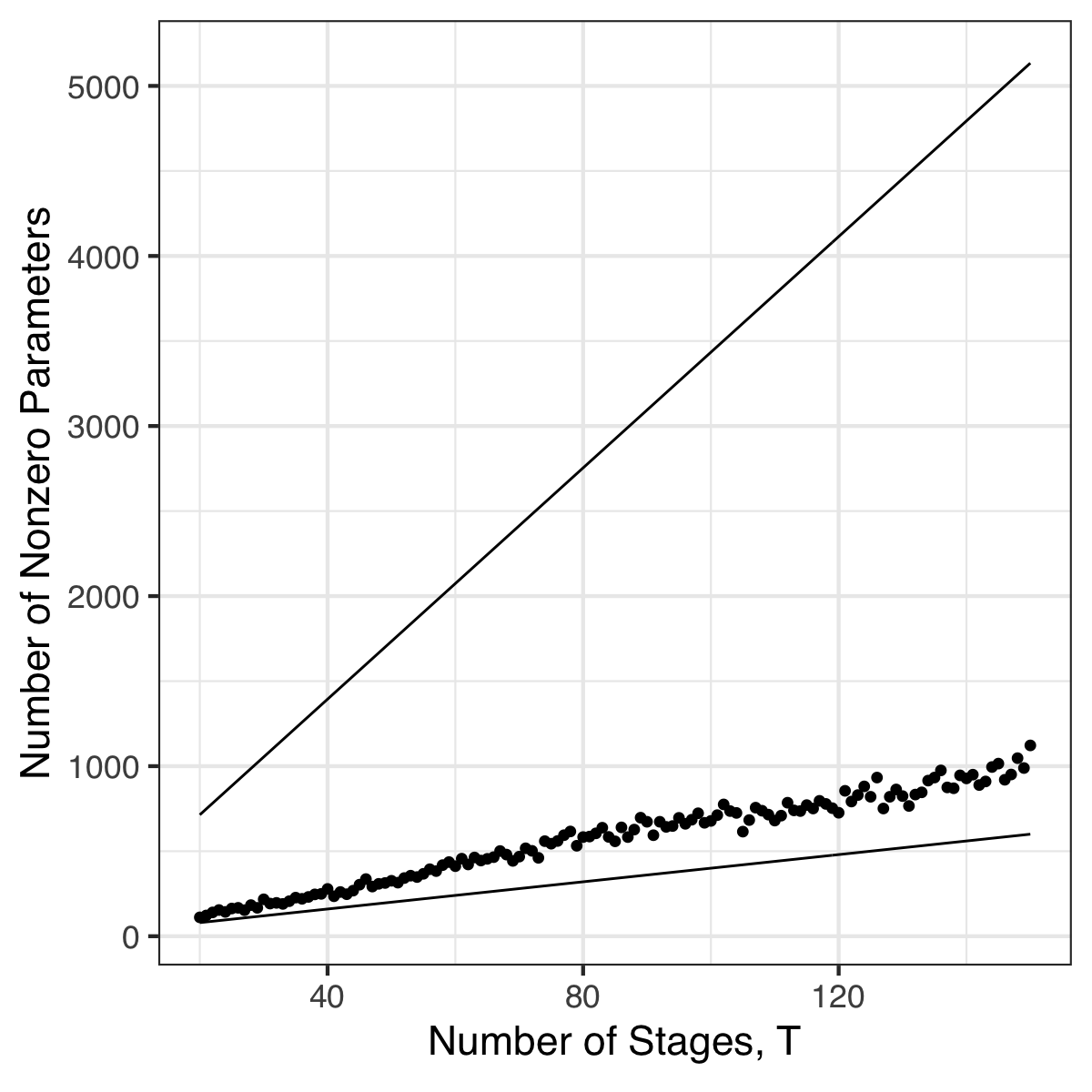} }
{Sparsity of optimal linear decision rules for production-inventory problem, $E = 4$. \label{fig:plot2}}
   {Each point represents the optimal linear decision rules computed for the corresponding number of stages $T$ and for $E=4$ factories. Left figure shows the percentage of parameters of optimal linear decision rules which are nonzero. Right figure shows the number of nonzero parameters in optimal linear decision rules compared to the upper bound from Theorem~\ref{thm:main} (top solid black line) and the number of parameters in static decision rules (bottom solid black line).   }
\end{figure}
\begin{figure}[h]
\FIGURE{\centering \includegraphics[width=0.47\linewidth]{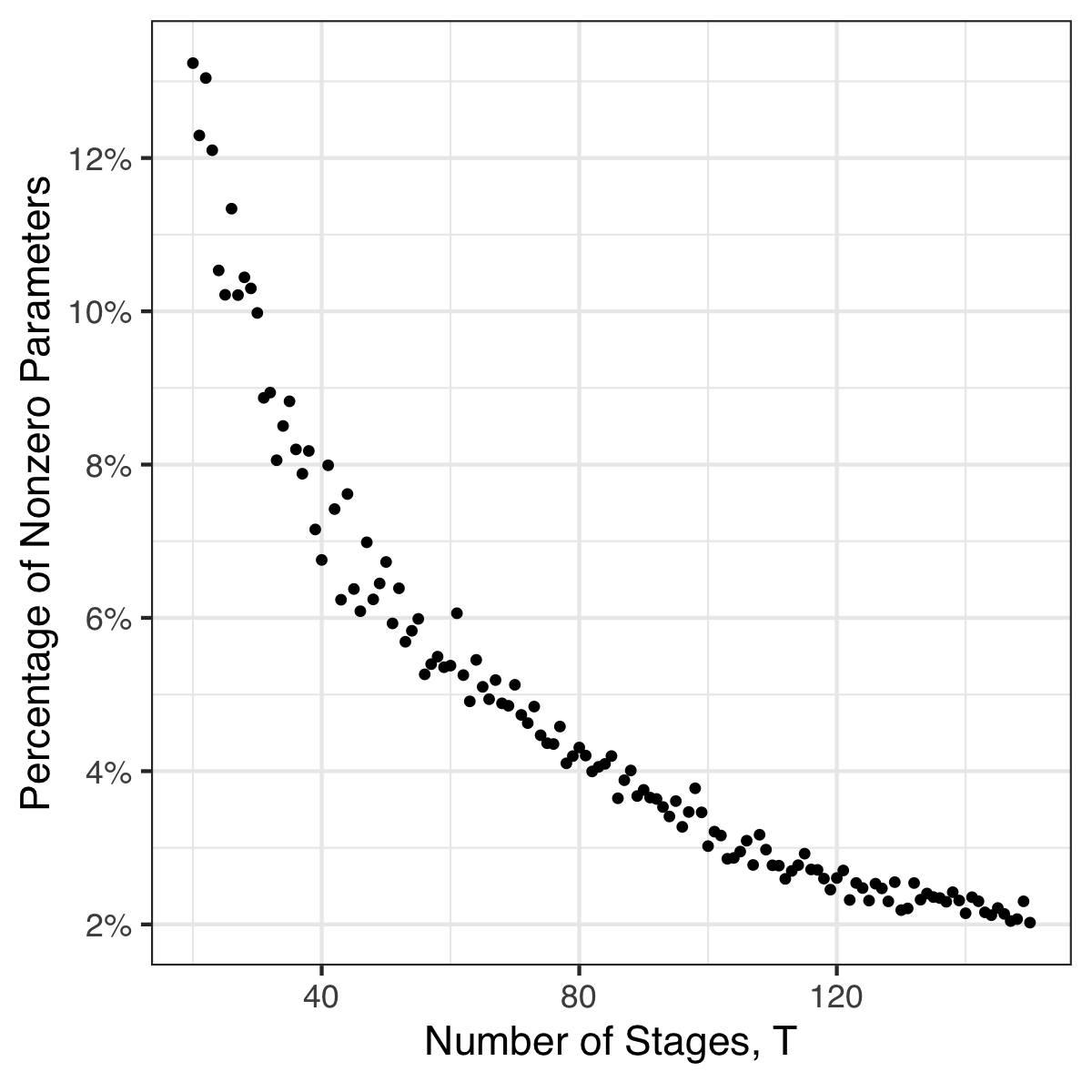} \includegraphics[width=0.47\linewidth]{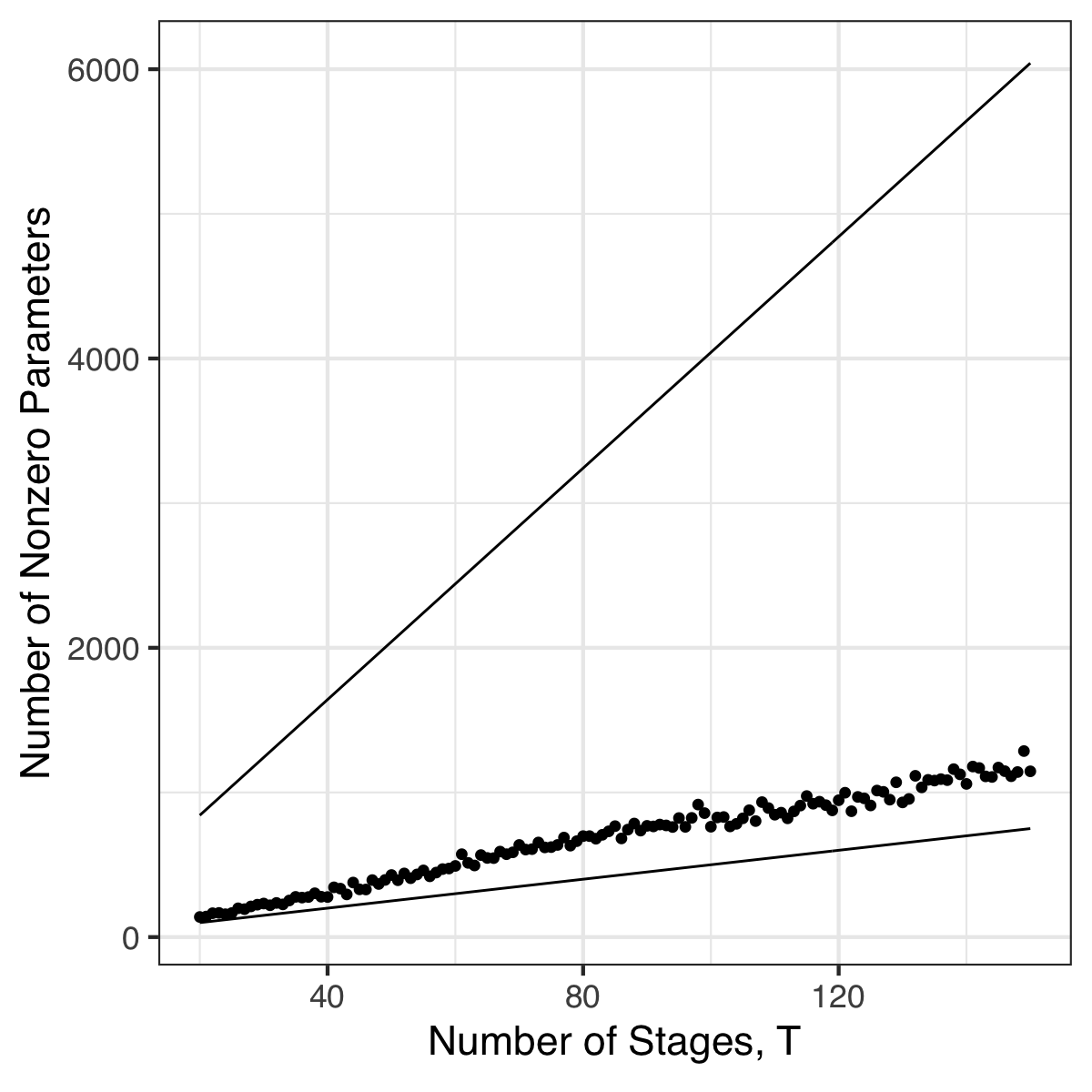} }
{Sparsity of optimal linear decision rules for production-inventory problem, $E = 5$. \label{fig:plot3}}
   {Each point represents the optimal linear decision rules computed for the corresponding number of stages $T$ and for $E=5$ factories. Left figure shows the percentage of parameters of optimal linear decision rules which are nonzero. Right figure shows the number of nonzero parameters in optimal linear decision rules compared to the upper bound from Theorem~\ref{thm:main} (top solid black line) and the number of parameters in static decision rules (bottom solid black line).   }
\end{figure}

\subsection{Comparison of LP Solvers} \label{appx:numerical:lpcomparison}
In Figure~\ref{fig:gurobi}, we present numerical results which are similar to those from Figure~\ref{fig:active_set_fixed_E} for solving the robust counterpart with Gurobi using primal simplex, dual simplex, and barrier methods,  and the recent first-order based PDLP solver~\cite{applegate2021practical} implemented in HiGHS. In these additional numerical experiments, the methods are run to optimality (but terminated if the computation time exceeds 1000 seconds). In all of the experiments in Figure~\ref{fig:gurobi} for which the computation time is equal to 1000 seconds, the solver returned neither a primal feasible solution nor a dual feasible solution. These additional experiments thus demonstrate that the slow computation time of the robust counterpart in Figure~\ref{fig:active_set_fixed_E} is not a consequence of using barrier method to solve the linear optimization problem obtained from the robust counterpart technique, but rather  a consequence of the shear size of the linear optimization problem that must be solved.  

In Figure~\ref{fig:gurobi}, we observe that PDLP~\cite{applegate2021practical} exhibits worse numerical performance compared to the Gurobi barrier method. Notably, PDLP requires a significant number of iterations—sometimes exceeding one million—to converge for these instances. We believe this is fundamentally due to the highly degenerate structure of our dual formulation, which is known to adversely impact the convergence of PDLP~\cite{lu2024geometry,lu2024first}. While GPU implementations of PDLP~\cite{lu2023cupdlp,lu2023cupdlpc} could potentially enhance numerical performance, they still require a similar number of iterations. Moreover, GPU memory is generally more expensive than CPU memory, making this approach less practical in certain scenarios. In this work, we do not compare the performance of GPU-based PDLP, as our active-set algorithm is implemented on a CPU. Exploring how GPUs might further accelerate the solution of production inventory problems via the active set algorithm is an avenue we leave for future research.

\begin{figure}[t]
\FIGURE{\centering \includegraphics[width=1\linewidth]{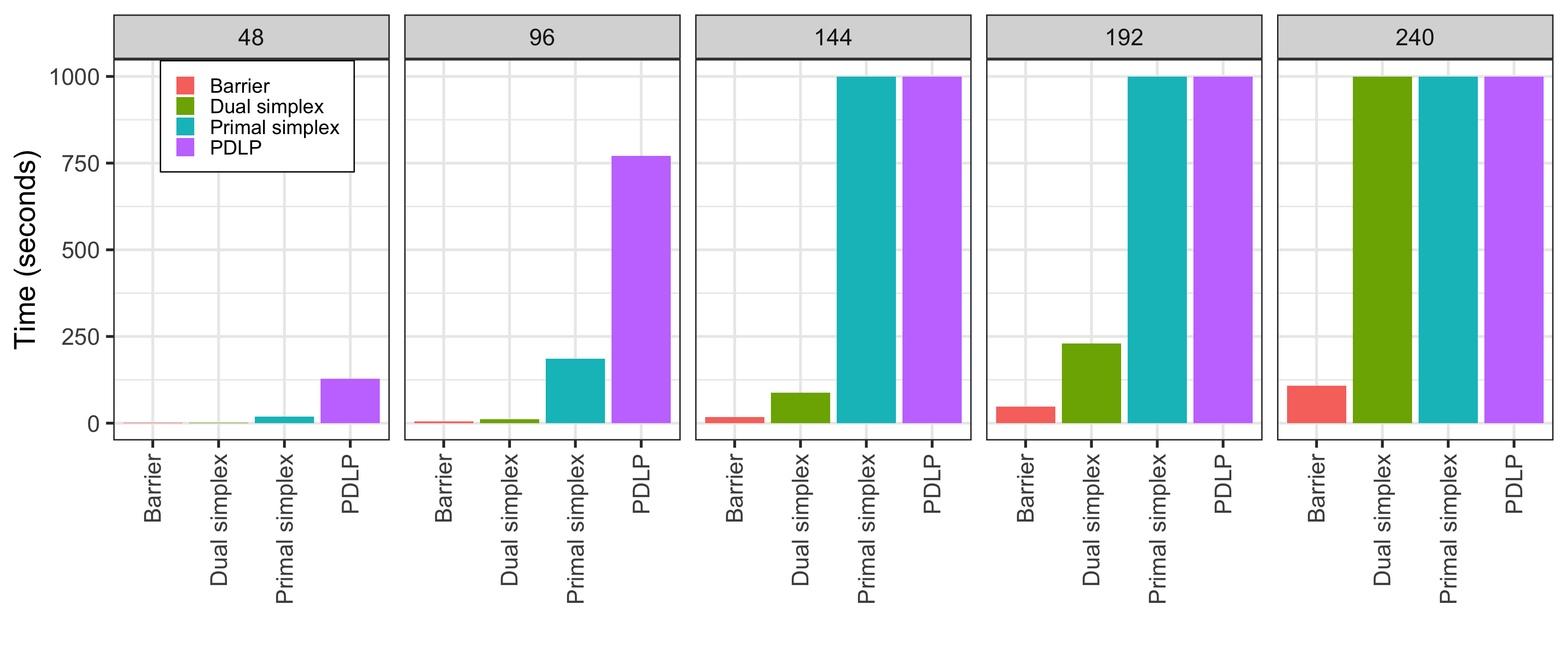} }
{Computation times for robust counterpart with $E = 5$ factories. \label{fig:gurobi}}
   {Results shown for experiments with $E=5$ factories and $T \in \{48,96,144,192,240\}$ time periods. Bars show the computation time (in seconds) for solving the robust counterpart using primal simplex, dual simplex, and barrier method. All methods are either solved to optimality or terminated at 1000 seconds. }
\end{figure}

 \end{APPENDICES}

\end{document}